\documentclass[11pt,reqno,a4paper]{amsart}
\usepackage[T1]{fontenc}
\usepackage{lmodern,amsmath,amssymb,amsthm,mathtools,mathrsfs,esint}
\usepackage{enumitem,booktabs,array,needspace,etoolbox}
\usepackage[margin=25mm]{geometry}
\usepackage[expansion=false]{microtype}
\usepackage[colorlinks=true,linkcolor=blue,citecolor=blue,urlcolor=blue]{hyperref}
\hypersetup{pdftitle={Non-simple blow-up for the Chern--Simons--Higgs equation on Riemannian surfaces},
pdfauthor={Youngae Lee and Lei Zhang},
pdfsubject={A priori structure, mass gaps, and existence of collapsing clusters}}
\allowdisplaybreaks[1]
\AtBeginDocument{%
  \setlength{\abovedisplayskip}{10pt plus 3pt minus 2pt}%
  \setlength{\belowdisplayskip}{10pt plus 3pt minus 2pt}%
  \setlength{\abovedisplayshortskip}{5pt plus 2pt minus 1pt}%
  \setlength{\belowdisplayshortskip}{2pt plus 2pt minus 1pt}%
  \setlength{\jot}{5pt}%
}
\numberwithin{equation}{section}
\newtheorem{theorem}{Theorem}[section]
\newtheorem{proposition}[theorem]{Proposition}
\newtheorem{lemma}[theorem]{Lemma}
\newtheorem{corollary}[theorem]{Corollary}
\theoremstyle{definition}
\newtheorem{definition}[theorem]{Definition}
\theoremstyle{remark}
\newtheorem{remark}[theorem]{Remark}
\newtheorem*{remark*}{Remark}
\BeforeBeginEnvironment{theorem}{\Needspace{12\baselineskip}}
\BeforeBeginEnvironment{lemma}{\Needspace{5\baselineskip}}
\BeforeBeginEnvironment{proposition}{\Needspace{5\baselineskip}}
\BeforeBeginEnvironment{corollary}{\Needspace{5\baselineskip}}
\newcommand{\R}{\mathbb R}

\newcommand{\cS}{\mathcal S}
\newcommand{\eps}{\varepsilon}
\newcommand{\etae}{\eta_{\eps}}
\newcommand{\de}{\delta_{\eps}}
\newcommand{\we}{w_{\eps}}

\newcommand{\loc}{\mathrm{loc}}

\newcommand{\tot}{\mathrm{tot}}

\newcommand{\C}{\mathbb C}

\DeclareMathOperator{\osc}{osc}

\newcommand{\angT}{\mathbb R/(2\pi\mathbb Z)}

\DeclareMathOperator{\Span}{span}
\title[Non-simple Chern--Simons blow-up on surfaces]{Non-simple blow-up for the Chern--Simons--Higgs equation on Riemannian surfaces}
\author{Youngae Lee}
\address[Youngae Lee]{Department of Mathematical Sciences,
Ulsan National Institute of Science and Technology (UNIST), Ulsan, Republic of Korea}
\email{youngaelee@unist.ac.kr}
\author{Lei Zhang}
\address[Lei Zhang]{Department of Mathematics, University of Florida,
Gainesville, FL 32611, USA}
\email{leizhang@ufl.edu} 
\date{}
\subjclass[2020]{35B44, 35J61, 35J75}
\keywords{Chern--Simons--Higgs equation, non-simple blow-up, a priori analysis, geometric selection,
CS type profile, mean-field limit, collapsing bubbles, Pohozaev identity}
\begin{document}
 
\begin{abstract}
The existence and structure of non-simple blow-up are fundamental
questions for the self-dual Chern--Simons--Higgs equation. Physically,
this phenomenon describes
magnetic-flux concentration in several cores whose mutual distances
are much larger than their widths, although all cores approach the
same vortex point. Despite substantial progress for singular Liouville
equations, the existence of non-simple blow-up solutions to the full
Chern--Simons--Higgs equation on closed surfaces has remained open in both
the Chern--Simons (CS) and mean-field regimes. We prove that the cores of every non-simple cluster have
equal limiting masses, share a common profile type, and form an
asymptotically regular polygon. These identities yield an explicit finite
list of admissible cluster masses and a four-unit mass-gap criterion.
We also construct non-simple blow-up solutions on closed Riemannian
surfaces in both regimes. Together, these results establish a sharp
arithmetic criterion for single-point non-simple CS blow-up on each fixed flat torus
when a suitable exterior divisor may be chosen, and show that vortex
blow-up need not be simple, answering Li--Liu's question in the negative.
\end{abstract}
\maketitle
\section{Introduction and main results}\label{sec:introduction}
On a smooth closed connected Riemannian surface $(M,g)$, we consider
the scalar self-dual Chern--Simons--Higgs equation
\begin{equation}\label{eq:main}
 \Delta_g u_\eps+\eps^{-2}F(u_\eps)
 =4\pi\sum_{i=1}^{N_{\mathrm v}}N_i\delta_{p_i},
\end{equation}
where $\eps>0$, the vortex points $p_i$ are fixed and distinct, and
the strengths $N_i>0$ are fixed. We write $N_{\tot}=\sum_iN_i$ and set
\begin{equation}\label{eq:F-definition}
 F(t)=e^t(1-e^t),\qquad
 \mathcal P(t)=e^t-\tfrac12e^{2t},\qquad
 \mathcal P'=F.
\end{equation}
Integrating \eqref{eq:main} over $M$ gives
$\eps^{-2}\int_M F(u_\eps)\,dV_g=4\pi N_{\tot}$.

The central problem of this paper is to determine whether concentration
at a prescribed singular point can develop several distinct regular
cores, and to characterize and construct such configurations.
In this non-simple regime, the core widths are much smaller than the
mutual distances, although all cores converge to the same point.
The resulting loss of compactness involves two separated spatial scales
and cannot be described by a single entire profile.

Our main contributions are a classification of the admissible cluster
masses and existence constructions for the full equation in both the
Chern--Simons and mean-field regimes. We prove equality of the limiting
core masses and a common profile type, and identify the asymptotically
polygonal structure. The resulting mass identities give a finite list
of possible cluster masses and a four-unit mass-gap criterion for
simplicity. Conversely, we construct non-simple families under explicit
geometric conditions, answering Li--Liu's question in the negative.
In particular, suitable fixed vortex data yield two-core CS concentration
on every closed surface, and on each fixed flat torus we obtain a sharp
arithmetic criterion for single-point CS concentration when the exterior
divisor may be chosen. The mean-field construction realizes a two-core
cluster at a strength-one vortex on a torus with a suitable smooth metric.

Previous analytical studies address existence, multiplicity, entire
profiles, and concentration; see, for example,
\cite{CaffarelliYang,CFL,ChoeKim,CKL,LinYanCMP,LinYanARMA,LinYanAIHP,LinYanAdvMath,Tarantello}.
Non-topological condensates with simple mean-field concentration at
vortex points were constructed in~\cite{DelPinoEspositoFigueroaMusso}.
The difficulty in the non-simple regime is that the translation, scale,
and mass parameters of the cores are coupled. Our analysis combines
precise core estimates and moment expansions with local balance
identities and a variational Lyapunov--Schmidt reduction. Curvature
and the regular part of the Green function enter the reduced equations,
linking the local concentration mechanism to the global geometry.

We begin by making precise the notion of non-simple blow-up.
Let
$\mathcal Z=\{p_1,\ldots,p_{N_{\mathrm v}}\}$ be the vortex set and
$|M|_g$ the area of $M$. Normalize the Green function by
\begin{equation}\label{eq:green}
 \left\{
 \begin{aligned}
  -\Delta_{g,x}G_g(x,y)&=\delta_y-|M|_g^{-1},\\
  \int_M G_g(x,y)\,dV_g(x)&=0.
 \end{aligned}
 \right.
\end{equation}
For solutions $\{u_\eps\}$ of \eqref{eq:main}, set
\begin{equation}\label{eq:u0-w}
 \begin{aligned}
  u_0(x)&=-4\pi\sum_{i=1}^{N_{\mathrm v}}N_iG_g(x,p_i),\qquad\text{and}\qquad
  w_\eps&=u_\eps-u_0-2\log\eps.
 \end{aligned}
\end{equation}
Write $d_g$ for geodesic distance and $B_r^g(q)$ for the geodesic disk;
in local coordinates or rescaled variables, $B_r(q)$ denotes a Euclidean
disk. Subsequences are taken whenever needed below.
Throughout the paper, blow-up families are understood to belong to the
concentrating non-topological alternative described in
Subsection~\ref{ap:setting}.

\begin{definition}\label{def:simple}
(i) A point $p\in M$ is called a blow-up point if there exist
points $x_\eps\to p$ such that $w_\eps(x_\eps)\to+\infty$
as $\eps\to0$.

(ii) Let $p$ be a blow-up point, and let $n$ denote its strength,
where $n=0$ if $p\notin\mathcal Z$ and $n=N_i$ if
$p=p_i\in\mathcal Z$. The point $p$ is called simple if,
for some sufficiently small $r>0$ with no other blow-up point in
$\overline{B_r^g(p)}$, there are maximum points
$x_\eps$ of $w_\eps$ on $\overline{B_r^g(p)}$, with $x_\eps\to p$, such that
\begin{equation}\label{eq:simple}
 \limsup_{\eps\to0}\sup_{x\in B_r^g(p)\setminus\{x_\eps\}}
 \bigl\{w_\eps(x)+2(1+n)\log d_g(x,x_\eps)\bigr\}<\infty.
\end{equation}
Otherwise, $p$ is called non-simple.
\end{definition}
\begin{remark}
The definition of simple blow-up is independent of the choice of
maximum point of $w_\eps$ on the same disk.
Set $a=2(n+1)$ and $H_\eps=w_\eps(x_\eps)$.
If \eqref{eq:simple} holds and $y_\eps$ is another maximum, then
$d_g(x_\eps,y_\eps)\le Ce^{-H_\eps/a}$. Since $w_\eps\le H_\eps$,
the triangle inequality gives
\[
 d_g(x,y_\eps)\le C\bigl(e^{-w_\eps(x)/a}+e^{-H_\eps/a}\bigr)
 \le2Ce^{-w_\eps(x)/a},
\]
which proves \eqref{eq:simple} with center $y_\eps$.
\end{remark}
Blow-up at regular points is simple in this sense; see
\cite[Lemma 3.5]{ChoeKim}, \cite{LiLiu}, and the application of
Lemma~\ref{ap:regular-core} below.
At vortex points, Li--Liu~\cite{LiLiu} establish an energy
identity and prove simplicity for certain small values of the
total vortex strength $N_{\tot}$, under additional assumptions
on the individual vortex strength $N_i$ and, in some cases,
on the number of blow-up points. They also ask whether blow-up at a vortex point must always be simple. In the non-simple alternative of their
analysis, the limiting masses $b_j\ge4$ of the regular bubbles
satisfy
\[
 \begin{aligned}
  M_{p_i}=\sum_{j=1}^{S}b_j, \qquad
  \sum_{j=1}^{S}b_j^2
 =\left(\sum_{j=1}^{S}b_j\right)^2
  -4N_i\sum_{j=1}^{S}b_j,
 \end{aligned}
\]
Here 
\[
 M_{p_i}:=\lim_{r\downarrow0}\lim_{\eps\to0}
(2\pi\eps^2)^{-1}\int_{B_r^g(p_i)}F(u_\eps)\,dV_g
\]
is the normalized local mass at $p_i$.

We first compare the core heights to obtain equal limiting masses
and a common profile type. The force balance then determines the
polygonal configuration. Finally, an estimate at a larger scale shows
that these cores account for the entire mass at the vortex. These
arguments give the following classification and simplicity criterion.

\begin{theorem} 
\label{ap:mass-gap}
Let $(u_\eps)$ be a family of blow-up solutions of \eqref{eq:main}.
For each vortex $p_i$, set
\[
 \mathcal A_i=\{k\in\mathbb N:1\le k\le N_i\},\qquad
 M_{i,k}=\frac{4N_i(k+1)}{k}.
\]
If a non-simple cluster with $S=k+1$ cores occurs at $p_i$,
then $k\in\mathcal A_i$, and  
\[
 \begin{aligned}
 b_1=\cdots=b_{k+1}&=\frac{4N_i}{k},\qquad 
 M_{p_i}=\sum_{j=1}^{k+1}b_j&=M_{i,k}.
 \end{aligned}
\]
The remaining mass necessarily satisfies
\begin{equation}\label{ap:necessary-alternative}
 2N_{\tot}-M_{i,k}\in\{0\}\cup[4,\infty).
\end{equation}
\end{theorem}

\medskip

\begin{remark}
In particular, if, for every $i$ and every $k\in\mathcal A_i$,
\begin{equation}\label{ap:four-unit-gap}
 2N_{\tot}-M_{i,k}\in(-\infty,0)\cup(0,4),
\end{equation}
then every blow-up point is simple: the theorem excludes non-simple
blow-up at vortices, and regular blow-up points are simple.
This criterion imposes no upper bound on the individual or total vortex strength.
If $p_i$ is the only blow-up point, then non-simple blow-up
at $p_i$ requires the stronger condition
\[
 2N_{\tot}=M_{i,k}\qquad\text{for some }k\in\mathcal A_i.
\]
Thus the mass classification also recovers the simplicity
conclusions for the single-point configurations considered
in~\cite{LiLiu}. For example, when $N_i=1$, the only candidate
non-simple mass is $8$, so a sole blow-up point of strength one
must be simple when $N_{\tot}=6$ or $7$.
\end{remark}

As in the simple blow-up setting of~\cite{LinYanAIHP}, we distinguish
two profile regimes.

\begin{definition}
Let $\{u_\varepsilon\}$ be a family of solutions of
\eqref{eq:main}, and let $q$ be a blow-up point.
Fix a sufficiently small radius $r>0$.
We say that $q$ is of mean-field type (MF type) if
\[
 \lim_{\varepsilon\to0}
 \sup_{x\in B_r^g(q)}u_\varepsilon(x)=-\infty,
\]
and of Chern--Simons type (CS type) if
\[
 \lim_{\varepsilon\to0}
 \sup_{x\in B_r^g(q)}u_\varepsilon(x)>-\infty.
\]
\end{definition}
\begin{remark}
The regular entire profiles arising in CS-type blow-up have mass
strictly greater than $4$; mean-field profiles have mass $4$.
In view of Theorem~\ref{ap:mass-gap}, a non-simple cluster
at $p_i$ has exactly $S=N_i+1$ cores, each with local mass $4$,
if and only if it is of mean-field type.
\end{remark}
Non-simple blow-up for the limiting singular Liouville and mean-field
equations has been studied extensively
\cite{KuoLin,BartolucciTarantello,WeiZhang,WeiZhangVanishing,
WeiZhangLaplacian,DAprileWeiZhangTAMS}.
Constructions under suitable assumptions on the potential are
given in~\cite{DAprileWeiZhangDCDS,DAprileWeiZhangCVPDE},
while~\cite{DAprileWeiZhangAJM} establishes exclusion results
under additional boundary or concentration assumptions.

For the full equation~\eqref{eq:main}, we construct non-simple clusters
in both regimes by a variational Lyapunov--Schmidt reduction. After
matching the cores, we solve the equation subject to finitely many
kernel constraints. At critical points of the reduced energy,
the remaining multipliers vanish. The leading point forces leave the
scale and angle undetermined; higher moments and geometric terms determine
them. In the mean-field regime, a finer odd correction also determines
the common center. These constructions answer Li--Liu's question
in the negative.

To state the Chern--Simons type constructions, we
distinguish the concentration vortex $p$ of strength $n$ from the
exterior vortices $p_\ell$ of strengths $n_\ell$:
\begin{equation}
 \label{mc:pde}
 \begin{aligned}
  \Delta_g u+\eps^{-2}F(u)
 &=4\pi n\delta_p+4\pi\sum_{\ell=1}^k n_\ell\delta_{p_\ell}, \\
  N&=n+\sum_\ell n_\ell.
 \end{aligned}
\end{equation}
For a cluster of $S$ cores of mass $m$, force balance and exhaustion
of the total mass require, respectively,
\begin{align}
 m(S-1)&=4n,\qquad S\ge2,\quad m>4,\label{mc:compat}\\
 Sm&=2N.\label{mc:total}
\end{align}
In isothermal coordinates at $p=0$, normalized by
\begin{equation}\label{mc:chart}
 g=e^\sigma|dz|^2,\qquad \sigma(0)=0,\quad
 \nabla\sigma(0)=0,\quad D^2\sigma(0)=-K_g(p)I,
\end{equation}
write $G=G_g$, set $R(x,y)=G(x,y)+(2\pi)^{-1}\log|x-y|$, and define
\begin{equation}
 \label{mc:H}
 \begin{aligned}
  \mathcal H(x)&=\pi m(S+1)R(x,0)
       -4\pi\sum_\ell n_\ell G(x,p_\ell), \\
  h&=\nabla\mathcal H(0), \\
  B&=D^2\mathcal H(0).
 \end{aligned}
\end{equation}
Relations~\eqref{mc:compat}--\eqref{mc:total} make $\mathcal H$
harmonic near $p$. For $S=2$, define
\begin{equation}\label{mc:pair-matrix}
 A_{\rm eff,m}=B-\frac{mK_g(p)}4I
                +\frac4m h\otimes h-\frac2m|h|^2I.
\end{equation}
Its eigenvalues are independent of the normalized chart
(Subsection~\ref{sec:geometric-data}).

For $S\ge3$ on a flat torus, the relevant condition on the fixed divisor is
\begin{equation}\label{af:jet}
 \mathcal H(z)=\mathcal H(0)+\operatorname{Re}(\gamma_Sz^S)
                   +O(|z|^{S+1}),\qquad \gamma_S\ne0,
\end{equation}
in a Euclidean coordinate.

\begin{theorem}[Existence of CS-type non-simple
blow-up]\label{mc:main}
Let the fixed integer divisor in \eqref{mc:pde} and $(S,m)$ satisfy
\eqref{mc:compat}--\eqref{mc:total}. Assume one of the following:
\begin{enumerate}[itemsep=3pt,topsep=5pt]
\item[(i)] $S\ge3$ and $K_g(p)<0$;
\item[(ii)] $S=2$ and $A_{\rm eff,m}$ has a positive eigenvalue;
\item[(iii)] $(M,g)$ is a flat torus, $S\ge3$, and \eqref{af:jet} holds.
\end{enumerate}
Then \eqref{mc:pde} admits a family of CS-type non-simple
blow-up solutions $\{u_\eps\}$ for sufficiently
small $\eps>0$, with $S$ cores at $p$ and satisfying
\[
 \eps^{-2}F(u_\eps)\,dV_g
 \rightharpoonup 4\pi N\delta_p\qquad\text{as }\eps\to0.
\]
\end{theorem}
Negative curvature also suffices for two cores, since
$\operatorname{tr}A_{\rm eff,m}=-mK_g(p)/2$. At nonnegative
curvature the background can still supply a positive eigenvalue.

\begin{corollary}[Two cores on every fixed surface]\label{cor:every-surface}
For every point $p$ on a smooth closed surface and integer $n\ge2$,
there exists $a\ne p$ such that the fixed divisor $n[p]+3n[a]$ admits the
two-core family of Theorem~\ref{mc:main}(ii).
\end{corollary}
The case $n=2$ attains the minimum CS-type total strength eight: if
$k=S-1<N_i$, then
\[
 2N_{\tot}\ge\frac{4N_i(k+1)}k
 \ge\frac{4(k+1)^2}k\ge16.
\]
The mean-field construction below instead has total strength four.

\begin{corollary}[Realization on every flat torus]\label{mc:flat-realization}
On every flat torus, for every prescribed point $p$ and pair $(S,m)$
satisfying \eqref{mc:compat}--\eqref{mc:total}, there is a fixed integer
divisor preserved only by the identity isometry for which
\eqref{mc:pde} admits a family of non-simple CS-type blow-up solutions
with $S$ cores at $p$.
\end{corollary}
In particular, fix a flat torus, a point $p$, and positive integers
$n,N$. There exists a fixed exterior integer divisor of total
strength $N-n$ admitting single-point CS-type non-simple
concentration at $p$ if and only if, for some integer $S\ge2$,
\[
 m=\frac{2N}{S}>4,\qquad m(S-1)=4n.
\]
Equivalently, $r=N-2n\ge3$ and $N/r$ is an integer at least two.
This is a sharp arithmetic existence criterion with the exterior
divisor allowed to vary. For example, $m=8$, $n=2(S-1)$ and
$N=4S$ are realized for every $S\ge2$ on every flat torus.

Finally, we construct mean-field pairs at a strength-one vortex with
total vortex strength four.
\begin{theorem}[A mean-field two-core family on the square torus]
\label{mfg:cor:asymmetric-square-torus}
On $E=\C/(\mathbb Z+i\mathbb Z)$ with strength-one vortices at
$0,\frac12,\frac i2,\frac{1+i}{2}$, there exists a fixed smooth metric
$g$ such that, for every sufficiently
small $\eps>0$, \eqref{eq:main} admits a mean-field type non-simple
blow-up solution $u_\eps$ with two cores at zero. The resulting family
satisfies
\[
 \eps^{-2}F(u_\eps)\,dV_g\rightharpoonup16\pi\delta_0
 \qquad\text{as }\eps\to0.
\]
\end{theorem}
The metric is given explicitly in \eqref{mfg:explicit-asymmetric-metric},
and Subsection~\ref{mfg:explicit-construction} proves that it has no
nonidentity isometry.
The general geometric criterion and its proof appear in
Theorem~\ref{mfg:thm:general-quartic-pair} and
Section~\ref{sec:mf-global-direct}.

The preceding theorems also have a direct interpretation in self-dual
gauge theory. For integer vortex strengths, \eqref{eq:main} is the
scalar Chern--Simons--Higgs equation, with $u_\eps=\log|\phi_\eps|^2$
\cite{HongKimPac,JackiwWeinberg,SpruckYang}.
The density $\eps^{-2}F(u_\eps)$ is proportional to the magnetic field,
and Gauss law relates the electric-charge density to that field.
Thus our constructions exhibit charged flux clusters whose distinct
cores are invisible in the single Dirac mass of the weak limit;
the classification prescribes their equal limiting flux shares and
admissible core numbers. On flat tori these solutions describe periodic
condensates, while the curved-surface results identify geometric
conditions for such concentration
\cite{DelPinoEspositoFigueroaMusso,KimKimRiemannSurfaces}.
Spatial structure can carry information beyond total energy, as
illustrated by the angular-momentum and slow-motion analysis of planar
self-dual vortices in~\cite{KimLeeVortexDynamics}.
Here the conclusions concern static solution families; their dynamical
properties remain a subject for further study.

Section~\ref{sec:setting} contains the preliminaries, and
Section~\ref{ap:structure} proves the cluster structure and mass gap.
The CS reduction and constructions are presented in
Sections~\ref{mc:analytic-section}--\ref{sec:cs-realizations};
Section~\ref{sec:mf-global-direct} treats the mean-field regime.
\section{Geometric preliminaries and planar profiles}\label{sec:setting}
\subsection{Concentration and local coordinates}\label{ap:setting}
We recall the concentrating blow-up alternative of
\cite[Theorem 3.1(iii)]{ChoeKim}; see also \cite{LinYanAIHP}.
For a fixed surface and fixed vortex data, the concentration analysis
in local conformal charts gives
\begin{equation}\label{ap:strict}
 \limsup_{\eps\to0}\left(\sup_M u_\eps\right)<0.
\end{equation}
The blow-up set $\cS_u$ is finite and nonempty, and after passing
to a subsequence,
\begin{gather}
 \mu_\eps:=\frac{F(u_\eps)}{2\pi\eps^2}\,dV_g
 \rightharpoonup\sum_{q\in\cS_u}M_q\delta_q,\qquad M_q\ge4,
 \label{ap:concentration}\\
 w_\eps\longrightarrow-\infty
 \quad\text{locally uniformly on }M\setminus\cS_u.
 \label{ap:exterior-decay}
\end{gather}
The integrated equation~\eqref{eq:main} and \eqref{ap:concentration} give
\begin{equation}\label{ap:total-mass}
 \mu_\eps(M)=\sum_{q\in\cS_u}M_q=2N_{\tot}.
\end{equation}
With $\overline u_\eps=|M|_g^{-1}\int_Mu_\eps\,dV_g$, the
normalization~\eqref{eq:green} gives the representation
\begin{equation}\label{ap:green-representation}
 u_\eps(x)-u_0(x)
 =\overline u_\eps
 +2\pi\int_M G_g(x,y)\,d\mu_\eps(y).
\end{equation}
For the local analysis at a vortex $p_i$, choose a conformal normal
coordinate $x=0$ with
\begin{equation}
 \label{ap:chart}
 \begin{aligned}
  g&=e^{\phi_i(x)}|dx|^2, \\
  \phi_i(0)&=|\nabla \phi_i(0)|=0,\quad 
D^2\phi_i(0)=-K_g(p_i)I.
 \end{aligned}
\end{equation}
The harmonic quadratic part is removed by a holomorphic coordinate
change. Set $\gamma_i(x)=G_g(x,p_i)+(2\pi)^{-1}\log|x|$ and
\begin{equation}\label{ap:local-coefficient}
 K_i
 =4\pi(N_{\tot}-N_i)\gamma_i
       -4\pi\sum_{j\ne i}N_jG_g(\,\cdot\,,p_j).
\end{equation}
The coefficient $K_i$ is harmonic near $p_i$, since
\[
 \Delta G_g(\,\cdot\,,p_j)=\Delta\gamma_i=e^{\phi_i}/|M|_g
\]
 for
$j\ne i$. We reserve $K_g$ for Gaussian curvature.

\subsection{Compatibility and coordinate invariance}\label{sec:geometric-data}
For a CS cluster exhausting the total mass,
\eqref{mc:compat}--\eqref{mc:total} give
\begin{equation}\label{mc:arithmetic}
 r:=N-2n=\frac{2n}{S-1}\in\mathbb N,\qquad
 m=2r,\quad n=\frac{r(S-1)}2,\quad N=rS.
\end{equation}
Thus $r\ge3$ and $r(S-1)$ is even. For integer divisors, the possible core masses are
$6,8,10,\ldots$ for odd $S$, and $8,12,16,\ldots$ for even $S$.
In particular $m>6$ implies $m\ge8$, and $S=2$ requires $m=4n\ge8$.
For \eqref{mc:pde},
\begin{equation}\label{mc:u0}
 u_0=-4\pi nG(\cdot,p)-4\pi\sum_\ell n_\ell G(\cdot,p_\ell).
\end{equation}
An orientation-preserving transition $z=f(\zeta)$ between charts
\eqref{mc:chart} obeys
$\widetilde\sigma=\sigma\circ f+2\log|f'|$. The normalizations in \eqref{mc:chart} imply
$|f'(0)|=1$, $f''(0)=f'''(0)=0$, so
$f(\zeta)=e^{i\alpha}\zeta+O(|\zeta|^4)$. The law
\[
 \widetilde R(\zeta,0)=R(f(\zeta),0)
+(2\pi)^{-1}\log|\zeta/f(\zeta)|
\]
 gives
\[
 \widetilde{\mathcal H}=\mathcal H\circ f+O(|\zeta|^3).
\]
For $Q=Df(0)$, therefore,
$\widetilde h=Q^Th$, $\widetilde B=Q^TBQ$, and
$\widetilde A_{\rm eff,m}=Q^TA_{\rm eff,m}Q$.
Complex conjugation covers orientation reversal. Hence the
eigenvalues of \eqref{mc:pair-matrix} are intrinsic.

\subsection{CS profiles and linearized modes}\label{sec:linearized}
The translational modes span the bounded kernel in the CS reduction, while
the second angular mode determines the first correction to the point forces.
Let $U_\alpha$ be the radial solution of
\begin{equation}\label{eq:entire}
 \begin{gathered}
 \Delta U_\alpha+F(U_\alpha)=0\quad\text{in }\R^2,\\
 U_\alpha(0)=\alpha<0,\quad U_\alpha'(0)=0,\quad
 U_\alpha(\infty)=-\infty.
 \end{gathered}
\end{equation}
Define the associated mass map by
\begin{equation}\label{eq:mass-map}
 m(\alpha)=\frac1{2\pi}\int_{\R^2}F(U_\alpha)\,dz.
\end{equation}
The family depends smoothly on $\alpha$, and
\begin{equation}\label{eq:profile-facts}
 m:(-\infty,0)\longrightarrow(4,\infty)
 \text{ is onto},\qquad m'(\alpha)>0.
\end{equation}
See \cite{CFL} and \cite[pp.~840--841 and Lemma~2.4]{CKL}
for details.

Throughout, a tilde denotes mass integrated over a finite disk.
Unless uniformity is stated, $o_R(1)+o_\eps(1)$ means first
$\eps\to0$ at fixed $R$, then $R\to\infty$.

Set $L_\alpha=\Delta+F'(U_\alpha)$. Writing \eqref{eq:entire} as
$U_\alpha''+r^{-1}U_\alpha'+F(U_\alpha)=0$, we obtain the following
monotonicity and profile asymptotics:
\begin{align}
 (rU_\alpha')'&=-rF(U_\alpha),\qquad U_\alpha'(r)<0\quad(r>0),
 \label{eq:radial-monotonicity}\\
 U_\alpha(r)&=-m\log r+c_\infty+O(r^{2-m}),\nonumber\\
 U_\alpha'&=-m/r+O(r^{1-m}),\qquad
 U_\alpha''=m/r^2+O(r^{-m}),\label{eq:Usecond-asymptotic}\\
 D^2U_\alpha(0)&=-\tfrac12\kappa_\alpha I_2,\quad
 U_\alpha'(r)=-\tfrac12\kappa_\alpha r+O(r^3),\quad
 \kappa_\alpha=F(\alpha)>0.\label{eq:kappa-alpha}
\end{align}
Here $m=m(\alpha)$ and $c_\infty=c_\infty(\alpha)$; see
\cite[Lemma 2.6]{CFL}.

\begin{lemma}[Sublinear kernel and normalization]\label{thm:linearized-classification}
If $L_\alpha\phi=0$ in $\R^2$ and
$|\phi(x)|\le C(1+|x|)^\tau$ for some $0<\tau<1$, then
\begin{equation}\label{eq:kernel-classification}
 \phi=c_0\partial_\alpha U_\alpha+
 c_1\partial_{x_1}U_\alpha+c_2\partial_{x_2}U_\alpha.
\end{equation}
Thus $\phi(0)=0$ and $\nabla\phi(0)=0$ imply $\phi=0$.
\end{lemma}
\begin{proof}
The Fourier coefficient $h(r)$ of angular order $k$ in $\phi$ satisfies
\begin{equation}\label{eq:mode-equation}
 h''+r^{-1}h'-k^2r^{-2}h+F'(U_\alpha)h=0.
\end{equation}
For $k=0$, regular solutions are multiples of $\partial_\alpha U_\alpha$,
whose value at zero is one. For $k=1$, substituting $h=wU_\alpha'$ into \eqref{eq:mode-equation} gives
$(r(U_\alpha')^2w')'=0$; regularity forces $w$ to be constant.
For $k\ge2$, set $\psi=-U_\alpha'>0$ and choose the sign of a nonzero
regular solution so that $h(r)\sim ar^k$, $a>0$, at zero.
Then
\begin{equation}\label{eq:higher-mode-comparison}
 (r\psi^2(h/\psi)')'=(k^2-1)r^{-1}\psi h.
\end{equation}
While $h>0$, \eqref{eq:higher-mode-comparison} implies that $h/\psi$
is positive and increasing, so $h$ cannot have a first zero.
For large $r$, $r\psi^2(h/\psi)'\ge c>0$; since
$\psi\sim m/r$, this gives $(h/\psi)'\ge c'r$ and $h\ge c''r$,
contradicting sublinear growth.
Since $\partial_\alpha U_\alpha(0)=1$, the value condition gives
$c_0=0$; the gradient condition and \eqref{eq:kappa-alpha} then give
$c_1=c_2=0$.
\end{proof}

\begin{lemma}[The entire profile and its second mode]\label{mc:planar}
For $m>4$, mass parametrization of \eqref{eq:entire} gives a locally
smooth family $U_m$ with tail
$U_m=-m\log r+c_\infty(m)+O(r^{2-m})$.
The bounded kernel of $\Delta+F'(U_m)$ is spanned by
$\partial_1U_m,\partial_2U_m$;
$Z_{0,m}=\partial_mU_m$ has logarithmic coefficient $-1$ and
$\int_{\R^2}F'(U_m)Z_{0,m}=2\pi$.
For $m>6$, the positive regular second mode satisfies
\begin{align}
&\Phi_m''+r^{-1}\Phi_m'-4r^{-2}\Phi_m+F'(U_m)\Phi_m=0,\nonumber\\
 &\Phi_m(r)=r^2+b(m)r^{-2}+O(r^{4-m})\ \text{as}\ r\to\infty,\label{mc:Phi}\\
 &b(m)=\tfrac14\int_0^\infty r^3F'(U_m)\Phi_m\,dr>0.
 \label{mc:b}
\end{align}
Moreover, writing $P=\mathcal P$ as in \eqref{eq:F-definition},
\begin{equation}\label{mc:pressure}
 \int_{\R^2}P(U_m)\,dz=\pi m^2/2.
\end{equation}
The estimates are uniform on compact mass intervals in $(6,\infty)$.
\end{lemma}
\begin{proof}
Since $m'(\alpha)>0$ by \eqref{eq:profile-facts}, the inverse function
theorem gives a smooth mass parametrization $U_m=U_{\alpha(m)}$.
Differentiating the asymptotic expansion at infinity gives
$Z_{0,m}=-\log r+c_\infty'(m)+O(r^{2-m}\log r)$.
Thus Lemma~\ref{thm:linearized-classification} shows that the bounded
kernel is spanned by the two translational modes. Differentiation under the
integral in $\int_{\R^2}F(U_m)=2\pi m$ gives
$\int_{\R^2}F'(U_m)Z_{0,m}=2\pi$; on compact mass intervals the
tail is dominated by $C r^{-m_*}\log r$, with $m_*>4$.

Put $\psi=-U_m'>0$. The comparison identity
\eqref{eq:higher-mode-comparison} shows that the regular second mode
is positive, with $(\Phi_m/\psi)'>0$ and a positive $r^2$ coefficient.
Normalize that coefficient to one. Variation of parameters in the
second-mode equation, using the fundamental solutions $r^2,r^{-2}$ of
$\Phi''+r^{-1}\Phi'-4r^{-2}\Phi=0$, gives \eqref{mc:Phi}. Integrating
$(r^3\Phi_m'-2r^2\Phi_m)'=-r^3F'(U_m)\Phi_m$
gives \eqref{mc:b}. To prove $b(m)>0$, use
$(F(U_m))'=-\psi F'(U_m)$:
\[
 \int_\varepsilon^R r^3F'(U_m)\Phi_m\,dr
 =-\left[F(U_m)r^3\frac{\Phi_m}{\psi}\right]_\varepsilon^R
 +\int_\varepsilon^R F(U_m)
       \left(r^3\frac{\Phi_m}{\psi}\right)'dr.
\]
The bracketed function is $O(r^4)$ at zero and $O(r^{6-m})$ at
infinity, so the boundary terms vanish as $\varepsilon\downarrow0$ and
$R\to\infty$ for $m>6$. Since $F(U_m)>0$
and $(r^3\Phi_m/\psi)'>0$, the remaining integral is positive.
For the similarly normalized mode at $m=6$, the truncated moment
$\int_0^R r^3F'(U_6)\Phi_6\,dr=A_6\log R+O(1)$,
$A_6=e^{c_\infty(6)}$, grows logarithmically; see
Subsection~\ref{cc:force}.

Finally, multiplying $U_m''+r^{-1}U_m'+F(U_m)=0$ by $r^2U_m'$
and using $P'=F$, we obtain
\[
 \left(\tfrac12r^2(U_m')^2+r^2P(U_m)\right)'=2rP(U_m).
\]
The expression in parentheses tends to zero at $r=0$ and to $m^2/2$
as $r\to\infty$. Integration gives
$\int_0^\infty rP(U_m)\,dr=m^2/4$, hence \eqref{mc:pressure}.
Smooth dependence on $m$ and the tail bounds above give uniformity on
compact mass intervals.
\end{proof}

\subsection{Linearization at a regular polygon}
Identify $\R^2$ with $\C$. For
$q=(q_0,\ldots,q_{S-1})\in\C^S$ with
$q_j\ne0$ for all $j$ and $q_j\ne q_\ell$ for $j\ne\ell$,
define
\begin{equation}\label{mc:pointforce}
 \mathcal F_j(q)
 =2n\frac{q_j}{|q_j|^2}
 -m\sum_{\ell\ne j}
       \frac{q_j-q_\ell}{|q_j-q_\ell|^2}.
\end{equation}
Set
\[
 \mathcal F=(\mathcal F_0,\ldots,\mathcal F_{S-1}),
 \qquad
 f=\overline{\mathcal F}=(f_0,\ldots,f_{S-1}).
\]
The force $\mathcal F$ is the Euclidean gradient of
\[
 W(q)=2n\sum_j\log|q_j|
      -m\sum_{j<\ell}\log|q_j-q_\ell|.
\]

Let $z=(z_0,\ldots,z_{S-1})$, where
$z_j=e^{2\pi ij/S}$, $j=0,\ldots,S-1$, and define
\begin{equation}\label{mc:E}
 E=\left\{\xi\in\C^S:
          \sum_j\xi_j\overline{z_j}=0\right\}.
\end{equation}
Thus $E$ is the real orthogonal complement of the dilation
and rotation directions $z$ and $iz$.

\begin{lemma}[Polygon nondegeneracy]\label{mc:polygon}
Assume \eqref{mc:compat}. Then the following hold.
\begin{enumerate}[label=(\roman*)]
\item
$\mathcal F(de^{i\theta}z)=0$ for all $d>0$ and $\theta\in\R$.

\item
The real symmetric Jacobian
$\mathscr L=D\mathcal F(z)$ satisfies
\[
 \ker\mathscr L=\Span_\R\{z,iz\}.
\]
Its restriction to $E$ is invertible, with $S-1$ positive
and $S-1$ negative eigenvalues.

\item
For every $H\in\C$, identified with the constant vector
$(H,\ldots,H)$,
\begin{equation}\label{mc:hshape}
 \begin{aligned}
  \mathscr L\left(
    \frac{\overline H}{2n}z_j^2
  \right)_{j=0}^{S-1}
  &=-H,\\
  \sum_{j=0}^{S-1}z_j^2&=0
  \qquad(S\ge3).
 \end{aligned}
\end{equation}
\end{enumerate}
\end{lemma}

\begin{proof}
For (i), let $P(t)=t^S-1$. Since the $z_j$ are its roots,
\[
 \sum_{\ell\ne j}\frac{1}{z_j-z_\ell}
 =\frac{P''(z_j)}{2P'(z_j)}
 =\frac{S-1}{2z_j}.
\]
Hence \eqref{mc:compat} gives
\[
 f_j(z)=\frac{2n-m(S-1)/2}{z_j}=0.
\]
The identity
$\mathcal F(de^{i\theta}z)
=d^{-1}e^{i\theta}\mathcal F(z)$
then proves (i).

For (ii), write $w_j=z_jv_j$. Differentiating $f$ and using
\eqref{mc:compat}, we obtain
\[
 \begin{aligned}
  (Df(z)w)_j&=z_j^{-1}(Lv)_j,\\
  (Lv)_j&=\frac m4\sum_{\ell\ne j}
       \frac{v_\ell-v_j}{\sin^2(\pi(j-\ell)/S)}.
 \end{aligned}
\]
The matrix $L$ is real, symmetric, and circulant. Its Fourier
modes $v_j^{(k)}=e^{2\pi ikj/S}$ have eigenvalues
\[
 \ell_k=-\frac m2 k(S-k),\qquad k=0,\ldots,S-1,
\]
as follows from the identity
\[
 \sum_{\ell=1}^{S-1}
 \frac{1-\cos(2\pi k\ell/S)}{\sin^2(\pi\ell/S)}
 =2k(S-k).
\]
Thus $\ell_0=0$ and $\ell_k<0$ for $1\le k<S$, so the
kernel of $L$ consists precisely of constant vectors.

Since $\mathcal F=\overline f$, the real Jacobian satisfies
\[
 (\mathscr Lw)_j=z_j\overline{(Lv)_j}.
\]
Consequently, $\mathscr Lw=0$ if and only if $v$ is a constant
complex vector, which gives
$\ker\mathscr L=\Span_\R\{z,iz\}$.
Moreover, $w\in E$ if and only if $\sum_jv_j=0$.
Because $L$ preserves the zero-sum subspace, the preceding
formula shows that $\mathscr L$ preserves $E$.
Its kernel has trivial intersection with $E$, so
$\mathscr L|_E$ is invertible.

To count the signs of the eigenvalues, write $v=a+ib$ with
$a,b\in\R^S$. In the real orthogonal coordinates
$w_j=z_j(a_j+ib_j)$, the Jacobian acts as
\[
 (a,b)\longmapsto(La,-Lb).
\]
It therefore has $S-1$ negative and $S-1$ positive eigenvalues
on $E$, proving (ii).

For (iii), the Fourier mode $k=1$ gives
$Lz=-m(S-1)z/2=-2nz$. Set
\[
 \xi_j=\frac{\overline H}{2n}z_j^2,
 \qquad
 v_j=\frac{\xi_j}{z_j}=\frac{\overline H}{2n}z_j.
\]
Then $Lv=-\overline H z$, and hence
\[
 (\mathscr L\xi)_j
 =z_j\overline{(Lv)_j}
 =-H z_j\overline{z_j}
 =-H.
\]
Finally, $\sum_{j=0}^{S-1}z_j^2=0$ for $S\ge3$ by the
geometric-series identity. This proves (iii).
\end{proof}

\section{Structure of non-simple clusters}\label{ap:structure}
We isolate the regular cores, compare their heights, and use force balance
to determine the polygon. Excluding mass outside this polygon completes
the classification.

\subsection{Rescaling and regular cores}
Under \eqref{ap:strict}--\eqref{ap:exterior-decay}, fix a non-simple
vortex $p_i$ and write $n=N_i$. In its centered conformal chart,
choose $r_{\mathrm{iso}}>0$ so that $B_{4r_{\mathrm{iso}}}$ contains
no other vortex or blow-up point. Put $a=2(n+1)$ and let $z_\eps$
maximize $\we$ on $\overline{B_{r_{\mathrm{iso}}}}$.
Exterior decay gives $z_\eps\to0$. Maximality and the triangle inequality give
\[
 \sup_{B_{r_{\mathrm{iso}}}\setminus\{z_\eps\}}\{\we(x)+a\log|x-z_\eps|\}
 \le \sup_{B_{r_{\mathrm{iso}}}\setminus\{0\}}\{\we(x)+a\log|x|\}+a\log2.
\]
Since the chart distance and $d_g$ are uniformly comparable,
non-simplicity gives, after taking a subsequence,
$x_\eps\in B_{r_{\mathrm{iso}}}$ satisfying
\begin{equation}\label{ap:first-maximum}
 \we(x_\eps)+2(n+1)\log|x_\eps|
 =\max_{B_{r_{\mathrm{iso}}}}
       \{\we(x)+2(n+1)\log|x|\}\longrightarrow+\infty.
\end{equation}
Exterior decay gives $\de:=|x_\eps|\to0$.
Combining \eqref{ap:first-maximum} with $u_0(x)=2n\log|x|+O(1)$
and $\sup_Mu_\eps\le-c_{\mathrm{nt}}<0$ yields
$\etae:=\eps/\de\to0$. Define
\begin{equation}\label{ap:first-rescaling}
 v_\eps(y):=\we(\de y)-4\pi N_{\tot}\gamma_i(\de y)
 +2(n+1)\log\de+K_i(p_i).
\end{equation}
Then
\begin{equation}
 \label{ap:first-equation}
 \begin{aligned}
  \Delta v_\eps+f_\eps&=0, \\
  f_\eps:&=e^{\phi_i(\de y)}a_\eps e^{v_\eps}
                 (1-\etae^2a_\eps e^{v_\eps}),
 \end{aligned}
\end{equation}
where
\begin{equation}\label{ap:first-coefficient}
 a_\eps(y):=|y|^{2n}
       \exp\{K_i(p_i+\de y)-K_i(p_i)\}.
\end{equation}
Here $\etae^2a_\eps e^{v_\eps}=e^{u_\eps(p_i+\de y)}$, and
\begin{equation}\label{ap:mass-rescaling}
 \frac1{2\pi}f_\eps(y)\,dy
 =\frac{F(u_\eps(p_i+\de y))}{2\pi\eps^2}
             e^{\phi_i(\de y)}\de^2\,dy.
\end{equation}
Thus normalized local mass is preserved. The subtraction of
$4\pi N_{\tot}\gamma_i$ in \eqref{ap:first-rescaling} cancels the smooth
background source, so \eqref{ap:first-equation} is exact.
The density has bounded integral and
$f_\eps\asymp |y|^{2n}e^{v_\eps}\ge0$ on
$B_{2r_{\mathrm{iso}}/\de}$. On fixed disks away from zero,
\begin{equation}\label{ap:local-coefficient-bounds}
 C^{-1}\le a_\eps\le C,\qquad
 \|\log a_\eps\|_{C^{4,\sigma}}\le C.
\end{equation}

\begin{lemma}[Regular cores with smooth local coefficients]
\label{ap:regular-core}
Let $V_j$ solve $\Delta V_j+D_j=0$ on a fixed disk $B_{4r}(q)$,
where $D_j=B_jA_je^{V_j}(1-\tau_j^2A_je^{V_j})$ and $\tau_j\to0$.
Assume $A_j,B_j>0$, $\Delta\log A_j=0$, and, for some $C,c_0>0$,
\[
 \begin{gathered}
 C^{-1}\le A_j,B_j\le C,\qquad
 \|\log A_j\|_{C^{4,\sigma}}+
 \|\log B_j\|_{C^{4,\sigma}}\le C,\\
 \tau_j^2A_je^{V_j}\le1-c_0,\qquad
 \int_{B_{4r}(q)}D_j\le C.
 \end{gathered}
\]
Suppose $q$ is the only blow-up point in $B_{4r}(q)$ and
$V_j\to-\infty$ locally uniformly away from $q$. In the Dirichlet
Green representation $V_j=h_j+\int G_{4r}(\cdot,y)D_j(y)\,dy$,
assume only $\osc_{B_{3r}(q)}h_j\le C$. Choose a maximum point
$x_j$ of $V_j$ on $\overline{B_{3r}(q)}$ and set 
\[
 \begin{gathered}
  \rho_j=e^{-V_j(x_j)/2}, \\
  g_j(z)=\rho_j^2D_j(x_j+\rho_jz), \\
  \text{and}\quad \widetilde\mu_j=(2\pi)^{-1}\int_{B_r(x_j)}D_j.
 \end{gathered}
\]
Then $x_j\to q$, $\rho_j\to0$, and, for each $0<\theta<1$,
\begin{equation}\label{ap:regular-decay}
 g_j(z)\le C_\theta(1+|z|)^{-4+\theta},
 \qquad |z|\le2r/\rho_j.
\end{equation}
In particular, for $1\le R\le r/\rho_j$,
\begin{equation}\label{ap:regular-tail}
 \int_{B_r(x_j)\setminus B_{R\rho_j}(x_j)}D_j
       \le C_\theta R^{-2+\theta}.
\end{equation}
Uniformly for $|x-x_j|\le r$,
\begin{equation}\label{ap:regular-log}
 V_j(x)=V_j(x_j)-\frac{\widetilde\mu_j}{2}
       \log(1+|x-x_j|^2/\rho_j^2)+O(1).
\end{equation}
After passage to a subsequence, one regular profile exhausts the atom:
its normalized mass is $4$ in the mean-field case and greater than $4$
in the CS case. All estimates are uniform as
$\tau_j^2A_j(x_j)e^{V_j(x_j)}\to0$.
\end{lemma}

\begin{proof}
Following \cite[Lemmas 2.7, 2.9--2.11]{LinYanAIHP} and
\cite[Theorem 3.1]{Choe}, absorb the harmonic factor by replacing
$V_j$ with $V_j+\log A_j$. On any scale $\ell_j\to0$,
$B_j(x_j+\ell_jz)$ converges locally, after extraction, to a positive
constant. The primitive
$Q_j(x,V)=B_jA_je^V(1-\tau_j^2A_je^V/2)$ satisfies
$0\le Q_j\le CD_j$ and has explicit spatial derivative
$D_j\nabla\log A_j+Q_j\nabla\log B_j$. Thus the rescaled translation Pohozaev identities have coefficient
errors $O(\ell_j\int D_j)$ and a harmonic error $O(\ell_j)$.
On fixed compact sets, sources outside a rescaled radius $T$
contribute at most $CT^{-1}\int D_j$. Letting $j\to\infty$ and
then $T\to\infty$ removes these errors.

At scale $\rho_j$, local compactness gives an entire regular profile.
If $\tau_j^2A_j(x_j)e^{V_j(x_j)}\to0$, its normalized mass is $4$;
otherwise it exceeds $4$; see \cite[pp.~840--841]{CKL} and
\cite[Lemmas 2.5--2.6]{LinYanAIHP}. A sufficiently large fixed
rescaled disk therefore contains normalized mass at least $4-\theta$.

Suppose a disk of radius $\ell_j/2$, centered at distance
$\ell_j\gg\rho_j$ from $x_j$, retains positive mass.
Exterior decay forces $\ell_j\to0$. Since $\tau_j=O(\rho_j)$, $W_j(z)=V_j(x_j+\ell_jz)+2\log\ell_j$ satisfies
\[
 \begin{aligned}
  \Delta W_j+C_je^{W_j}&=0, \\
  C_j&=B_j^\ell A_j^\ell
       (1-\lambda_j^2A_j^\ell e^{W_j}), \\
  0<c&\le C_j\le C,
 \end{aligned}
\]
where $A_j^\ell(z)=A_j(x_j+\ell_jz)$, similarly for $B_j^\ell$,
and $\lambda_j=\tau_j/\ell_j\to0$.
The mass bound and $W_j(0)=2\log(\ell_j/\rho_j)\to\infty$
allow us to apply \cite[Theorem 3]{BM} on expanding disks. We obtain finitely
many atoms $\xi_a$ of normalized masses $\mu_a>0$, including zero
and a nonzero atom, with $W_j\to-\infty$ elsewhere. The translation
Pohozaev identity around each atom gives
\[
 \sum_{b\ne a}\mu_b(\xi_a-\xi_b)/|\xi_a-\xi_b|^2=0
\]
 by the exterior Green-gradient limit and the preceding error bounds;
no profile exhaustion is needed. Taking scalar products with
$\mu_a\xi_a$ and summing gives $\sum_{a<b}\mu_a\mu_b=0$, a
contradiction. Hence
\begin{equation}\label{ap:moving-disk}
 \lim_{R\to\infty}\limsup_{j\to\infty}
 \sup_{R\rho_j\le|x-x_j|\le2r}
       \int_{B_{|x-x_j|/2}(x)}D_j=0.
\end{equation}
These disks lie in $B_{4r}(q)$ for large $j$.

To deduce decay from \eqref{ap:moving-disk}, set
$\widehat V_j(z)=V_j(x_j+\rho_jz)-V_j(x_j)$ and
$\Omega_j=(B_{4r}(q)-x_j)/\rho_j$. Subtracting the Green formula at
$x_j$ gives, uniformly for $|z|\le2r/\rho_j$,
\begin{equation}\label{ap:regular-green}
 \widehat V_j(z)=\frac1{2\pi}\int_{\Omega_j}
          \log\frac{|y|}{|z-y|}\,g_j(y)\,dy+O(1).
\end{equation}
The remainder is bounded by the harmonic oscillation, regular Green
part and total mass. Maximality of $V_j(x_j)$ gives $g_j\le C$ on
$B_a(z)$, $0<a<1$, whose image in the original coordinates lies in
$B_{2r+\rho_j}(x_j)\subset B_{3r}(q)$.

Choose $R_\theta$ so that
$(2\pi)^{-1}\int_{B_{R_\theta}}g_j\ge4-\theta/4$ for large $j$.
For $s=|z|\ge2R_\theta$, split \eqref{ap:regular-green} into
$B_{R_\theta}$, $B_{s/2}(z)$ and their complement. The first term
is at most $-(4-\theta/4)\log s+C_\theta$; the complement contributes
at most $C$ because $|y|\le3|y-z|$. In $B_{s/2}(z)$, the bound
$g_j\le C$ gives
\[
 \frac1{2\pi}\int_{B_a(z)}
       \log\frac{|y|}{|z-y|}\,g_j(y)\,dy
       \le Ca^2\log(2+s)+C_a,
\]
while the contribution of $B_{s/2}(z)\setminus B_a(z)$ is at most
$(2\pi)^{-1}\int_{B_{s/2}(z)}g_j\,\log(Cs/a)$.
Choose $Ca^2<\theta/4$ and use \eqref{ap:moving-disk} to make
$(2\pi)^{-1}\int_{B_{s/2}(z)}g_j<\theta/4$. For bounded $s$,
use $\widehat V_j\le0$.
Thus 
\[
 \widehat V_j(z)\le-(4-\theta)\log(1+|z|)+C_\theta
\]
 on
$|z|\le2r/\rho_j$, and $g_j\le Ce^{\widehat V_j}$ proves
\eqref{ap:regular-decay}.

For the annuli
\[
 A_{j,k}=\{2^kR\rho_j\le|x-x_j|<2^{k+1}R\rho_j\}\cap B_{2r}(x_j),
\]
this gives
\begin{equation}
 \label{ap:annular-sum}
 \begin{aligned}
  \int_{A_{j,k}}D_j&\le C_\theta(2^kR)^{-2+\theta}, \\
  \int_{B_{2r}(x_j)\setminus B_{R\rho_j}(x_j)}D_j
       &\le\frac{C_\theta R^{-2+\theta}}{1-2^{-2+\theta}}.
 \end{aligned}
\end{equation}
Summing proves \eqref{ap:regular-tail}, excludes diffuse neck mass,
and gives
\[
 \int_{|y|\le2r/\rho_j}\log(2+|y|)g_j(y)\,dy\le C_\theta.
\]
Hence the logarithmic remainders in the Green-kernel decomposition
of \cite[Theorem 3.1]{Choe} are uniformly bounded, with the finite-disk
mass as leading coefficient \cite[equation (32)]{Choe}.
Replacing the mass of $B_{2r}(x_j)$ by that of $B_r(x_j)$ introduces an error
$O(\rho_j^{2-\theta}|\log\rho_j|)=o(1)$, yielding
\eqref{ap:regular-log}.
\end{proof}

To apply this lemma at a regular blow-up point of $\we$, choose a
vortex-free conformal chart and smooth $\Gamma$ with
$\Delta\Gamma=e^\phi/|M|_g$. Set
$V_\eps=\we-4\pi N_{\tot}\Gamma$,
$A=e^{u_0+4\pi N_{\tot}\Gamma}$, and $B=e^\phi$.
Then $\Delta\log A=0$, and $V_\eps$ satisfies the lemma's equation
with $\tau=\eps$. The global Green formula bounds the harmonic oscillation.
Since $4\pi N_{\tot}\Gamma$ has bounded oscillation,
\eqref{ap:regular-log} places a maximum of $\we$ within $O(\rho_\eps)$
of the maximum of $V_\eps$, where $\rho_\eps=e^{-\max V_\eps/2}$.
The same estimate, with this change of center and the smooth term
restored, proves \eqref{eq:simple} with $n=0$.
\subsection{Force balance and equality of masses}
\begin{lemma}[Concentration points and force balance]\label{ap:balance}
After passage to a subsequence, there are finitely many distinct
nonzero centers $q_t$, $t=0,\ldots,S-1$, with $|q_0|=1$.
Choose maxima $q_{t,\eps}\to q_t$ of $v_\eps$ on fixed disjoint
disks about $q_t$. Denote their limiting normalized masses by $b_t>0$
and the possible mass at zero by $c$. Then $c=0$, $S\ge2$, and
\begin{equation}\label{ap:balance-equation}
 2n\frac{q_t}{|q_t|^2}
       -\sum_{s\ne t}b_s\frac{q_t-q_s}{|q_t-q_s|^2}=0.
\end{equation}
Moreover, $v_\eps\to-\infty$ uniformly on a small disk about zero.
\end{lemma}

\begin{proof}
Apply the concentration alternative of \cite{ChoeKim} and
\cite[Section 2]{LinYanAIHP} on expanding annuli about zero,
starting with the unit-circle blow-up point from
\eqref{ap:first-maximum}. On each disk containing a single atom, the Green
formula bounds the harmonic oscillation because the remaining sources
are separated from the disk and have bounded mass.
Lemma~\ref{ap:regular-core} applies with
$A_j=a_\eps$, $B_j=e^{\phi_i(\de\,\cdot)}$, $\tau_j=\etae$,
since $\log a_\eps$ is harmonic away from zero,
\eqref{ap:local-coefficient-bounds} bounds the coefficients, and
$1-\etae^2a_\eps e^{v_\eps}\ge1-e^{-c_{\mathrm{nt}}}>0$.
Thus each regular atom is exhausted by one regular profile,
$b_t\ge4$, and $S$ is finite.

The Green representation gives, on compact subsets of
$\R^2\setminus\{0,q_0,\ldots,q_{S-1}\}$,
\begin{equation}\label{ap:exterior-gradient}
 \nabla v_\eps(y)\longrightarrow
       -c\frac{y}{|y|^2}-\sum_t b_t\frac{y-q_t}{|y-q_t|^2}.
\end{equation}
Indeed, for fixed large $R$, local measure convergence and exterior
decay give an error $o_\eps(1)$; smooth Green terms contribute
$O(\de)$, and sources outside $B_R$ contribute $O(R^{-1})$.
Let $\eps\to0$, then $R\to\infty$. Subtracting Green representations
at two exterior points also bounds their oscillation.

The densities $e^{\phi_i(\de y)}\etae^2a_\eps^2e^{2v_\eps}$ are
nonnegative and bounded by $Cf_\eps$. Apply the radial Pohozaev identity to \eqref{ap:first-equation}
on $B_r(0)$, taking continuity radii of the limiting measure,
then letting $\eps\to0$ and $r\downarrow0$:
\[
 \pi c^2=4\pi(n+1)c+\lim_{r\downarrow0}\lim_{\eps\to0}
       \int_{B_r}e^{\phi_i(\de y)}\etae^2a_\eps^2e^{2v_\eps}\,dy.
\]
Thus $c=0$ or $c\ge4(n+1)$. The primitive
\[
 P_\eps(y,v)=e^{\phi_i(\de y)}a_\eps e^v
(1-\etae^2a_\eps e^v/2)
\]
 has explicit spatial derivative
$f_\eps\nabla\log a_\eps+
P_\eps\nabla_y\phi_i(\de y)$. Since $|P_\eps|\le Cf_\eps$, the metric contribution vanishes.
The translation Pohozaev identity on each core disk and
\eqref{ap:exterior-gradient} give
\[
 (2n-c)\frac{q_t}{|q_t|^2}
          -\sum_{s\ne t}b_s\frac{q_t-q_s}{|q_t-q_s|^2}=0.
\]
Taking scalar products with $b_tq_t$ and summing yields
\begin{equation}\label{ap:unequal-virial}
 (2n-c)\sum_t b_t=\sum_{s<t}b_sb_t.
\end{equation}
Thus $c\le2n$, which excludes $c\ge4(n+1)$ and forces $c=0$.
Since $n>0$, \eqref{ap:unequal-virial} also excludes $S=1$.

On a disk $B_\varrho(0)$ without regular atoms, decompose
$v_\eps=h_\eps+\zeta_\eps$ into its harmonic extension and
zero-boundary Green potential. Exterior decay gives
$\sup h_\eps\to-\infty$, and $c=0$ gives
$\int_{B_\varrho}f_\eps\to0$. For each fixed $p>1$,
\cite[Theorem 1]{BM} bounds $\int_{B_\varrho}e^{p\zeta_\eps}$.
Since 
\[
 \begin{gathered}
  f_\eps\le Ce^{v_\eps}, \\
  \|f_\eps\|_{L^p}\le Ce^{\sup h_\eps}\|e^{\zeta_\eps}\|_{L^p}\to0.
 \end{gathered}
\]
The Dirichlet $W^{2,p}$ estimate gives $\|\zeta_\eps\|_\infty\to0$,
hence uniform decay of $v_\eps$ near zero.
\end{proof}

Lemma~\ref{ap:balance} applies at any scale $s_\eps\to0$ with
$\eps/s_\eps\to0$ and a nonzero regular concentration point.

\begin{lemma}[Choice of center at a vortex]\label{ap:choice-of-center}
Under \eqref{ap:strict}--\eqref{ap:exterior-decay}, let $p=p_i$,
$n=N_i>0$, and $a=2(n+1)$. For maximum points $x_\eps$ as in
Definition~\ref{def:simple}, put $H_\eps=\we(x_\eps)$.
Then \eqref{eq:simple} is equivalent to
\[
 \sup_{x\in B_r^g(p)\setminus\{p\}}
       \{\we(x)+a\log d_g(x,p)\}\le C.
\]
Either bound implies $d_g(x_\eps,p)=O(e^{-H_\eps/a})$.
\end{lemma}
\begin{proof}
The fixed-center bound implies \eqref{eq:simple} by the maximum
comparison at the beginning of this section. Conversely, assume
\eqref{eq:simple}. If the asserted distance bound fails, pass to a
subsequence with $d_\eps:=|x_\eps|\to0$ and
$H_\eps+a\log d_\eps\to+\infty$ in the chart centered at $p$.
Rescaling as in \eqref{ap:first-rescaling}, with $\de=d_\eps$,
gives $\etae=\eps/d_\eps\to0$ and a regular concentration point
$q=\lim x_\eps/d_\eps$, $|q|=1$. On the expanding disk,
\eqref{eq:simple} gives
\[
 v_\eps(y)+a\log|y-x_\eps/d_\eps|\le C.
\]
Thus $v_\eps$ is locally bounded above away from $q$, and no other
finite concentration point can occur. Near zero,
$f_\eps(y)\le C|y|^{2n}$, so there is no atom there either.
This contradicts $S\ge2$ in Lemma~\ref{ap:balance},
whose proof applies at this scale. The distance bound follows.
Finally, maximality and the triangle inequality give
\[
 e^{\we(x)/a}d_g(x,p)
 \le e^{\we(x)/a}d_g(x,x_\eps)
       +e^{H_\eps/a}d_g(x_\eps,p)\le C,
\]
which proves the fixed-center bound.
\end{proof}

Choose $r>0$ so that the disks $B_{2r}(q_{t,\eps})$ are disjoint and
avoid zero, and define
\begin{equation}
 \label{ap:actual-masses-heights}
 \begin{aligned}
  \widetilde m_{t,\eps}:&=\frac1{2\pi}
                  \int_{B_r(q_{t,\eps})}f_\eps, \\
  \vartheta_{t,\eps}:&=v_\eps(q_{t,\eps})+2\log\etae.
 \end{aligned}
\end{equation}
The rescaled core width is
\[
 \rho_{t,\eps}=e^{-v_\eps(q_{t,\eps})/2}
=\etae e^{-\vartheta_{t,\eps}/2},
\]
 and the height of the original solution is
\[
 \vartheta_{t,\eps}+\log a_\eps(q_{t,\eps})
=u_\eps(p_i+\de q_{t,\eps}).
\]

Applying \eqref{ap:regular-log} on slightly larger core disks gives,
uniformly on $B_r(q_{t,\eps})$,
\begin{equation}\label{ap:core-logarithm}
 v_\eps(y)=v_\eps(q_{t,\eps})
   -\frac{\widetilde m_{t,\eps}}2
       \log\bigl(1+e^{v_\eps(q_{t,\eps})}|y-q_{t,\eps}|^2\bigr)+O(1).
\end{equation}

\begin{lemma}[Equal masses and a common profile type]\label{ap:equal-mass}
There is $m\ge4$ such that
\begin{equation}\label{ap:equal-mass-limit}
 \widetilde m_{t,\eps}\longrightarrow m
                   \qquad(t=0,\ldots,S-1).
\end{equation}
Either all cores are of mean-field type and $m=4$, or all are of
Chern--Simons type and $m>4$.
\end{lemma}
\begin{proof}
Evaluate \eqref{ap:core-logarithm} at fixed distances from each core,
as in \cite[Lemma 3.3]{CKL} and \cite[proof of Theorem 1.1]{LinYanAIHP}.
Bounded exterior oscillation and cancellation of $-2\log\etae$ give
\begin{equation}\label{ap:height-comparison}
 \begin{aligned}
 \widetilde m_{t,\eps}\log\etae
       -\frac{\widetilde m_{t,\eps}-2}{2}\vartheta_{t,\eps}
 ={}&\widetilde m_{s,\eps}\log\etae
       -\frac{\widetilde m_{s,\eps}-2}{2}\vartheta_{s,\eps}+O(1).
 \end{aligned}
\end{equation}
After extraction, each $\vartheta_{t,\eps}$ is bounded or tends to
$-\infty$, giving mass greater than $4$ or equal to $4$, respectively.
If $\vartheta_{t,\eps}$ is bounded and $\vartheta_{s,\eps}\to-\infty$,
\eqref{ap:height-comparison} becomes
\[
 (\widetilde m_{t,\eps}-\widetilde m_{s,\eps})\log\etae
       =-\frac{\widetilde m_{s,\eps}-2}{2}\vartheta_{s,\eps}+O(1),
\]
whose sides tend to opposite infinities. Thus all cores have the same
type. For bounded heights, divide \eqref{ap:height-comparison} by
$\log\etae$ to obtain equal limiting masses; otherwise all masses
are $4$. In the CS case, strict monotonicity of the mass map
\eqref{eq:profile-facts} also gives a common limiting height for the original solution
$\vartheta_{t,\eps}+\log a_\eps(q_{t,\eps})$.
\end{proof}

\subsection{Polygon structure and exhaustion of the vortex mass}
\begin{proposition}[Mass and polygon identities]\label{ap:mass-polygon}
After rotation and relabeling,
\begin{equation}\label{ap:polygon}
 q_t=e_t:=e^{2\pi\mathrm i t/S},\qquad t=0,\ldots,S-1.
\end{equation}
The common core mass and the full vortex mass satisfy
\begin{align}
 m(S-1)&=4N_i,\label{ap:mass-compatibility}\\
 M_{p_i}&=Sm=\frac{4N_iS}{S-1},\label{ap:full-vortex-mass}\\
 1&\le S-1\le N_i,\label{ap:core-number}
\end{align}
with equality on the right exactly in the mean-field case.
\end{proposition}

\begin{proof}
Substituting the common mass into \eqref{ap:balance-equation} gives
\begin{equation}\label{ap:equal-balance}
 2n\frac{q_t}{|q_t|^2}
       -m\sum_{s\ne t}\frac{q_t-q_s}{|q_t-q_s|^2}=0.
\end{equation}
Taking scalar products with $q_t$ and summing in $t$ gives
$2nS=mS(S-1)/2$, proving \eqref{ap:mass-compatibility}.
Identify the plane with $\C$ and rotate so that $q_0=1$.
Complex conjugation of \eqref{ap:equal-balance} gives
\begin{equation}\label{ap:complex-balance}
 \sum_{s\ne t}\frac{q_t}{q_t-q_s}=\frac{S-1}{2}.
\end{equation}
For $p(z)=\prod_t(z-q_t)$, the relation
$p''(q_t)/(2p'(q_t))=\sum_{s\ne t}(q_t-q_s)^{-1}$ implies that
$zp''-(S-1)p'$ vanishes at all $S$ distinct roots. Its degree is
less than $S$, so it vanishes identically. Thus $p(z)=z^S-1$,
which proves \eqref{ap:polygon}.

To recover the full vortex mass, fix $R_0>2$ with all $q_t\in B_{R_0}$.
We first claim
\begin{equation}\label{ap:outer-weighted-bound}
 v_\eps(y)+2(n+1)\log|y|\le C,
           \qquad R_0\le|y|\le2r_{\mathrm{iso}}/\de.
\end{equation}
Otherwise choose $y_\eps$ with
$v_\eps(y_\eps)+2(n+1)\log|y_\eps|\to\infty$.
Decay away from the centers and away from $p_i$ gives
$T_\eps=|y_\eps|\to\infty$ and $\ell_\eps=\de T_\eps\to0$.
Set 
\[
 \begin{gathered}
  g_\eps(z)=v_\eps(T_\eps z)+2(n+1)\log T_\eps, \\
  \widehat\eta_\eps=\etae/T_\eps=\eps/\ell_\eps, \\
  \text{and}\quad \widehat a_\eps(z)=T_\eps^{-2n}a_\eps(T_\eps z)
=|z|^{2n}e^{K_i(p_i+\ell_\eps z)-K_i(p_i)}.
 \end{gathered}
\]
 Then
\begin{equation}\label{ap:second-rescaling}
 \begin{gathered}
 \Delta g_\eps+\widehat f_\eps=0,\\
 \widehat f_\eps(z):=T_\eps^2f_\eps(T_\eps z)
 =e^{\phi_i(\ell_\eps z)}\widehat a_\eps(z)e^{g_\eps(z)}
       (1-\widehat\eta_\eps^2\widehat a_\eps(z)e^{g_\eps(z)}).
 \end{gathered}
\end{equation}
The mass and the strictly negative upper bound for the original solution are preserved, and
$\widehat\eta_\eps\to0$. The new scale has a regular concentration
point on the unit circle and a central atom of normalized mass
$\widehat c\ge Sm>0$: for fixed $\varrho>0$, $B_{\varrho T_\eps}$
eventually contains every fixed disk around the original cluster.
This contradicts Lemma~\ref{ap:balance} at scale $\ell_\eps$,
proving \eqref{ap:outer-weighted-bound}.

To exclude diffuse outer mass, localize \eqref{ap:green-representation}:
\begin{equation}
 \label{ap:outer-green}
 \begin{aligned}
  v_\eps(y)-v_\eps(0)
  &=\frac1{2\pi}\int_{B_{2r_{\mathrm{iso}}/\de}}
           \log\frac{|z|}{|y-z|}\,f_\eps(z)\,dz+O(1), \\
  |y|&\le r_{\mathrm{iso}}/\de.
 \end{aligned}
\end{equation}
Smoothness of the regular Green part and separation of omitted
sources bound the remainder uniformly. For $|y|\ge2R_0$, split
the integral into $B_{R_0}$, $B_{|y|/2}(y)$ and their complement.
The first contribution is $-\widetilde M_\eps(R_0)\log|y|+O(1)$,
where 
\[
 \widetilde M_\eps(R_0):=(2\pi)^{-1}\int_{B_{R_0}}f_\eps\to Sm.
\]
Uniform decay near zero bounds the logarithmic moment there.
On $B_{|y|/2}(y)$, \eqref{ap:outer-weighted-bound} gives
$f_\eps(z)\le C|z|^{-2}$; scaling by $|y|$ bounds the kernel integral.
On the complement, use $|z|\le3|z-y|$ and the mass bound. Thus, for fixed
$\kappa>0$ and small $\eps$,
\begin{equation}
 \label{ap:outer-decay}
 \begin{aligned}
  v_\eps(y)&\le v_\eps(0)-(Sm-\kappa)\log|y|+C_\kappa, \\
  2R_0&\le|y|\le r_{\mathrm{iso}}/\de.
 \end{aligned}
\end{equation}
The loss $\kappa$ absorbs the error $\widetilde M_\eps(R_0)-Sm$.
Since $Sm=4n+m$ and $m\ge4$,
choose $0<\kappa<Sm-2n-2$ to obtain
\begin{equation}\label{ap:outer-neck-exhaustion}
 \int_{B_{r_{\mathrm{iso}}/\de}\setminus B_{2R_0}}f_\eps
       \le C_\kappa e^{v_\eps(0)}
          \int_{2R_0}^\infty t^{2n-Sm+\kappa+1}\,dt
                       \longrightarrow0.
\end{equation}
Since $B_{2R_0}$ has normalized mass $Sm+o_\eps(1)$, rescaling back
proves \eqref{ap:full-vortex-mass}. Finally, $m\ge4$ gives
\eqref{ap:core-number}.
\end{proof}

\begin{proof}[Proof of Theorem~\ref{ap:mass-gap}]
For a non-simple cluster at $p_i$,
Proposition~\ref{ap:mass-polygon} gives $k=S-1\in\mathcal A_i$,
$b_1=\cdots=b_{k+1}=4N_i/k$, and $M_{p_i}=M_{i,k}$.
By \eqref{ap:concentration} and \eqref{ap:total-mass},
\[
 2N_{\tot}-M_{i,k}
 =\sum_{q\in\cS_u\setminus\{p_i\}}M_q.
\]
The sum is zero if there are no other blow-up points;
otherwise it is at least $4$, since every $M_q\ge4$.
This proves \eqref{ap:necessary-alternative}.
\end{proof}
\section{Lyapunov--Schmidt reduction in the CS regime}\label{mc:analytic-section}
We solve \eqref{mc:pde} up to $2S$ translation multipliers. By the
natural constraint, these multipliers vanish at critical points of the
reduced energy. Fix an admissible pair $(S,m)$ with $S\ge2$ and $m\ge6$,
and the integer divisor in \eqref{mc:pde}. Thus
\[
 n=\frac{m(S-1)}4,\qquad N=\frac{mS}{2},\qquad
 \frac m2=N-2n\in\mathbb Z.
\]
In particular, either $m=6$ with $S$ odd, or $m\ge8$.
We use \eqref{eq:green}, \eqref{mc:u0}, and \eqref{mc:chart}
throughout this section; constants may depend on $S$, $m$, the metric,
and the divisor.

Put $L=1+|\log\varepsilon|$ and
\begin{equation}\label{mc:analytic-regimes}
 s=\begin{cases}
 \varepsilon^{2/3},&m>6,\\
 \varepsilon^{2/3}L^{1/6},&m=6.
 \end{cases}
\end{equation}
For $j=0,\ldots,S-1$, set $e_j(\theta)=\exp(i\theta+2\pi ij/S)$.
Let $\xi=(\xi_0,\ldots,\xi_{S-1})\in\mathbb C^S$ and define the
reference configuration $q^0=(q_0^0,\ldots,q_{S-1}^0)$ by
\begin{equation}\label{mc:analytic-parameters}
 \begin{aligned}
  q_j^0&=sd\,e_j(\theta)+s^2\xi_j,\qquad
  d\in[d_-,d_+]\Subset(0,\infty),\\
  |\xi|&\le C_*,\qquad \delta=sd.
 \end{aligned}
\end{equation}
The shape coordinates are taken in a fixed basis or in a basis depending on
$\theta$ with uniformly bounded first derivatives. When differentiating along
a prescribed curve $\xi=\xi(\theta)$, we also assume that $\xi'(\theta)$ is
uniformly bounded. For each reference configuration $q^0$, we consider
unrestricted positions
\[
 q=(q_0,\ldots,q_{S-1})\in
 B_{cs^2}(q^0)\subset(\mathbb R^2)^S\simeq\mathbb C^S.
\]
Position derivatives are taken independently with respect to all $2S$ real
coordinates of $q$. For sufficiently small $\varepsilon$, there is a constant
$c_S>0$, independent of $\varepsilon$ and uniform over these neighborhoods,
such that
\[
 |q_j-q_\ell|\ge c_Ss\quad(j\ne\ell),
 \qquad |q_j|\ge c_Ss\quad(j=0,\ldots,S-1).
\]
The constant $c_S$ may depend on $S$ and the fixed interval $[d_-,d_+]$.
Set $\rho_j=\varepsilon e^{-\sigma(q_j)/2}$,
$\rho=\varepsilon$, and $\eta=\rho/\delta$. Then
\[
 \begin{gathered}
  \eta^2\asymp s\quad(m>6), \\
  \eta^2\asymp sL^{-1/2}\quad(m=6),
 \end{gathered}
\]
and in both cases $\eta^2L^2\to0$, $\eta^3\le C\rho$,
and $\eta^{m-2}=o(s^2)$.
For the profiles of Lemma~\ref{mc:planar}, restrict masses to a
compact interval with lower endpoint $a_*>4$. We also use
\begin{equation}\label{mc:analytic-profile}
 \begin{aligned}
  Z_a:&=\partial_aU_a=-\log r+c_\infty'(a)
 +O(r^{2-a}\log(2+r)), \\
  \int_{\R^2}F'(U_a)Z_a&=2\pi.
 \end{aligned}
\end{equation}
Smooth ODE dependence and the integrable majorants
$r^{1-a_*}\log^k(2+r)$ justify differentiation of these tails.
Choose $4<\beta<\min\{5,m-2\}$ for $m>6$ (also $\beta=5$ if
$m\ge8$), or $\beta=3$ for $m=6$, and define
\[
 \begin{aligned}
  \omega_\varepsilon&=\sum_j\rho_j^{-2}
 (1+d_g(x,q_j)/\rho_j)^{-\beta}, \\
  \|f\|_*&=\sup_M|f|/\omega_\varepsilon.
 \end{aligned}
\]

\subsection{The relative logarithmic matching system}
Let $\mu=(\mu_0,\ldots,\mu_{S-1})\in\mathbb R^S$, and write
$a_j=m+\mu_j$ for $j=0,\ldots,S-1$. For each unrestricted position
$q$ in the above neighborhood, set
\begin{align}
 B_j(x)&=u_0(x)+2\pi a_jR(x,q_j)
                  +2\pi\sum_{\ell\ne j}a_\ell G(x,q_\ell),\notag\\
 C_j(\mu,q)&=a_j\log\rho_j+c_\infty(a_j)-B_j(q_j).
                                                        \label{mc:analytic-C}
\end{align}
When $\sum_{j=0}^{S-1}\mu_j=0$, the identity
$\sum_{j=0}^{S-1}a_j=mS=2N$ makes $B_j$ harmonic on its core disk.
\begin{lemma}\label{mc:matching}
For a fixed sufficiently large constant $C$ and sufficiently small
$\varepsilon$, the equations
\[
 C_0(\mu,q)=\cdots=C_{S-1}(\mu,q),\qquad
 \sum_{j=0}^{S-1}\mu_j=0
\]
have a unique solution $\mu\in\mathbb R^S$ satisfying $|\mu|\le Cs/L$,
for every $q$ in the position neighborhood specified above.
The solution depends $C^1$ on the unrestricted positions $q$ and satisfies
\(
 |\partial_q\mu|\le \frac{C}{sL}.
\)
Along the reference family \eqref{mc:analytic-parameters}, its derivatives
with respect to the scaled coordinates $\alpha=(d,\theta,\xi)$ satisfy
\(
 |\partial_\alpha\mu|\le \frac{Cs}{L}.
\)
Here the real and imaginary parts of $\xi$ are treated as independent
coordinates.
\end{lemma}
\begin{proof}
Let \(E=\left\{\mu\in\mathbb R^S:
\sum_{j=0}^{S-1}\mu_j=0\right\}\), and let \(P_E\) be the orthogonal
projection from \(\mathbb R^S\) onto \(E\). Writing
\(\boldsymbol1=(1,\ldots,1)^T\), we have
\(
 P_Ev=v-\frac1S(\boldsymbol1^Tv)\boldsymbol1,
 \) and  \(P_E\boldsymbol1=0.
\)

For fixed \(q\), let \(M=(M_{j\ell})_{j,\ell=0}^{S-1}\) be the matrix
with entries \(M_{j\ell}:=\partial C_j/\partial\mu_\ell\). Then
\[
 M_{jj}=\log\rho_j+c_\infty'(a_j)-2\pi R(q_j,q_j),
 \qquad
 M_{j\ell}=-2\pi G(q_j,q_\ell)\quad(j\ne\ell).
\]
The uniform separation of the positions gives
\begin{equation}
 \label{mc:analytic-matrix}
 \begin{aligned}
 &M=\log(\varepsilon/\delta)I
     +\log\delta\,\boldsymbol1\boldsymbol1^T+A,\quad\text{where}\quad  \|A\|\le C,\\
 &\|(P_EM|_E)^{-1}\|\le \frac{C}{L}.
 \end{aligned}
\end{equation}

On the regular polygon, \(|e_j|=1\) and
\(\prod_{\ell\ne j}|e_j-e_\ell|=S\). Thus the singular matching
terms are independent of \(j\) and are annihilated by \(P_E\).
The smooth terms and the \(O(s^2)\) displacements yield
\(
 |P_EC(0,q)|\le Cs.
\)
Set \(\lambda=\log(\varepsilon/\delta)\) and write
\[
 P_EC(\mu,q)=P_EC(0,q)+\lambda\mu+\mathcal A(\mu,q),
 \qquad \mu\in E,
\]
where \(\mathcal A(0,q)=0\) and \(\|D_\mu\mathcal A\|\le C\).
Since \(|\lambda|\asymp L\), the map
\[
 \mu\longmapsto
 -\lambda^{-1}\bigl(P_EC(0,q)+\mathcal A(\mu,q)\bigr)
\]
is a contraction on a sufficiently large ball
\(\{\mu\in E:|\mu|\le Cs/L\}\). Its unique fixed point satisfies
\(P_EC(\mu,q)=0\).

Along the reference family, write
\[
 q_j^0=sQ_j(\alpha),\qquad
 Q_j(\alpha)=de_j(\theta)+s\xi_j,
 \qquad \alpha=(d,\theta,\xi).
\]
The same cancellation on the regular polygon gives
\[
 \begin{aligned}
  |P_EC(0,q^0)|
   +|\partial_\alpha P_EC(0,q^0)|&\le Cs,\quad\
  \|\partial_\alpha A\|
   +|\partial_\alpha\lambda|&\le C.
 \end{aligned}
\]
Differentiating \(P_EC(\mu,q^0)=0\) with respect to \(\alpha\), and
holding \(\mu\) fixed in the derivative on the right-hand side, gives
\[
 \begin{aligned}
  (P_EM|_E)\partial_\alpha\mu
   &=-\partial_\alpha P_EC(\mu,q^0),\\
  |\partial_\alpha P_EC(\mu,q^0)|
   &\le C(s+|\mu|).
 \end{aligned}
\]
Together with \eqref{mc:analytic-matrix} and \(|\mu|\le Cs/L\), this
yields \(|\partial_\alpha\mu|\le Cs/L\).

Similarly, \(|\partial_qC|\le C/s\). Differentiating
\(P_EC(\mu,q)=0\) with respect to the unrestricted positions and using
\eqref{mc:analytic-matrix}, we obtain
\(
 |\partial_q\mu|\le \frac{C}{sL}.
\)
\end{proof}

\subsection{Approximation}
Use the matched masses and $C=C_j$ from Lemma~\ref{mc:matching}.
Choose $r_c>0$ so that $B_{3r_c\delta}(q_j)$ are disjoint and avoid zero.
Let $0\le\chi\le1$ be smooth and radial, equal to one on $[0,1]$
and zero on $[2,\infty)$. Put
\begin{align}
 O&=C+u_0+2\pi\sum_j a_jG(\cdot,q_j),\notag\\
 I_j(x)&=U_{a_j}((x-q_j)/\rho_j)+B_j(x)-B_j(q_j),\notag\\
 u^{\rm app}&=O+\sum_j\chi(|x-q_j|/(r_c\delta))(I_j-O).
                                                        \label{mc:analytic-app}
\end{align}
The core singularities cancel, so $u^{\rm app}-u_0$ and the residual
\[
 E_{\rm app}:=\Delta_g(u^{\rm app}-u_0)+\varepsilon^{-2}F(u^{\rm app})
-4\pi N/|M|_g
\]
 are smooth.
\begin{lemma}\label{mc:analytic-residual}
The residual satisfies 
\[
 \begin{gathered}
  \|E_{\rm app}\|_*\le C\eta^2, \qquad 
  \|\varepsilon^{-2}e^{u^{\rm app}}\|_*\le C, \\
  \text{and}\quad \sup u^{\rm app}\to U_m(0)<0.
 \end{gathered}
\]
 On $D_j=B_{r_c\delta}(q_j)$,
\begin{align}
 |B_j(q_j+\rho_jz)-B_j(q_j)|&\le C(\rho|z|+\eta^2|z|^2),\notag\\
 |\sigma(q_j+\rho_jz)-\sigma(q_j)|&\le C(\rho s|z|+\rho^2|z|^2).
                                                   \label{mc:analytic-incident}
\end{align}
Moreover,
\begin{equation}\label{mc:analytic-tail}
 \int_{M\setminus\cup D_j}\varepsilon^{-2}e^{u^{\rm app}}\,dV_g
          \le C\eta^{m-2}=o(s^2).
\end{equation}
\end{lemma}
\begin{proof}
The polygon identity
\[
 \begin{gathered}
  \sum_{\ell\ne j}(e_j-e_\ell)/|e_j-e_\ell|^2=(S-1)e_j/2 \\
  \text{and}\quad 2n=m(S-1)/2
 \end{gathered}
\]
 cancel the singular gradient of order
$\delta^{-1}$. The $O(s^2)$ center shifts and mass changes contribute $O(1)$ and
$O(|\mu|/s)=O(L^{-1})$; $D^2B_j=O(s^{-2})$.
Taylor expansion and conformal normality yield
\eqref{mc:analytic-incident}.
Since $|a_j-m|L=O(s)$, the tails
$(1+r)^{-a_j}$ and $(1+r)^{-m}$ are comparable for $r\le C/\eta$.
The weighted inner residual is bounded by
\[
 C(1+r)^{\beta-m}(\rho r+\eta^2r^2+\rho sr+\rho^2r^2)
\le C\eta^2.
\]
On the cutoff annulus, matching leaves only
$U_{a_j}(r)+a_j\log r-c_\infty(a_j)$, whose derivatives in the original coordinates
of order $k\le2$ are $O(\delta^{-k}\eta^{m-2})$.
Thus the weighted cutoff error is $O(\eta^{m-\beta})=o(\eta^2)$.
The common constant is
$C=m\log\varepsilon+2n\log\delta+O(1)$, and outside the cores
\[
 e^{u^{\rm app}(x)}\le
 \begin{cases}
 C\eta^m,&|x|\le C_S\delta,\\
 C\varepsilon^m\delta^{2n}|x|^{-m-2n},
       &C_S\delta\le|x|\le r_0,\\
 C\varepsilon^m\delta^{2n},&x\notin B_{r_0}(0).
 \end{cases}
\]
Integration gives
\eqref{mc:analytic-tail}, division by $\omega_\varepsilon$ gives exterior
residual $C\eta^{m-\beta}$, and inner convergence gives the remaining claims.
\end{proof}

\subsection{The inverse with all translation constraints}
Impose orthogonality to the approximate translations. Write $U_j(x)=U_{a_j}((x-q_j)/\rho_j)$,
$\chi_j=\chi(|x-q_j|/(r_c\delta))$, and, for fixed $R_0>1$,
\[
 \begin{aligned}
  V_g^0&=e^{-\sigma}\sum_j\chi_j\rho_j^{-2}F'(U_j), \\
  K_{j,b}&=e^{-\sigma}\rho_j^{-2}\chi(|x-q_j|/(R_0\rho_j))
              \partial_bU_{a_j}((x-q_j)/\rho_j), \\
  b&=1,2.
 \end{aligned}
\]
Extend these functions by zero outside the chart. The $\chi_j$ have
disjoint supports and equal one on $D_j$; each $K_{j,b}$ is supported
in $D_j$ and has zero integral.

\begin{lemma}\label{mc:inverse}
For every continuous $h$, there is a unique solution
$(\phi,t)\in W^{2,p}(M)\times\mathbb R^{2S}$, for every finite $p>1$,
of
\begin{equation}\label{mc:analytic-linear}
 (\Delta_g+V_g^0)\phi=h+\sum_{j,b}t_{j,b}K_{j,b},\qquad
 \int_MK_{j,b}\phi\,dV_g=0.
\end{equation}
Uniformly on the parameter neighborhood,
\[
 \begin{gathered}
  \|\phi\|_\infty\le CL\|h\|_*, \qquad 
  |t|\le C\|h\|_*.
 \end{gathered}
\]
The same holds for the potential $\varepsilon^{-2}F'(u^{\rm app})$.
\end{lemma}
\begin{proof}
\emph{Projection and radial estimates.} Set $P=\|\phi\|_\infty$ and
$H=\|h\|_*$. Radial integration gives
\begin{equation}
 \label{mc:analytic-weight-integrals}
 \begin{aligned}
  \int_M\omega_\varepsilon\,dV_g&\le C, \\
  \int_{M\setminus B_r(0)}\omega_\varepsilon\,dV_g
 &\le C(\rho/r)^{\beta-2}\quad(C_S\delta\le r\le r_0).
 \end{aligned}
\end{equation}
Testing \eqref{mc:analytic-linear} with the rescaled
$\chi_j\partial_bU_{a_j}$ gives the Gram blocks
\[
 \int\chi(|z|/R_0)\partial_bU_{a_j}\partial_cU_{a_j}\,dz
=\mathfrak G(a_j)\delta_{bc},
\]
 with $\mathfrak G(a_j)\ge c>0$.
The differentiated cutoff and potential discrepancy have $L^1$
norms $O(\eta)$ and $O(\eta^{m-1})$, respectively; hence
\begin{equation}\label{mc:analytic-tpre}
 |t|\le C(H+\eta P).
\end{equation}

Let $R_j=r_c\delta/\rho_j$, $\phi_j(z)=\phi(q_j+\rho_jz)$,
and let $y_j$ be its circular average. On $|z|\le R_j$, \eqref{mc:analytic-linear} becomes
\[
 (\Delta_z+F'(U_{a_j}))\phi_j
 =\rho_j^2e^{\sigma(q_j+\rho_jz)}h(q_j+\rho_jz)
 +\sum_bt_{j,b}\chi(|z|/R_0)\partial_bU_{a_j}(z).
\]
Averaging in angle eliminates the projection terms and gives
$y_j''+r^{-1}y_j'+F'(U_{a_j})y_j=f_j$, where $f_j$ is the circular
average of the first term on the right. Separation implies
\[
 \rho_j^2\omega_\varepsilon(q_j+\rho_jz)
\le C\{(1+|z|)^{-\beta}+\eta^\beta\}
\le C(1+|z|)^{-\beta}
\]
 on this disk. Hence $|f_j(r)|\le CH(1+r)^{-\beta}$, and we claim

\begin{equation}
 \label{mc:analytic-radial}
 \begin{aligned}
  y_j&=A_jZ_{a_j}+T_j, \qquad 
  |T_j(r)|&\le CH\log(2+r), \\
  \int_0^{R_j}r|F'(U_{a_j})T_j|\,dr&\le CH.
 \end{aligned}
\end{equation}
Choose a second solution $Y_a$ of $Y_a''+r^{-1}Y_a'+F'(U_a)Y_a=0$,
smoothly in $a$, with $r(Z_aY_a'-Z_a'Y_a)=1$.
It grows at most logarithmically at zero. With $r=e^t$, both
solutions satisfy $X''+q_a(t)X=0$, where
$|q_a(t)|\le Ce^{-(a_*-2)t}$ and $a_*>4$.
The integral equation and Gronwall's inequality give
\[
 \begin{gathered}
  |Z_a(r)|+|Y_a(r)|\le C(1+\log r), \\
  r|Z_a'(r)|+r|Y_a'(r)|\le C\quad(r\ge1).
 \end{gathered}
\]
For a radial source $f$, the regular particular solution is
\[
 T(r)=\int_0^r t\{Z_a(t)Y_a(r)-Y_a(t)Z_a(r)\}f(t)\,dt.
\]
These bounds yield
\begin{gather*}
 \frac{|T(r)|}{\log(2+r)}+|rT'(r)|
 +\int_0^R t|F'(U_a(t))T(t)|\,dt\le C\mathfrak J_R(f),\qquad r\le R,\\
 \mathfrak J_R(f):=\int_0^R
 t\{1+|\log t|+\log(2+t)\}|f(t)|\,dt.
\end{gather*}
Since $\beta>2$, $|f(t)|\le CH(1+t)^{-\beta}$ implies
$\mathfrak J_R(f)\le CH$. This proves \eqref{mc:analytic-radial}
and the flux bound without dividing by $Z_a$.

At $R_j/2$, \eqref{mc:analytic-radial} gives
$|A_j|\le C(P/L+H)$.
Integrating \eqref{mc:analytic-linear} over $M$ and using
$\int_MK_{j,b}\,dV_g=0$ gives
\begin{equation}\label{mc:analytic-radial-sum}
 \left|\sum_jA_j\right|
 \le C(H+\eta^{m-2}P)+C\eta^{m-2}L\sum_j|A_j|.
\end{equation}
Here 
\[
 \int_0^{R_j}rF'(U_{a_j})Z_{a_j}\,dr
=1+O(\eta^{m-2}L),
\]
 and the cutoff annuli contribute
$O(\eta^{m-2}P)$.

\emph{The global estimate.} If $P\le CLH$ fails, normalize a sequence
so that $P=1$ and $LH\to0$. Then \eqref{mc:analytic-tpre} and
\eqref{mc:analytic-radial-sum} give $t\to0$, $A_j=O(L^{-1})$
and $L\sum_jA_j\to0$.
At cluster scale $x=\delta y$, the potential and forcing away
from the limiting cores are $O(\eta^{m-2})$ and
$O(H\eta^{\beta-2})$. The bounded harmonic limit on the punctured
plane extends to a constant $C_*$.
On a fixed cluster-scale circle about each core,
$-A_j\log(R_j/2)\to C_*$.
Since $\log(R_j/2)=L/3+O(\log L)$, summation gives $SC_*=0$,
then $A_jL\to0$ and $\sup_{r\le R_j}|y_j(r)|\to0$.
The constraints and positive Gram blocks eliminate the translation
limits at core scale. For $C_S\delta\le r\le r_0$, all potentials
and projections lie in $B_r(0)$. Integrating \eqref{mc:analytic-linear}
on its complement gives, for the circular average $\overline\phi_0$
about $p=0$,
\[
 \begin{aligned}
  2\pi r\overline\phi_0'(r)
 &=-\int_{M\setminus B_r(0)}h\,dV_g, \\
  |\overline\phi_0(r)-\overline\phi_0(C_S\delta)|
 &\le CH\eta^{\beta-2}.
 \end{aligned}
\]
Thus the exterior averages vanish. At an intermediate distance
$\rho\ll\ell\ll\delta$ from a core, or
$\delta\ll\ell\ll1$ from the cluster, its rescaling has forcing
$O(H(\rho/\ell)^{\beta-2})$ and potential
$O((\rho/\ell)^{m-2})$. The bounded harmonic limit extends across the puncture and is
constant; its vanishing radial or exterior average makes it zero.
At fixed distance from $p$, harmonicity on the closed surface also
gives a zero limit. All possible maximum locations thus contradict $P=1$. Hence $P\le CLH$, and
\eqref{mc:analytic-tpre} gives $|t|\le CH$.

The augmented operator
$H^2(M)\times\R^{2S}\to L^2(M)\times\R^{2S}$
is a compact perturbation of $(\phi,t)\mapsto(\Delta_g\phi,t)$,
of index zero. The estimate makes its kernel trivial; Fredholm theory and elliptic
regularity give the inverse.

\emph{The full potential.} Set
$W=\varepsilon^{-2}F'(u^{\rm app})-V_g^0$ and
$\kappa_\varepsilon=\|W\|_*$. By \eqref{mc:analytic-incident}, the inner contribution is $C\eta^2$.
On the cutoff annuli and exterior, the tails in
Lemma~\ref{mc:analytic-residual} give $C\eta^{m-\beta}$. Thus
\begin{equation}
 \label{mc:analytic-potential-perturbation}
 \begin{aligned}
  \kappa_\varepsilon&\le C(\eta^2+\eta^{m-\beta}), \quad
  \|W\phi\|_*&\le\kappa_\varepsilon P, \quad 
  L\kappa_\varepsilon&=o(1).
 \end{aligned}
\end{equation}
For the full potential, rewrite \eqref{mc:analytic-linear} with
source $h-W\phi$ and apply the cutoff inverse:
\begin{equation}\label{mc:analytic-full-inverse}
 P\le CL(H+\kappa_\varepsilon P),\qquad
 |t|\le C(H+\kappa_\varepsilon P).
\end{equation}
For $CL\kappa_\varepsilon<1/2$, absorption gives the estimates and
contraction gives existence and uniqueness. The radial and exterior-flux
estimates hold with $H+\kappa_\varepsilon P$ replacing $H$.
\end{proof}

\subsection{The nonlinear family and its differentiability}
\begin{proposition}\label{mc:projected-family}
For fixed sufficiently large $C_1$, there is a unique correction
$\|\phi\|_\infty\le C_1\eta^2L$ satisfying all $2S$ constraints
in \eqref{mc:analytic-linear}, such that $u=u^{\rm app}+\phi$ solves
\begin{equation}
 \label{mc:analytic-projected}
 \begin{aligned}
  \Delta_gu+\varepsilon^{-2}F(u)
 &=4\pi n\delta_p+4\pi\sum_\ell n_\ell\delta_{p_\ell}
       +\sum_{j,b}t_{j,b}K_{j,b}, \\
  |t|&\le C\eta^2.
 \end{aligned}
\end{equation}
For each fixed $\varepsilon$ this family is $C^1$ in the full position
neighborhood. It has $u-u_0$ smooth and $\sup_Mu\le-\gamma<0$.
On $D_j$, writing $V_j(z)=u(q_j+\rho_jz)-U_{a_j}(z)$ and
$\Gamma_j=\sigma(q_j+\rho_jz)-\sigma(q_j)$, one has
\begin{align}
 |V_j(z)|&\le C\eta^2(|z|^2+L),&\|V_j\|_{C^2(B_R)}&\le C_R\eta^2L,\notag\\
 |e^{\Gamma_j}F(U_{a_j}+V_j)-F(U_{a_j})|
     &\le C\eta^2(1+|z|)^{-m}(|z|^2+L).\label{mc:analytic-V}
\end{align}
\end{proposition}
\begin{proof}
The bound $|F''(u^{\rm app}+t)|\le Ce^{u^{\rm app}}$ for bounded $t$
gives a nonlinear remainder of weighted norm $C\|\phi\|_\infty^2$
and Lipschitz constant
$C(\|\phi\|_\infty+\|\psi\|_\infty)$.
Lemma~\ref{mc:inverse} gives a contraction on the stated ball with
factor $O(\eta^2L^2)=o(1)$. The approximation and interior elliptic
estimates give \eqref{mc:analytic-V} and strict negativity.
For fixed $\varepsilon$, the equation and constraints define a $C^1$
map from $C^{2,\alpha}(M)\times\R^{2S}$ to
$C^{0,\alpha}(M)\times\R^{2S}$, also in the position parameters.
Its linearized potential differs from
$\varepsilon^{-2}F'(u^{\rm app})$ by $O(\|\phi\|_\infty)$ in
weighted norm. Since $L\|\phi\|_\infty=o(1)$, absorption gives
invertibility. The implicit function theorem and uniqueness yield
a $C^1$ family.
\end{proof}

\subsection{Comparison of matched and nonlinear core masses}
Define the signed measure
\[
 d\nu=(2\pi)^{-1}\left\{\varepsilon^{-2}F(u)
                    -\sum_{j,b}t_{j,b}K_{j,b}\right\}dV_g.
\]
Integrating \eqref{mc:analytic-projected} gives $\nu(M)=mS$.
The exterior measure $\nu_{\rm out}=\nu|_{M\setminus\cup D_j}$ is
nonnegative and satisfies
\begin{equation}\label{mc:analytic-measure-tail}
 \nu_{\rm out}(M)=O(\eta^{m-2})=o(s^2).
\end{equation}
Write $\nu_j=\nu|_{D_j}$. Since each projection integrates to zero
on $D_j$, $\widetilde a_j=\nu(D_j)$ is the positive normalized
nonlinear mass.
\begin{lemma}\label{mc:actual-masses}
For each $j$,
\begin{equation}\label{mc:analytic-actual-mass}
 \widetilde a_j-a_j=O(\eta^4L^2+\rho)=o(s/L).
\end{equation}
Moreover,
\[
 \begin{aligned}
  \left|\int_{D_j}(x-q_j)\,d\nu\right|&\le C\rho\eta^2L, \qquad
  \int_{D_j}|x-q_j|^2\,d|\nu|&\le C\rho^2.
 \end{aligned}
\]
\end{lemma}
\begin{proof}
Put $R_j=r_c\delta/\rho_j$ and $A_\varepsilon=\eta^4L^2$.
Subtracting $\Delta U_{a_j}+F(U_{a_j})=0$ from the rescaled
\eqref{mc:analytic-projected} gives
\[
 \Delta V_j+F'(U_{a_j})V_j =-N_j+\sum_bt_{j,b}\chi(|z|/R_0)\partial_bU_{a_j},
\]
where
\[
 N_j=F(U_{a_j}+V_j)-F(U_{a_j})-F'(U_{a_j})V_j
 +(e^{\Gamma_j}-1)F(U_{a_j}+V_j).
\]
By \eqref{mc:analytic-V},
\begin{equation}\label{mc:analytic-mass-N}
 |N_j(z)|\le C(1+|z|)^{-m}
 \{\eta^4(|z|^2+L)^2+\rho s|z|+\rho^2|z|^2\}.
\end{equation}
On $|z|\le R_j$, $V_j$ is bounded. For $m>6$, choose the mass interval
above six, ensuring uniform fourth-moment integrability. For $m=6$, write $a_j=6+\mu_j$ and
$T_j^*=\log R_j\asymp L$. Since $|\mu_j|T_j^*=O(s)$, for each
fixed integer $k\ge0$, even if $a_j<6$, we have
\[
 \int_1^{R_j}r^{5-a_j}(\log r)^k\,dr
 =\int_0^{T_j^*}e^{-\mu_j t}t^k\,dt
 \le C(T_j^*)^{k+1}.
\]
Thus the fourth-order term in \eqref{mc:analytic-mass-N} has
logarithmic moment $O(\eta^4L^2)$, as do lower-order terms; also
$\rho s+\rho^2\le C\eta^4L^2$.

Consequently,
\[
 \int_0^{R_j}r(1+|\log r|+\log(2+r))|\overline N_j|\,dr
\le CA_\varepsilon.
\]
Averaging the equation for $V_j$ removes the projections; the
variation-of-parameters formula used in \eqref{mc:analytic-radial} gives
\begin{equation}
 \label{mc:analytic-mass-radial}
 \begin{aligned}
  &\overline V_j=A_jZ_{a_j}+T_j, \quad
  |T_j(r)|\le CA_\varepsilon\log(2+r), \\
  &\int_0^{R_j}r|F'(U_{a_j})T_j|\,dr+|rT_j'(r)|
 \le CA_\varepsilon.
 \end{aligned}
\end{equation}
The averaged equation gives the derivative bound. Harmonicity gives
$|\overline V_j(R_j)|\le C\eta^2L$; with
$|Z_{a_j}(R_j)|\ge cL$, this yields $|A_j|\le C\eta^2$.
The flux identity
\[
 \widetilde a_j-(2\pi)^{-1}\int_{|z|<R_j}F(U_{a_j})\,dz
 =-R_j\overline V_j'(R_j)
\]
and the profile tail $O(\eta^{m-2})$ then imply
\begin{equation}\label{mc:analytic-mass-b}
 b_j:=\widetilde a_j-a_j-A_j=O(A_\varepsilon),
 \qquad \sum_jA_j=O(A_\varepsilon).
\end{equation}
Summing and using $\nu(M)=\sum_ja_j$ and
\eqref{mc:analytic-measure-tail} gives the second relation.

Let $\nu_j^0$ be radial about $q_j$, with density
$(2\pi)^{-1}F(U_{a_j})$ on $|z|<R_j$. The density difference
in \eqref{mc:analytic-V} and the compact projections give
\begin{equation}
 \label{mc:analytic-mass-moments}
 \begin{aligned}
  \int_{D_j}|y-q_j|\,d|\nu_j-\nu_j^0|&\le C\rho\eta^2L, \\
  \int_{D_j}|y-q_j|\,d|\nu_j|&\le C\rho, \\
  \int_{D_j}|y-q_j|^2\,d|\nu_j|&\le C\rho^2.
 \end{aligned}
\end{equation}
The first integral converges for $m\ge6$; the second moment is finite
for $m>4$. Compact projections contribute $O(\rho|t_j|)$.
Since $\nu_j^0$ has zero first moment, these give the asserted bounds.
For $j\ne\ell$, radial averaging against $\nu_\ell^0$ and the
first bound in \eqref{mc:analytic-mass-moments} give
\begin{equation}\label{mc:analytic-crosslog}
 \int_{D_\ell}\log|q_j-y|\,d\nu_\ell
 =\widetilde a_\ell\log|q_j-q_\ell|+O(\eta^3L).
\end{equation}
The smooth Green terms contribute $O(\rho)$, and the exterior potential
contributes $O(\eta^{m-2}L)$ at each center.

Write the exact Green representation as
$u=C^{\rm act}+u_0+2\pi\int_MG(\cdot,y)\,d\nu(y)$.
Then $H_j(x)=u(x)+\int_{D_j}\log|x-y|\,d\nu_j(y)$ is harmonic on
$D_j$, since $\Delta_xu=-2\pi\,d\nu/dx$ there. The self-potential
has boundary mean $\widetilde a_j\log(r_c\delta)$. Thus
\eqref{mc:analytic-profile} and \eqref{mc:analytic-mass-radial} give
\[
 \begin{aligned}
 H_j(q_j)={}&a_j\log\rho_j+c_\infty(a_j)
 +A_j\{\log\rho_j+c_\infty'(a_j)\}\\
 &+b_j\log(r_c\delta)+T_j(R_j)
 +O\!\left(\eta^{m-2}(1+|A_j|L)\right).
 \end{aligned}
\]
Comparing with the Green representation at $q_j$, using
\eqref{mc:analytic-crosslog} and \eqref{mc:analytic-C}, yields
\begin{equation}
 \label{mc:analytic-actual-matrix}
 \begin{aligned}
  MA&=(C^{\rm act}-C)\boldsymbol1+\mathcal E, \\
  |\mathcal E|&\le C(A_\varepsilon L+\eta^3L+\rho+\eta^{m-2}L).
 \end{aligned}
\end{equation}
Here $M$ is the mass matrix in \eqref{mc:analytic-matrix}.

The terms involving $b_j$ and the exterior tail are bounded by $CA_\varepsilon L$.
Write $A=A_E+\bar A\boldsymbol1$, with
$\bar A=O(A_\varepsilon)$ by \eqref{mc:analytic-mass-b}.
Projecting onto $E$ removes the constant term. By
\eqref{mc:analytic-matrix},
\[
 |A_E|\le CL^{-1}(|\mathcal E|+|\bar A|) \le C(A_\varepsilon+\eta^3+\rho/L) \le C(A_\varepsilon+\rho).
\]
Together with \eqref{mc:analytic-mass-b}, this gives
\eqref{mc:analytic-actual-mass}, since
$A_\varepsilon\le Cs^2L^2$ and
$(s^2L^2+\rho)/(s/L)\le C(sL^3+s^{1/2}L)\to0$.
\end{proof}

\subsection{The natural constraint and the energy differential}
For $w=u-u_0$ and $\mathcal P'=F$, define
\begin{equation}\label{mc:analytic-energy}
 \mathcal J_\varepsilon(w)=\frac12\int_M|\nabla_gw|^2\,dV_g
 -\varepsilon^{-2}\int_M\mathcal P(u_0+w)\,dV_g
 +\frac{4\pi N}{|M|_g}\int_Mw\,dV_g.
\end{equation}
Integer vortex strengths make this functional smooth on the family.
Equation \eqref{mc:analytic-projected} gives
\begin{equation}\label{mc:analytic-energy-variation}
 \partial_\alpha\mathcal J_\varepsilon(w)
 =-\sum_{j,b}t_{j,b}\int_MK_{j,b}\partial_\alpha u\,dV_g.
\end{equation}
\begin{lemma}[Natural constraint]\label{mc:natural}
In unrestricted position coordinates,
\begin{align*}
 \mathcal B_{(j,b),(\ell,c)}
 &:=\int_MK_{j,b}\partial_{q_{\ell,c}}u\,dV_g,\\
 \mathcal B&=-\operatorname{diag}
       (\rho_j^{-1}\mathfrak G(a_j)I_2)+o(\rho^{-1}).
\end{align*}
Thus $\rho\mathcal B$ is uniformly invertible, and an energy
critical point on the full family has $t=0$.
\end{lemma}
\begin{proof}
On the support of $K_{j,b}$, $u^{\rm app}=I_j$.
Translation contributes $-\rho_j^{-1}\partial_cU_{a_j}$;
the radial mass, scale and constant incident derivatives pair to zero.
The remaining incident variation is $O(\rho/s^2)$, using
$|\partial_q a_j|\le C/(sL)$.

Set $\kappa_{j,b}=e^\sigma K_{j,b}$. Differentiating the
Euclidean orthogonality gives
\[
 \int\kappa_{j,b}\partial_q\phi
=-\int(\partial_q\kappa_{j,b})\phi.
\]
The compact projection satisfies
\[
 \|\partial_q\kappa_{j,b}\|_{L^1}
\le C\{\rho^{-1}+(sL)^{-1}+s\},
\]
 so multiplication by
$\|\phi\|_\infty=O(\eta^2L)$ gives $o(\rho^{-1})$.
Since $\rho^2/s^2=O(\eta^2)\to0$, the incident error is also
$o(\rho^{-1})$. The Gram blocks and \eqref{mc:analytic-energy-variation} prove the claim.
\end{proof}

Define
\[
 \begin{aligned}
  (Q_j)_{ab}&=\int\chi(|z|/R_0)\partial_bU_{a_j}
                         \partial_a(U_{a_j}+V_j)\,dz, \\
  P_j&=-\frac{Q_jt_j}{2\pi\widetilde a_j\rho_j}.
 \end{aligned}
\]
Since $Q_j=\mathfrak G(a_j)I_2+O(\eta^2L)$ is uniformly invertible,
the normalized forces $P_j$ vanish exactly when $t=0$.
\begin{lemma}[Scaled energy differential]\label{mc:scaled-energy}
Uniformly on compact scaled parameter sets
$\alpha=(d,\theta,\xi)$,
\begin{equation}\label{mc:scaled-energy-formula}
 \partial_\alpha\mathcal J_\varepsilon(w)
 =-2\pi\sum_j\widetilde a_jP_j\cdot\partial_\alpha q_j
 +O(\eta^2\rho+s\eta^4+s^2\eta^4L)
 =-2\pi\sum_j\widetilde a_jP_j\cdot\partial_\alpha q_j+o(s^2).
\end{equation}
No uniform derivative bound on $P$ or $\phi$ is required.
\end{lemma}
\begin{proof}
Matching and conformal normality give
$|\partial_\alpha a_j|\le Cs/L$ and
$|\partial_\alpha\rho_j|/\rho_j\le Cs^2$.
Moreover,
\begin{equation}\label{mc:analytic-incident-derivative}
 |\partial_\alpha[\nabla B_j(q_j)]|\le C.
\end{equation}
The singular $s^{-1}$ term vanishes on scaled, rotated polygons;
the next term and its first derivatives are bounded.
Mass terms contribute $O(|\mu|/s+|\partial_\alpha\mu|/s)=O(L^{-1})$,
and smooth Green derivatives are bounded.

Set
\[
 \begin{gathered}
  k_b(a,z)=\chi(|z|/R_0)\partial_bU_a(z), \\
  \kappa_{j,b}=e^\sigma K_{j,b} =\rho_j^{-2}k_b\left(a_j,\frac{x-q_j}{\rho_j}\right).
 \end{gathered}
\]
For the derivative along the core,
\[
 D_{j,\alpha}=\partial_\alpha+(\partial_\alpha q_j)\cdot\nabla_x,
\]
direct differentiation gives
\[
 \begin{aligned}
  D_{j,\alpha}\kappa_{j,b}
 &=\rho_j^{-2}\left\{
 (\partial_\alpha a_j)\partial_a k_b
 -(\partial_\alpha\log\rho_j)(2k_b+z\cdot\nabla_z k_b)
 \right\}, \\
  z&=\frac{x-q_j}{\rho_j}.
 \end{aligned}
\]
Compact support and the parameter bounds imply
$\|D_{j,\alpha}\kappa_{j,b}\|_{L^1}\le C(s/L+s^2)$.
Differentiate $\int\kappa_{j,b}\phi\,dx=0$ and integrate the
translation term by parts to obtain
\[
 \int\kappa_{j,b}D_{j,\alpha}\phi\,dx
 =-\int(D_{j,\alpha}\kappa_{j,b})\phi\,dx
 =O(s\eta^2+s^2\eta^2L).
\]

The derivative $D_{j,\alpha}U_{a_j}((x-q_j)/\rho_j)$ is radial
and pairs to zero with $\kappa_{j,b}$. The function
\[
 H_{j,\alpha}(x) :=D_{j,\alpha}\{B_j(x)-B_j(q_j)\}
\]
is harmonic on the core disk, vanishes at $q_j$, and satisfies
\[
 \nabla_xH_{j,\alpha}(q_j)
=\partial_\alpha[\nabla B_j(q_j)].
\]
Angular integration eliminates every harmonic term except the linear term:
\[
 \begin{aligned}
  \int\kappa_{j,b}H_{j,\alpha}\,dx
 &=\rho_j c_1(a_j)
       \partial_\alpha[\partial_bB_j(q_j)], \\
  c_1(a)&=\pi\int_0^{2R_0}\chi(r/R_0)U_a'(r)r^2\,dr.
 \end{aligned}
\]
Thus the incident pairing is $O(\rho)$ and 
\[
 \begin{gathered}
  \int_MK_{j,b}D_{j,\alpha}u\,dV_g
=O(\rho+s\eta^2+s^2\eta^2L).
 \end{gathered}
\]
Substitute 
\[
 \partial_\alpha u=D_{j,\alpha}u-
(\partial_\alpha q_j)\cdot\nabla_xu
\]
 in
\eqref{mc:analytic-energy-variation}. The spatial pairing is
$\rho_j^{-1}Q_jt_j=-2\pi\widetilde a_jP_j$; using $|t|\le C\eta^2$
bounds the remainder by
$C(\eta^2\rho+s\eta^4+s^2\eta^4L)=o(s^2)$ in both regimes.
\end{proof}

\section{Moment estimates and reduced energies}\label{mc:force-section}
Second angular moments and stress identities determine the reduced
energy; at $m=6$, the moment coefficient is logarithmic.

Fix $S\ge3$ and use \eqref{mc:analytic-parameters}, with
$e_j=e^{i\theta}z_j$, $z_j=\omega^j$, and $\omega=e^{2\pi i/S}$.
Write $K=K_g(p)$, $h=\nabla\mathcal H(0)$, and $B=D^2\mathcal H(0)$;
then $D^2\sigma(0)=-KI$ and $\operatorname{tr}B=0$.
All estimates are uniform on compact parameter sets.

Write $m_j=a_j$, $M_j=\widetilde a_j$, and $\mu_j=M_j-m$.
For $m>6$, admissibility gives $m\ge8$, and
$\eta^2\asymp s\asymp\delta$, $\rho\asymp\delta^{3/2}$, and
$\rho_j/\rho=1+O(\delta^2)$. Matching and
Lemma~\ref{mc:actual-masses} give
\begin{gather*}
 m_j=m+\mu_j^0,\quad \sum_j\mu_j^0=0,\quad
 |\mu_j^0|\le C\delta/L,\quad M_j-m_j=O(\delta^2L^2+\rho),\\
 \mu_j:=M_j-m=O(\delta/L),\qquad
 \tau:=\sum_jM_j-Sm=O(\eta^{m-2}).
\end{gather*}
By \eqref{mc:analytic-V}, the corrected family retains the bounds
$O(\eta^{m-2})$, $O(\eta^{m-2}/\delta^2)$, and
$O(\eta^{m-2}/\delta)$ for the exterior mass, matching-circle pressure,
and gradient of the removed potential, respectively.

Define the barycenters and complex moments of $\nu$ by
$\zeta_j=M_j^{-1}\int_{D_j}x\,d\nu$ and
$M_{j,k}=\int_{D_j}(x-\zeta_j)^k\,d\nu$.
Lemma~\ref{mc:actual-masses} and $M_j\ge c>0$ give
\begin{equation}\label{mc:barycenter}
 \zeta_j=q_j+O(\rho\eta^2L)
       =sd e_j+s^2\xi_j+o(s^2).
\end{equation}
Set $C_S=m(S^2-1)/24$. For $m>6$, write $b=b(m)>0$;
\eqref{mc:Phi}--\eqref{mc:b} hold uniformly near $m$.

\subsection{The logarithmic coefficient at mass six}\label{cc:force}
For $m=6$, admissibility gives odd $S\ge3$, $n=3(S-1)/2$,
and $N=3S$. Put $A_6=e^{c_\infty(6)}>0$ and use
$s=\varepsilon^{2/3}L^{1/6}$, $\delta=sd$, and $\eta=\varepsilon/\delta$.
The logarithmic second moment accounts for the extra factor in $s$:

\begin{equation}
 \label{cc:scales}
 \begin{aligned}
  \delta^2&\asymp\eta^4L, \quad 
  \eta^2L\asymp s\sqrt L\to0, \\
  \log(\delta/\varepsilon)&=\tfrac13L+O(\log L).
 \end{aligned}
\end{equation}

Take $\beta=3$ in Proposition~\ref{mc:projected-family}. Its bounds
on $V_j,t_j$ hold, and
\begin{equation}
 \label{cc:mass-input}
 \begin{aligned}
  m_j&=6+O(\delta/L), \quad
  \sum_jm_j=6S, \\
  M_j-6&=O(\delta/L), \quad 
  \tau:=\sum_jM_j-6S=O(\eta^4).
 \end{aligned}
\end{equation}
The exterior mass, pressure and potential-gradient bounds are
$O(\eta^4)$, $O(\eta^4/\delta^2)$ and $O(\eta^4/\delta)$.

For $\widetilde m$ near six, comparison gives a
positive regular second mode with $\Phi_{\widetilde m}(r)/r^2\to1$.
Writing $W_{\widetilde m}=F'(U_{\widetilde m})$, variation of
parameters yields
\begin{equation}\label{cc:Volterra}
 \Phi_{\widetilde m}(r)=
 r^2\left(1+\frac14\int_r^\infty
 t^{-1}W_{\widetilde m}(t)\Phi_{\widetilde m}(t)\,dt\right)
 +\frac{r^{-2}}4\int_0^r
 t^3W_{\widetilde m}(t)\Phi_{\widetilde m}(t)\,dt.
\end{equation}
The first integral converges uniformly. For $R\asymp\eta^{-1}$ and
$|\widetilde m-6|\le C\delta/L$, we have
$|\widetilde m-6|\log R=O(\delta)$; uniform ODE bounds give
\begin{equation}
 \label{cc:Phi-uniform}
 \begin{aligned}
  \Phi_{\widetilde m}(r)&=r^2+O(r^{-2}\log(2+r)), \\
  \Phi_{\widetilde m}(r)&\asymp r^2
 \quad(1\le r\le R).
 \end{aligned}
\end{equation}
Indeed, the two corrections in \eqref{cc:Volterra} are respectively
$O(r^{4-\widetilde m})=O(r^{-2})$ and
$O(r^{-2}\log(2+r))$ on this interval.
Since
\[
 W_{\widetilde m}(r)=A_{\widetilde m}r^{-\widetilde m}
+O(r^{2-2\widetilde m}),
\]
 with
$A_{\widetilde m}=e^{c_\infty(\widetilde m)}$, we obtain
\begin{equation}\label{cc:truncated-I}
 I_j(R):=\int_0^Rr^3F'(U_{m_j})\Phi_{m_j}\,dr
       =A_6\log R+o(L).
\end{equation}
The leading term is
$A_{m_j}\log R\{1+O(|m_j-6|\log R)\}$, and the remainder is
uniformly integrable, with total error
$O(1)+o(1)\log R+O(\delta\log R)=o(L)$.
At mass six,
\[
 \Phi_6=r^2+\frac{A_6}{4}r^{-2}\log r +b_{\rm ren}r^{-2}+O(r^{-6}\log r).
\]
We use only \eqref{cc:truncated-I}, without assuming a finite $b(6)$
or a sign for $b_{\rm ren}$.

\subsection{Core moments}
The following lemma concerns the projected family with $m>6$.
\begin{lemma}\label{mc:moments}
For the projected family,
\begin{align}
 M_{j,1}&=0,\qquad
 M_{j,2}=-2bC_S\frac{\varepsilon^4}{\overline{sd e_j}^{\,2}}
               +o(\varepsilon^4/\delta^2),\label{mc:second}\\
 &\sum_{k\ge3}\frac{|M_{j,k}|}{(a_S\delta)^{k+1}}
 \le C\eta^5L/\delta.\label{mc:higher}
\end{align}
Choose the separation constant $a_S>0$ and $r_c$ so that each support
centered at $\zeta_j$ has radius below $a_S\delta/4$.
The second-moment error is $O(\rho^2\eta^4L^2+\rho^4)$, plus
$o(\rho^4/\delta^2)$ from replacing $b(m_j)$ by $b$.
\end{lemma}

For $m=6$, the bounded integral $\int_1^Rr^4r^{-m_j}dr$ gives
\begin{equation}\label{cc:barycenter}
 \zeta_j-q_j=O(\varepsilon\eta^2L)=o(s^2).
\end{equation}
For $m=6$, $C_S=(S^2-1)/4$; put $r_j=\delta e^{i\theta}z_j$.

\begin{lemma}\label{cc:moments}
Uniformly for the critical projected family,
\begin{align}
 M_{j,2}
 &=-\frac{A_6C_S}{2}\frac{\varepsilon^4}{\overline{r_j}^{\,2}}
       \log(\delta/\varepsilon)
       +o(\varepsilon^4L/\delta^2),\label{cc:M2}\\
 \sum_{k\ge3}\frac{|M_{j,k}|}{(a_S\delta)^{k+1}}
 &\le C\left(\frac{\eta^4}{\delta}
                 +\frac{\eta^5L}{\delta}\right)=o(\delta).
 \label{cc:higher}
\end{align}
The error before replacing the incident and profile coefficients is
$O(\varepsilon^2(\eta^2+\eta^4L^2))$.
\end{lemma}
\begin{proof}[Proof of Lemmas~\ref{mc:moments} and \ref{cc:moments}]
Write $r_j=sd e_j$ and $C_S=m(S^2-1)/24$ in both cases.
At the exact polygon, the analytic incident coefficients are
$a_{j,2}=-C_Sr_j^{-2}$ and $a_{j,3}=C_Sr_j^{-3}$, since
$2n=m(S-1)/2$ and
\[
 \begin{aligned}
  \sum_{k=1}^{S-1}(1-\omega^k)^{-2}&=\frac{(S-1)(5-S)}{12}, \\
  \sum_{k=1}^{S-1}(1-\omega^k)^{-3}&=\frac{(S-1)(3-S)}8.
 \end{aligned}
\]
Perturbations give $a_{j,2}=-C_Sr_j^{-2}+O(\delta^{-1})+O(1)$
and $a_{j,3}=C_Sr_j^{-3}+O(\delta^{-2})$.
The change in $\rho_j^2a_{j,2}$ is $O(\eta^2\delta)$.

Set 
\[
 \mathcal F_2f=\pi^{-1}\int_0^{2\pi}
f(r,\vartheta)e^{2i\vartheta}\,d\vartheta,
\]
 so
$\mathcal F_2\operatorname{Re}(az^2)=\bar a r^2$.
The coefficient $V_{j,2}=\mathcal F_2V_j$ satisfies
\[
 V_{j,2}''+r^{-1}V_{j,2}'-4r^{-2}V_{j,2}
+F'(U_{m_j})V_{j,2}=-N_{j,2},
\]
 where $N_{j,2}$ contains the
nonlinear and conformal remainders; translations have only first modes.
Taylor expansion using \eqref{mc:analytic-V} gives, for
$R_j=r_c\delta/\rho_j\asymp\eta^{-1}$,
\[
 |N_{j,2}|\le C\eta^4
 \begin{cases}
 L^2r^2,&r\le1,\\
 r^{-4}+L^2r^{-8},&1\le r\le R_j,\quad m>6,\\
 r^{-2}+L^2r^{-6},&1\le r\le R_j,\quad m=6.
 \end{cases}
\]
Regularity gives second-order vanishing at zero. The linear conformal
term has no second mode; the remaining terms satisfy these bounds by
\[
 |\sigma(q_j+\rho_jz)-\sigma(q_j)|
\le C(\rho\delta|z|+\rho^2|z|^2).
\]
At mass six, use $\rho\delta=O(\eta^5L)$ and $\rho^2=O(\eta^6L)$.
For matched masses near eight, $r^{|m_j-8|}\le e^{C\delta}$ on
$1\le r\le c/\eta$; near six, $r^{-m_j}\le Cr^{-6}$ on this interval.

Write $V_{j,2}=c_j\Phi_{m_j}+p_j$, $p_j(R_j)=0$.
Matching and $\|\phi\|_\infty=O(\eta^2L)$ give
$c_j=\rho_j^2\overline{a_{j,2}}+O(\eta^4L)$:
use $\Phi_{m_j}(R_j)/R_j^2=1+O(\eta^4)$ for $m>6$ and
\eqref{cc:Phi-uniform} for $m=6$.
Variation of parameters gives
\[
 p_j(r)=\Phi_{m_j}(r)\int_r^{R_j}\frac1{t\Phi_{m_j}(t)^2}
       \left(\int_0^t\tau\Phi_{m_j}(\tau)N_{j,2}(\tau)\,d\tau\right)dt.
\]
Since $\Phi_{m_j}\asymp r^2$, the outer kernel is $O(t^{-5})$.
For $m>6$ this gives $|p_j|\le C\eta^4L^2$ on $r\le1$ and
$|p_j(r)|\le C\eta^4r^{-2}(L^2+\log(2+r))$ on $r\ge1$, whence
\[
 \int_0^{R_j}r^3(|F'(U_{m_j})p_j|+|N_{j,2}|)\,dr\le C\eta^4L^2.
\]
For $m=6$, the inner integral is $O(\eta^4(t^2+L^2))$, giving
$|p_j|\le C\eta^4L^2$ on $r\le1$ and
$|p_j(r)|\le C\eta^4(1+L^2r^{-2})$ on $r\ge1$. Hence
\begin{equation}\label{cc:mode-error}
 \int_0^{R_j}r^3(|F'(U_{m_j})p_j|+|N_{j,2}|)\,dr
 \le C(\eta^2+\eta^4L^2)=o(\eta^2L).
\end{equation}

The second moment about $q_j$ is therefore
\[
 \frac{\rho_j^2}{2}\int_0^{R_j}r^3
 (F'(U_{m_j})V_{j,2}+N_{j,2})\,dr
 =\begin{cases}
 2b(m_j)\rho_j^4\overline{a_{j,2}}+O(\rho^2\eta^4L^2),&m>6,\\
 \rho_j^2c_jI_j(R_j)/2+O(\rho^2(\eta^2+\eta^4L^2)),&m=6.
 \end{cases}
\]
For $m>6$, the omitted profile tail is $O(\eta^2)$; for $m=6$,
use \eqref{cc:truncated-I} and
$\log R_j=\log(\delta/\varepsilon)+O(1)$.
Recentering subtracts
$M_j(\zeta_j-q_j)^2=O(\rho^2\eta^4L^2)$.
For masses above six, perturbing the incident coefficient introduces an error of size
$\rho^4/\delta=O(\rho^2\eta^4)$, while replacing $\rho_j$ by
$\varepsilon$ introduces an error $O(\rho^4)$. These estimates give \eqref{mc:second} and \eqref{cc:M2}.

For higher moments, the third absolute moment of $\nu_j-\nu_j^0$
about $q_j$ is bounded by
\[
 C\rho^3\eta^2\int_0^{R_j}r^4(r^2+L)(1+r)^{-m}\,dr
 \le\begin{cases}
 C\rho^3\eta^2L,&m>6,\\
 C\rho^3(\eta+\eta^2L),&m=6.
 \end{cases}
\]
For $m>6$, this uses
$\int_1^{R_j}(r^6+Lr^4)r^{-m_j}\,dr\le CL$.
Projection and recentering errors satisfy the same bounds.
Geometric summation gives \eqref{mc:higher} and \eqref{cc:higher};
in the critical case, division by $\delta$ gives
$O(L^{-1}+\eta)=o(1)$.
\end{proof}

\subsection{Individual forces and weighted identities}
Put $J_0(x)=x/|x|^2$. In the following proposition, assume $m>6$ and set
$Q_S=b(m)m(S^2-1)^2/144>0$.
\begin{proposition}\label{mc:individual}
Uniformly on compact parameter sets,
\begin{align}
 P_j={}&2nJ_0(\zeta_j)-\sum_{k\ne j}M_kJ_0(\zeta_j-\zeta_k)
 +h+\left(B-\frac{mK}{4}I\right)\zeta_j\nonumber\\
 &-Q_S\varepsilon^4\frac{\zeta_j}{|\zeta_j|^6}+E_j,
 \qquad \max_j|E_j|=o(s).\label{mc:individual-force}
\end{align}
\end{proposition}
\begin{proof}
Multiplying \eqref{mc:analytic-projected} by $e^\sigma\nabla u$ and
integrating with respect to $dx$ over $\Omega_j=B_{2r_c\delta}(q_j)$
gives, with outward normal $\mathbf n$,
\begin{align*}
 \int_{\partial\Omega_j}\!
 \left\{\partial_{\mathbf n}u\,\nabla u
 -\tfrac12|\nabla u|^2\mathbf n
 +\varepsilon^{-2}e^\sigma\mathcal P(u)\mathbf n\right\}ds
 =\rho_j^{-1}Q_jt_j
 +\varepsilon^{-2}\int_{\Omega_j}e^\sigma
                    \mathcal P(u)\nabla\sigma\,dx.
\end{align*}
Here $\Omega_j$ contains no vortex or other projection support.
Keep the moments on $D_j=B_{r_c\delta}(q_j)$ and subtract the
density potential outside $\bigcup_kD_k$, together with the regular
Green term from $\tau$. Near $\partial\Omega_j$, the exterior density
is $O(\eta^{m-2}/\delta^2)$. Splitting its gradient integral at
distance $\delta$ and using the exterior mass bound gives
$O(\eta^{m-2}/\delta)$. The remaining field is harmonic in the
stress annulus.

With $z=x-\zeta_j$, write this field as
\[
 v=-M_j\log|z|+\operatorname{Re}\sum_{k\ge1}
(a_{j,k}z^k+b_{j,k}z^{-k}),
\]
 where
$b_{j,k}=M_{j,k}/k$ and $b_{j,1}=0$.
The $z^{-1}$ coefficient of $(2\partial_zv)^2$ is
$-2M_ja_{j,1}-2\sum_{k\ge1}k(k+1)b_{j,k}a_{j,k+1}$.
The Dirichlet stress is $\pi$ times its conjugate, so the normalized
force is the conjugate of
$a_{j,1}+M_j^{-1}\sum_{k\ge1}k(k+1)b_{j,k}a_{j,k+1}.$
Estimate \eqref{mc:higher} controls higher modes. At the polygon,
the quadrupole contribution is
\[
 -2\sum_{k\ne j}\frac{b_{k,2}}{(sd e_j-sd e_k)^3}
                  +\frac6m b_{j,2}a_{j,3}.
\]
Using
\[
 \sum_{k=1}^{S-1}\omega^{2k}(1-\omega^k)^{-3}
=C_S/m
\]
and Lemma~\ref{mc:moments}, its coefficient is
$2bC_S^2/m-6bC_S^2/m=-4bC_S^2/m=-Q_S$.
Its conjugate is therefore
$-Q_S\varepsilon^4sd e_j/|sd e_j|^6+o(s)$.
Replacing $sd e_j$ here by $\zeta_j$ introduces an error $O(s^2)$.

For the smooth field,
$\sum_jM_j\zeta_j=O(\delta^2)$ by
$\sum e_j=0$, $\mu_j=O(\delta/L)$, and \eqref{mc:barycenter}.
Thus the smooth field equals $h+B\zeta_j+O(\delta^2+\rho^2+|\tau|)$.
The normalized pressure contribution is
\[
 \frac1{2\pi M_j}\int_{D_j}\varepsilon^{-2}e^\sigma
       \mathcal P(u)\nabla\sigma\,dx
 =\frac m4\nabla\sigma(\zeta_j)+o(s),
\]
by \eqref{mc:pressure}, yielding $-mK\zeta_j/4$.
Apart from the $o(s)$ term due to $b(m_j)-b(m)$, the error is bounded by
\[
 C\left[\delta^2+\frac{\eta^{m-2}}{\delta}
 +\frac{\eta^5L}{\delta}
 +\frac{\rho^2\eta^4L^2+\rho^4}{\delta^3}
 +\delta\eta^2L+\rho
 +\delta\max_j(|\mu_j|+|\mu_j^0|)\right].
\]
Dividing by $s\asymp\delta$ gives respectively
$O(\delta)$, $O(\eta^{m-6})$, $O(\delta^{1/2}L)$,
$O(\delta L^2+\delta^2)$, $O(\delta L)$,
$O(\delta^{1/2})$ and $o(1)$.
All tend to zero; the projection stress $\rho_j^{-1}Q_jt_j$
gives the normalization in \eqref{mc:individual-force}.
\end{proof}

Set $\widehat Q_S=A_6\,6(S^2-1)^2/1728>0$ and $\lambda=-6K/4$.
In Proposition~\ref{mc:individual}, replacing $b$ by
$(A_6/4)\log(\delta/\varepsilon)$ and using
$\log(\delta/\varepsilon)/L\to1/3$ gives
\begin{align}
 P_j={}&2nJ_0(\zeta_j)-\sum_{k\ne j}M_kJ_0(\zeta_j-\zeta_k)
       +h+\left(B-\frac{6K}{4}I\right)\zeta_j\nonumber\\
 &-\widehat Q_S\varepsilon^4L\frac{\zeta_j}{|\zeta_j|^6}
       +o(s).\label{cc:force-expansion}
\end{align}
Using exterior mass $O(\eta^4)$, \eqref{cc:higher}, and second-moment
error $O(\rho^2(\eta^2+\eta^4L^2))$, the stress remainder divided by
$\delta$ is bounded by
\[
 C\{\delta+L^{-1}+\eta+\eta^2L +\max_j(|M_j-6|+|m_j-6|)\}+o(1)=o(1).
\]
The coefficient error is also
$o(\varepsilon^4L/\delta^5)=o(\delta)$, proving
\eqref{cc:force-expansion}.

For 
\[
 F_j^{\rm pt}:=2nJ_0(\zeta_j)-\sum_{k\ne j}M_kJ_0(\zeta_j-\zeta_k),
\]
pairwise symmetrization gives
\begin{align}
 \sum_jM_j\zeta_j\cdot F_j^{\rm pt}
 &=2n\sum_jM_j-\tfrac12\left[(\sum_jM_j)^2-\sum_jM_j^2\right]\nonumber\\
 &=-2n\tau+\tfrac12\sum_j\mu_j^2-\tfrac12\tau^2,\label{mc:Euler}\\
 \sum_jM_j\zeta_j^\perp\cdot F_j^{\rm pt}&=0.\label{mc:torque}
\end{align}
Thus mass differences contribute
$O(\sum_j\mu_j^2/\delta)=O(s/L^2)=o(s)$ in dilation,
while the total-mass defect contributes $O(\eta^{m-2}/\delta)=o(s)$.

\subsection{The reduced energy on a curved surface}
For $m>6$, use \eqref{mc:individual-force}. The case of mass six
is verified below with $Q_S=\widehat Q_S$.
In the fixed frame $z_j=\omega^j$, write
$q_j=se^{i\theta}(dz_j+s\xi_j)$ with
\[
 \xi\in\mathcal E:=\{\xi:\sum_j\overline{z_j}\xi_j=0\}.
\]
Let $\mathsf L$ be the equal-mass point-force Jacobian at $(z_j)$,
restricted to $\mathcal E$; it is real symmetric and invertible
by Lemma~\ref{mc:polygon}. Write $\mathscr L=\mathsf L$ and $E=\mathcal E$.
Put $h_\theta=e^{-i\theta}h$ and $\lambda=-mK/4$.
Taylor expansion of \eqref{mc:individual-force} gives
\begin{equation}\label{mc:offshape-force}
 P_j=e^{i\theta}\{d^{-2}(\mathsf L\xi)_j+h_\theta\}+o(1),
 \qquad |P_j|\le C.
\end{equation}
The mass and barycenter errors are $O(L^{-1})$ and $O(s^{1/2}L)$;
all remaining terms are $o(1)$.
Equations \eqref{mc:Euler}--\eqref{mc:torque} give
\begin{align}
 \sum_jM_j\zeta_j\cdot P_j
 &=ms^2\left(h\cdot e^{i\theta}\sum_j\xi_j
       +S\lambda d^2-\frac{SQ_S}{d^4}\right)+o(s^2),
 \label{mc:offshape-virial}\\
 \sum_jM_j\zeta_j^\perp\cdot P_j
 &=ms^2h\cdot i e^{i\theta}\sum_j\xi_j+o(s^2).
 \label{mc:offshape-torque}
\end{align}
The point-force terms are $o(s^2)$ since $\tau=o(s^2)$ and
$\sum\mu_j^2=O(s^2/L^2)$. Moreover,
\[
 \sum_jM_j\zeta_j=ms^2e^{i\theta}\sum_j\xi_j+o(s^2)
\]
gives the $h$ term, and the trace-free Hessian averages to zero
on the polygon.
Replacing $\zeta_j$ by $q_j$ in these sums introduces an error
$O(s^{5/2}L)=o(s^2)$ by \eqref{mc:barycenter}
and \eqref{mc:offshape-force}.

Define 
\[
 \Omega_\varepsilon=(ms^2)^{-1}\sum_jM_jP_j\cdot dq_j.
\]
Uniformly on compact parameter sets,
\begin{equation}\label{mc:oneform-limit}
 \Omega_\varepsilon=d\mathcal J_0+o(1),
\end{equation}
where, with the real Euclidean inner product,
\begin{equation}\label{mc:limiting-energy}
 \mathcal J_0(d,\theta,\xi)
 =\frac{\langle\xi,\mathsf L\xi\rangle}{2d^2}
 +h_\theta\cdot\sum_j\xi_j
 +\frac{S\lambda d^2}{2}+\frac{SQ_S}{4d^4}.
\end{equation}
To verify \eqref{mc:oneform-limit}, use
\[
 \begin{gathered}
  \partial_\xi q=s^2e^{i\theta}\partial\xi, \quad 
  \partial_\theta q=iq, \quad 
  \partial_dq=d^{-1}q-d^{-1}s^2e^{i\theta}\xi.
 \end{gathered}
\]
Equations \eqref{mc:offshape-force} and \eqref{mc:offshape-torque}
give the $\xi$ and $\theta$ components.
The $d$ component of the virial term is
$d^{-1}(h_\theta\cdot\sum_j\xi_j+S\lambda d^2-SQ_Sd^{-4})$.
The last term in $\partial_dq$ cancels its $h_\theta$ contribution
and adds $-d^{-3}\langle\xi,\mathsf L\xi\rangle$,
giving $\partial_d\mathcal J_0$.

For
$\mathcal I_\varepsilon=\mathcal J_\varepsilon(w)$,
Lemma~\ref{mc:scaled-energy} gives
\[
 d\mathcal I_\varepsilon=-2\pi\sum_jM_jP_j\cdot dq_j+o(s^2).
\]
Integrating along uniformly bounded paths and fixing an additive
constant gives
\[
 -\frac{\mathcal I_\varepsilon}{2\pi ms^2}-c_\varepsilon \longrightarrow\mathcal J_0\quad\text{in }C^1.
\]

For $m=6$, \eqref{mc:Euler}--\eqref{mc:torque} have
$\tau=O(\eta^4)=O(s^2/L)=o(s^2)$ and squared
mass errors $O(s^2/L^2)$. Replacing $\zeta_j$ by $q_j$ in the weighted sums introduces an error
$O(\rho\eta^2L/s^2)=O(\eta\sqrt L)=o(1)$ after normalization.
The normalized remainder in \eqref{mc:scaled-energy-formula} is
$O(\eta/\sqrt L+\eta^2/\sqrt L+\eta^4L)=o(1)$.
These bounds give the same $C^1$ energy convergence for $m=6$;
Lemma~\ref{mc:natural} still applies.

\subsection{The reduced energy on a flat torus}\label{af:section}
Assume case~(iii) of Theorem~\ref{mc:main}, including \eqref{af:jet}.
On a flat torus, the first nonzero harmonic jet determines the
separation scale and angle. Write $T=\C/\Lambda$, $A=|T|$, and
\[
 \begin{aligned}
  G(z,0)&=-\frac1{2\pi}\log|z|+R(z), \\
  R(z)&=R(0)+\frac{|z|^2}{4A}
       +\operatorname{Re}\sum_{\substack{k\ge2\\k\text{ even}}}r_kz^k.
 \end{aligned}
\]
Translation and inversion symmetry give this expansion; the quadratic
term contributes to the smooth interactions.

\subsubsection{The projected family at the flat scale}
Put $L=1+|\log\varepsilon|$, $a=s^{S/2}$, and
$s=\varepsilon^{4/(S+4)}$ for $m>6$, or
$s=\varepsilon^{4/(S+4)}L^{1/(S+4)}$ for $m=6$. On compact parameter
sets with $d>0$, use
\begin{equation}
 \label{af:parameters}
 \begin{aligned}
  q_j&=se^{i\theta}(dz_j+a\xi_j), \quad 
  z_j=e^{2\pi ij/S}, \quad 
  \xi\in\mathcal E, \\
  \delta&=sd, \quad 
  \rho_j=\varepsilon, \quad
  \eta=\varepsilon/\delta.
 \end{aligned}
\end{equation}
Here $a$ is a shape scale, distinct from the matched masses $a_j$, and
\begin{equation}\label{af:scales}
 a^2=s^S,\qquad
 a\asymp\begin{cases}\eta^2,&m>6,\\\eta^2\sqrt L,&m=6,\end{cases}
 \qquad \log(\delta/\varepsilon)=\frac{S}{S+4}L+O(\log L).
\end{equation}
In particular, $\eta L^k\to0$, $\eta a=o(\eta^2)$, and
$\eta^{m-2}=o(a^2)$.

For each scaled coordinate $\lambda$ of $(d,\theta,\xi)$, matching gives
\begin{equation}\label{af:matching}
 |a_j-m|+|\partial_\lambda a_j|\le Ca/L,\qquad
 |\partial_q a_j|\le C/(sL),\qquad \sum_j a_j=Sm.
\end{equation}
Indeed, at equal masses
\[
 B_j(q_j)=2n\log|q_j|-m\sum_{k\ne j}\log|q_j-q_k|
 +\mathcal H(q_j)+2\pi m\{\sum_kR(q_j-q_k)-SR(q_j)\}.
\]
At the polygon, power sums below degree $S$ vanish. The singular,
quadratic, and degree-$S$ terms are therefore independent of $j$;
the remaining differences are $O(s^{S+1})$. Relative displacements
of size $a$ give $O(a)$ differences, including one scaled derivative.
The $O(L^{-1})$ inverse on the zero-sum mass space proves
\eqref{af:matching}. These cancellations also give
\begin{equation}
 \label{af:incident}
 \begin{aligned}
  |\nabla B_j(q_j)|+|\partial_\lambda\nabla B_j(q_j)|&\le Ca/s, \\
  |B_j(q_j+\varepsilon z)-B_j(q_j)|
 &\le C(\eta a|z|+\eta^2|z|^2).
 \end{aligned}
\end{equation}
The inner residual is $O(\eta^2)$ and the cutoff error is
$O(\eta^{m-\beta})=o(\eta^2)$, with $\beta$ as in
Section~\ref{mc:analytic-section}. The inverse proof requires
separation, $\eta L\to0$, and $\log R_j/L\to S/(S+4)>0$,
all supplied by \eqref{af:parameters}--\eqref{af:scales}. Moreover,
\[
 L\|\varepsilon^{-2}F'(u^{\rm app})-V_g^0\|_*
\le CL(\eta^2+\eta^{m-\beta})=o(1).
\]
Lemma~\ref{mc:inverse} and contraction give a $C^1$ projected family
for fixed $\varepsilon$, with
\begin{equation}
 \label{af:projected}
 \begin{aligned}
  \|\phi\|_\infty&\le C\eta^2L, \quad 
  |t_j|\le C\eta^2, \\
  |V_j(z)|&\le C\eta^2(|z|^2+L), \quad 
  V_j=u(q_j+\varepsilon z)-U_{a_j}(z).
 \end{aligned}
\end{equation}

With $\nu$, $M_j$, and $\zeta_j$ as in Section~\ref{mc:force-section},
the exterior tail gives
\begin{equation}\label{af:total}
 \tau:=\sum_jM_j-Sm=O(\eta^{m-2})=o(a^2).
\end{equation}
The first comparison moment is $O(\varepsilon\eta^2L)$ and the
second absolute moment is $O(\varepsilon^2)$. Thus
\begin{equation}\label{af:green-mass}
 \int_{D_k}R(q_j-y)\,d\nu(y)-M_kR(q_j-q_k)
 =O(\varepsilon\eta^2L+\varepsilon^2).
\end{equation}
Retaining the $O(\eta^3)$ term in Lemma~\ref{mc:actual-masses},
the radial source and flux contribute $O(\eta^4L^2)$, the other-core
logarithms contribute $O(\eta^3L)$, and \eqref{af:green-mass} bounds
the smooth error. The inverse on zero-sum masses and total flux give
\begin{equation}\label{af:mass}
 M_j-a_j=O(\eta^4L^2+\eta^3+\varepsilon\eta^2
              +\varepsilon^2/L+\eta^{m-2})=o(a).
\end{equation}
For $m=6$, $|a_j-6|\log(1/\eta)=O(a)$ makes the radial bound
\[
 O(\eta^4\int_1^{c/\eta}(1+\log r)\,dr/r)=O(\eta^4L^2)
\]
uniform for masses near six, both above and below. Consequently,
\begin{equation}\label{af:actual}
 M_j=m+o(a),\qquad
 \zeta_j=q_j+O(\varepsilon\eta^2L)=q_j+o(sa).
\end{equation}

\subsubsection{Smooth interactions and the reduced energy}
Set $\mu_j=M_j-m$, $\mathfrak p_k=\sum_jM_j\zeta_j^k$, and
$T_2=\sum_jM_j|\zeta_j|^2$. The smooth force potential at fixed masses is $\sum_jM_jW(\zeta_j)+\pi\sum_{j,k}M_jM_kR(\zeta_j-\zeta_k)$,
where $W=\mathcal H-2\pi SmR$. Expanding $R$ gives, up to a constant,
\begin{equation}\label{af:regular-energy}
 \mathcal W_M=\sum_jM_j\mathcal H(\zeta_j)
 +\frac{\pi}{2A}(\tau T_2-|\mathfrak p_1|^2)+\mathcal R_M.
\end{equation}
Through degree $2S$, $\mathcal R_M$ is a linear combination
of $\tau\mathfrak p_k$ and $\mathfrak p_l\mathfrak p_{k-l}$,
$1\le l<k$, with remainder $O(s^{2S+1})$. By \eqref{af:actual},
$\mathfrak p_l=O(s^la)$ for $l<S$ and $\mathfrak p_l=O(s^l)$ otherwise. Products with both indices below $S$
are $O(s^ka^2)$; all others have degree at least $S+1$.
Since $a^2=s^S$ and $\tau=o(a^2)$, the corrections in
\eqref{af:regular-energy} and their fixed-mass dilation and rotation
derivatives are $o(a^2)$. Thus the smooth point forces satisfy
\begin{align}
 \sum_jM_j\zeta_j\cdot H_j^M
 &=mS^2a^2\operatorname{Re}(\gamma_Sd^Se^{iS\theta})+o(a^2),
 \label{af:smooth-dilation}\\
 \sum_jM_j\zeta_j^\perp\cdot H_j^M
 &=-mS^2a^2\operatorname{Im}(\gamma_Sd^Se^{iS\theta})+o(a^2).
 \label{af:smooth-torque}
\end{align}
These identities also hold for the smooth stress contribution.
The fixed smooth background and the $R$ interactions have harmonic
gradient components with uniformly bounded second derivatives,
so radial averaging is exact for the reference measures.
The second absolute moment of $\nu_j-\nu_j^0$ is bounded by
\begin{equation}\label{af:nonradial-second}
 C\varepsilon^2\eta^2\int_0^{c/\eta}r^3(r^2+L)(1+r)^{-m}dr
 \le C\varepsilon^2\eta^2L^2.
\end{equation}
Taylor expansion at the barycenters cancels the first moments;
recentering contributes $O(\varepsilon^2\eta^4L^2)$.
Replacing both source and target cores by point masses therefore
introduces an error $O(s\varepsilon^2\eta^2L^2)=o(a^2)$ in either
weighted sum. The removed exterior-density potential is estimated
separately: its gradient bound $O(\eta^{m-2}/s)$ contributes
$O(\eta^{m-2})=o(a^2)$ to these sums.

For $r_j=sde^{i\theta}z_j$ and $C_S=m(S^2-1)/24$, the incident
coefficients are $a_{j,2}=-C_Sr_j^{-2}+O(a/s^2)+O(1)$ and
$a_{j,3}=C_Sr_j^{-3}+O(a/s^3)+O(1)$.
The joint proof of Lemmas~\ref{mc:moments} and \ref{cc:moments} requires
only \eqref{af:projected}, $|a_j-m|L=O(a)=o(1)$, and
$R_j\asymp\eta^{-1}$. They give the leading second moment $2b\varepsilon^4\overline{a_{j,2}}$,
with $b=b(m)$ for $m>6$ and $b=(A_6/4)\log(\delta/\varepsilon)$
at mass six. The second-moment errors are
$O(\varepsilon^2\eta^4L^2)+o(\varepsilon^4/s^2)$ for $m>6$,
and $O(\varepsilon^2(\eta^2+\eta^4L^2))$
plus $o(\varepsilon^4L/s^2)$ coefficient error at mass six.
The higher-moment force is respectively $O(\eta^5L/s)$ or
$O((\eta^4+\eta^5L)/s)$.
The stress calculation thus gives the quadrupole force
$-Qa^2s^{-1}d^{-5}e^{i\theta}z_j+o(a^2/s)$, where
\begin{equation}\label{af:Q}
 Q=\begin{cases}
 mb(m)(S^2-1)^2/144,&m>6,\\
 A_6S(S^2-1)^2/[96(S+4)],&m=6.
 \end{cases}
\end{equation}
For $m=6$, use $b=(A_6/4)\log(\delta/\varepsilon)$ and
\eqref{af:scales}. Multiplying the remaining stress errors by $s/a^2$ gives
$C(\eta^{m-6}+\eta L+\eta^2L^2+s^2)+o(1)$ for $m>6$ and
$C(L^{-1}+\eta+\eta^2L)+o(1)$ at mass six. All tend to zero.

Equations \eqref{mc:Euler}--\eqref{mc:torque} give
\begin{equation}
 \label{af:virial}
 \begin{aligned}
  \sum_jM_j\zeta_j\cdot F_j^{\rm pt}
 &=-2n\tau+\tfrac12\sum_j\mu_j^2-\tfrac12\tau^2=o(a^2), \\
  \sum_jM_j\zeta_j^\perp\cdot F_j^{\rm pt}&=0.
 \end{aligned}
\end{equation}
Lemma~\ref{mc:polygon} and \eqref{af:actual} yield
\begin{equation}\label{af:shape}
 P=\frac{a}{s}e^{i\theta}d^{-2}\mathscr L\xi+o(a/s).
\end{equation}
The smooth field is $O(sa+s^{S-1})=o(a/s)$ and the quadrupole is
$O(a^2/s)$; mass and center errors are also $o(a/s)$.
Replacing $\zeta_j$ by $q_j$ in these contractions gives normalized
error $O(\eta^3L/a)=o(1)$. Hence
\begin{align}
 \sum_jM_jq_j\cdot P_j
 &=ma^2\{S^2\operatorname{Re}(\gamma_Sd^Se^{iS\theta})-SQd^{-4}\}
       +o(a^2),\label{af:dilation}\\
 \sum_jM_jq_j^\perp\cdot P_j
 &=-mS^2a^2\operatorname{Im}(\gamma_Sd^Se^{iS\theta})+o(a^2).
 \label{af:torque}
\end{align}
Using 
\[
 \begin{gathered}
  \partial_\xi q=sae^{i\theta}\partial\xi, \quad 
  \partial_\theta q=iq, \\
  \text{and}\quad \partial_dq=d^{-1}(q-sae^{i\theta}\xi)
 \end{gathered}
\]
 as before, we obtain
\begin{equation}
 \label{af:oneform}
 \begin{aligned}
  \frac1{ma^2}\sum_jM_jP_j\cdot dq_j&=d\mathcal E_0+o(1), \\
  \mathcal E_0&=\frac{\langle\xi,\mathscr L\xi\rangle}{2d^2}
 +S\operatorname{Re}(\gamma_Sd^Se^{iS\theta})+\frac{SQ}{4d^4}.
 \end{aligned}
\end{equation}
For 
\[
 \mathcal I_\varepsilon=\mathcal J_\varepsilon(u-u_0),
\]
 apply
Lemma~\ref{mc:scaled-energy} with fixed $\rho_j=\varepsilon$. The
material derivative of the projection kernel has $L^1$ norm
$O(a/L)$, so the correction pairing is $O(a\eta^2)$;
the incident pairing is $O(\varepsilon a/s)=O(\eta a)$ by
\eqref{af:incident}. Since $|t|\le C\eta^2$,
\begin{equation}\label{af:energy}
 d\mathcal I_\varepsilon=-2\pi\sum_jM_jP_j\cdot dq_j
                  +O(a\eta^3+a\eta^4)
 =-2\pi\sum_jM_jP_j\cdot dq_j+o(a^2).
\end{equation}
Integration, with an additive constant fixed, gives
$ -\mathcal I_\varepsilon/(2\pi ma^2)-c_\varepsilon\to\mathcal E_0
$ 
in $C^1$.

\section{Critical points and geometric realizations}\label{sec:cs-realizations}
We apply Brouwer degree theory to isolated critical points and a linking argument to critical circles.

\subsection{Persistence of the critical circle}\label{mc:variational}
The limiting critical circle has a saddle structure in the shape
variables. Its persistence requires only $C^1$ convergence.
For $k,l\ge1$ and $r>0$, set
\[
 \begin{gathered}
  N=\angT\times\overline B_r^k\times\overline B_{r/2}^l, \\
  f_0(\theta,x,y)=-|x|^2+|y|^2,
 \end{gathered}
\]
with the product metric.
\begin{lemma}[Persistence under $C^1$ perturbations]\label{lem:persistence}
Let $f$ be $C^1$ on a neighborhood of $N$. If
$\sup_N|f-f_0|<r^2/16,\qquad \sup_N|df-df_0|<r/4,$
then $f$ has an interior critical point.
If $f_t\to f_0$ in $C^1$ on a fixed neighborhood, then every small
$t$ admits such a point
$(\theta_t,x_t,y_t)$ with $|x_t|+|y_t|\to0$.
\end{lemma}
\begin{proof}
Fix $\theta_0$ and let $\Gamma$ consist of continuous
$\gamma:\overline B_r^k\to N$ equal to $(\theta_0,x,0)$ on the
boundary. Every $\gamma$ has a zero of its $x$-component $X$; otherwise
$rX/|X|$ would retract the disk onto its boundary. At such a zero,
$f\circ\gamma\ge-r^2/16$, while $\gamma(x)=(\theta_0,x,0)$ gives
$f\circ\gamma\le r^2/16$. Hence
\begin{equation}\label{eq:level}
 -r^2/16\le
 c:=\inf_{\gamma\in\Gamma}\max_{\overline B_r^k}f\circ\gamma
 \le r^2/16.
\end{equation}

Suppose there is no critical point on $N$, and set
$\kappa=\min_N|\nabla f|>0$. The hypotheses give
$\nabla_y f\cdot y\ge3r^2/8$ on $|y|=r/2$ and
$f\le-11r^2/16$ on $|x|=r$. A sufficiently close smooth approximation $W$ of $-\nabla f$,
periodic in $\theta$, satisfies
\begin{equation}\label{eq:descending}
 df(W)\le-\kappa^2/2,\qquad
 W_y\cdot y<0\quad (|y|=r/2).
\end{equation}
Choose a smooth $\chi:\R\to[0,1]$ equal to zero below $-r^2/2$
and one above $-r^2/4$. The $C^1$ field $\chi(f)W$ has a flow
$\Phi_t$ that preserves $N$, fixes the $x$-boundary, and decreases
$f$ at rate at least $\kappa^2/2$ whenever $f\ge-r^2/4$.
For $0<\eta<r^2/32$, take $\gamma\in\Gamma$ with
$\max f\circ\gamma\le c+\eta$. Since $c-\eta>-r^2/4$, at
$T=4\eta/\kappa^2$ the deformed disk satisfies
$\max f\circ\Phi_T\circ\gamma\le c-\eta$ and still belongs to
$\Gamma$, a contradiction. Boundary critical points are excluded
by $|\nabla_xf|\ge2r-r/4$ on $|x|=r$ and the positive
$y$-normal derivative. Finally, for critical points of $f_t$,
\[
 \begin{gathered}
  2\sqrt{|x_t|^2+|y_t|^2} \le\|df_t-df_0\|_{C^0(N)}\longrightarrow0.
 \end{gathered}
\]
\end{proof}

\subsection{Curved polygonal clusters}
Assume case~(i) of Theorem~\ref{mc:main}, so $K_g(p)<0$.
Set $Q_S=\widehat Q_S$ when $m=6$, and use the scale
\eqref{mc:analytic-regimes}.
By Lemma~\ref{mc:polygon}, the critical shape is
$\xi_{*,j}=d^2\overline{h_\theta}z_j^2/(2n)$.
Since $\sum_jz_j^2=0$, completing the square gives
\begin{equation}
 \label{mc:complete-square}
 \begin{aligned}
  \mathcal J_0&=
 \frac{\langle\mathscr L(\xi-\xi_*),\xi-\xi_*\rangle}{2d^2}
 +V(d), \\
  V(d)&=\frac{S\lambda d^2}{2}+\frac{SQ_S}{4d^4}.
 \end{aligned}
\end{equation}
Here $\lambda=-mK_g(p)/4>0$, so
$d_*=(Q_S/\lambda)^{1/6}$ is the unique critical radius and
$V''(d_*)=6S\lambda>0$.
For $m=6$, the radius is
\begin{equation}\label{cc:radius}
 d_*^6=\frac{\widehat Q_S}{\lambda}
       =\frac{A_6(S^2-1)^2}{-432K}.
\end{equation}
Decompose $\chi=\xi-\xi_*=\chi_-+\chi_+$ into the negative
and positive spectral spaces $E_-,E_+$, each of dimension $S-1$,
and put
\[
 \begin{aligned}
  x&=\frac{(-\mathscr L|_{E_-})^{1/2}\chi_-}{\sqrt2d}, \qquad 
  y'=\frac{(\mathscr L|_{E_+})^{1/2}\chi_+}{\sqrt2d}, \\
  t&=\operatorname{sgn}(d-d_*)\sqrt{V(d)-V(d_*)}.
 \end{aligned}
\]
Since $V''(d_*)>0$, the last coordinate is smooth with nonzero
derivative at $d_*$. These coordinates are $2\pi$-periodic in
$\theta$ and give, near the labelled critical circle,
\begin{equation}\label{mc:normalform}
 \mathcal J_0=V(d_*)-|x|^2+|y|^2,\qquad
 x\in\R^{S-1},\quad y=(y',t)\in\R^S.
\end{equation}

These coordinates locally parametrize the full $2S$-dimensional position space:
for $q=se^{i\theta}(dz+s\xi)$, $\xi\in E$,
we have $q_d=se^{i\theta}z$ and
$q_\theta=ise^{i\theta}(dz+s\xi)$.
If a tangent vector lies in the kernel of the differential of the position map, projection onto
$\Span_\R\{z,iz\}$ gives $d'=\theta'=0$, then $\xi'=0$.
The position map is therefore a local diffeomorphism.

Choose the additive constant in
\[
 \mathcal E_\eps=-\mathcal J_\eps(u-u_0)/(2\pi ms^2)-c_\eps.
\]
Equations \eqref{mc:scaled-energy-formula} and \eqref{mc:oneform-limit} give
$\mathcal E_\eps\to\mathcal J_0$ in $C^1$ near the critical circle.
In coordinates \eqref{mc:normalform}, Lemma~\ref{lem:persistence}
gives a critical point approaching this circle. The local
diffeomorphism and Lemma~\ref{mc:natural} eliminate all translation
multipliers, proving \eqref{mc:pde}.

\subsection{Two-core clusters on a fixed surface}\label{mc:pair-section}
For $S=2$, compatibility gives $m=4n\ge8$ and $N=m$.
The projected construction and stress calculation apply, retaining
$h$ and the trace-free Hessian $B$ in the dilation equation.
Each positive eigenvalue $\lambda$ of $A_{\rm eff,m}$ gives
\eqref{mc:locations}--\eqref{mc:measure} with
\begin{equation}\label{mc:pair-radius}
 \delta_\eps=\left(\frac{mb(m)}{16\lambda}\right)^{1/6}
                   \eps^{2/3}(1+o(1)).
\end{equation}

\begin{proof}[Proof of Theorem~\ref{mc:main}\textup{(ii)}]
Put $s=\eps^{2/3}$, $a=2n=m/2$, and $Q_2=mb(m)/16$.
For $q_\pm=\pm sd e+s^2c$, $|e|=1$, solving for the common
translation gives $c=d^2(2ee^T-I)h/a$. Substitution contributes
$h\cdot c=d^2\{2(h\cdot e)^2-|h|^2\}/a$ to the energy, leaving
$d^2e^TA_{\rm eff,m}e+\frac{Q_2}{2d^4}.$
The same formula follows from \eqref{mc:scaled-energy-formula}, retaining
$B$ in the virial and torque. The error estimates require no
curvature sign; the common-translation block has one positive
and one negative direction.

For a simple eigenvalue $\lambda>0$ with unit eigenvector $e$, set
$d_*^6=Q_2/\lambda$. The radial and angular second derivatives are
$12\lambda$ and $2d_*^2(\lambda_{\rm other}-\lambda)$, respectively.
Thus the critical point is nondegenerate, and $C^1$ convergence
preserves its nonzero local Brouwer degree.
For a repeated eigenvalue, $A_{\rm eff,m}=\lambda I$ and the critical
circle has one negative and two positive normal directions.
Completing the square in $c$ and using the radial coordinate in
\eqref{mc:normalform} allows us to apply Lemma~\ref{lem:persistence}.
In both cases Lemma~\ref{mc:natural} eliminates the multipliers;
the projected estimates give the stated limits.
\end{proof}

\subsection{Flat-torus clusters}
Return to the parameters \eqref{af:parameters} under
Theorem~\ref{mc:main}(iii). The limiting energy \eqref{af:oneform}
has a critical point at $\xi=0$,
$\gamma_Se^{iS\theta_*}=|\gamma_S|$, and
$d_*^{S+4}=Q/(S|\gamma_S|)$. Its Hessian has shape block
$d_*^{-2}\mathscr L|_{\mathcal E}$, radial entry
$S^2(S+4)|\gamma_S|d_*^{S-2}>0$, angular entry
$-S^3|\gamma_S|d_*^S<0$, and zero cross blocks.
Its nonzero Brouwer degree persists under $C^1$ convergence.
Projection onto $\operatorname{Span}_{\R}\{z,iz\}$ verifies that
$(d,\theta,\xi)$ are local position coordinates, and
\[
 \|\partial_qK\|_{L^1}\le C(\varepsilon^{-1}+(sL)^{-1})
\]
 gives
\[
 \mathcal B=-\varepsilon^{-1}\operatorname{diag}(\mathfrak G(a_j)I_2)
+o(\varepsilon^{-1})
\]
 as in Lemma~\ref{mc:natural}.
Thus all multipliers vanish, with
$\delta_\varepsilon=sd_*(1+o(1))$ and $\theta_\varepsilon\to\theta_*$.

\subsection{Concentration and completion of the constructions}\label{mc:completion-section}
All three constructions in Theorem~\ref{mc:main} satisfy
\begin{align}
 Q_{\eps,j}&=\delta_\eps e^{i(\theta_\eps+2\pi j/S)}+o(\delta_\eps),
       \quad 0\le j<S,\label{mc:locations}\\
 u_\eps(Q_{\eps,j}+\eps z)&\longrightarrow U_m(z)
       \quad\hbox{in }C^2_{\loc}(\R^2),\label{mc:profiles}\\
 \eps^{-2}F(u_\eps)\,dV_g&\rightharpoonup2\pi Sm\delta_p.\label{mc:measure}
\end{align}
For case~(i), $\delta_\varepsilon=s d_*(1+o(1))$, with $s$ from
\eqref{mc:analytic-regimes} and $d_*^6=Q_S/\lambda$; at mass six, see
\eqref{cc:radius}. For the pair and flat torus, use their radii above.
Angle convergence is understood subsequentially, except for the
chosen flat-torus branch, where $\theta_\varepsilon\to\theta_*$.

\begin{proof}[Completion of Theorem~\ref{mc:main}]
In case~(i), $d_\eps\to d_*$ and
$\xi_\eps-\xi_*(d_*,\theta_\eps)\to0$; the other cases have the
critical parameters found above. All multipliers vanish.
Proposition~\ref{mc:projected-family} and its flat-scale extension
\eqref{af:projected} give $\sup_Mu\le-\gamma<0$ and smoothness of $u-u_0$. Since $D^2U_m(0)=-F(U_m(0))I/2$,
local $C^2$ convergence produces a nondegenerate maximum
$Q_{\eps,j}=q_j+O(\rho\eta^2L)=q_j+o(s^2)$.
The core masses and exterior tail give \eqref{mc:measure}; on the
flat torus, use \eqref{af:total} and \eqref{af:actual}.

For $w_\eps=u_\eps-u_0-2\log\eps$, use
$u_0(x)=2n\log|x|+O(1)$ and $|Q_{\eps,j}|\sim sd_*$ to obtain
\[
 w_\eps(Q_{\eps,j})+2(n+1)\log|Q_{\eps,j}|
 =2\log(s/\eps)+O(1)\longrightarrow+\infty.
\]
Since $S\ge2$ and distinct cores are separated by at least
$c\delta_\eps$, any center lies at distance at least $c\delta_\eps/2$
from one core. The same divergence therefore holds with
$|Q_{\eps,j}|$ replaced by the distance to that center, proving
non-simplicity in Definition~\ref{def:simple}. The common outer constant
$m\log\eps+2n\log s+O(1)$ gives $w_\eps\to-\infty$ locally away
from $p$ and the fixed vortices, with no other concentrating cores.
\end{proof}

\subsection{Divisors on a prescribed surface}\label{sec:universal-proof}
\begin{proof}[Proof of Corollary~\ref{cor:every-surface}]
In normalized coordinates at $p=0$, fix a unit vector $v$ and put
$a_t=tv$. The divisor $n[p]+3n[a_t]$ is compatible with $S=2$,
$m=4n$, and $N=4n$. Smoothness of the regular Green part gives
\[
 \mathcal H_t(x)=\frac{3m}{2}\log|x-tv|+E_t(x),\qquad
 E_t(x)=3\pi m\{R(x,0)-R(x,tv)\},
\]
where $E_t$ has uniformly bounded derivatives through order two.
Hence $h_t=-3mv/(2t)+O(1)$ and
$B_t=3m(I-2v\otimes v)/(2t^2)+O(1)$. Substitution in
\eqref{mc:pair-matrix} yields
\[
 A_{\rm eff,m}(t)=\frac{3m}{t^2}(2v\otimes v-I)+O(t^{-1}).
\]
For fixed small $t>0$, the positive eigenvalue is simple.
Theorem~\ref{mc:main}(ii) gives concentration mass $4\pi N=16\pi n$,
with metric and divisor fixed as $\varepsilon\to0$.

On a flat torus, this divisor can have trivial isometry stabilizer. Choose $v$ outside the fixed lines of the nonidentity
reflections in the finite lattice point group, and
$t<\frac13\min_{\Lambda\setminus\{0\}}|w|$. Distinct multiplicities
force a divisor-preserving isometry to fix $0$ and $tv$; its linear
part then fixes $v$ and is the identity.
\end{proof}

\subsection{Realizing the finite jet}\label{fgj:realization}
\begin{proof}[Proof of Corollary~\ref{mc:flat-realization}]
Proposition~\ref{ap:mass-polygon} and \eqref{ap:total-mass} give the
necessary identities $Sm=2N$ and $m(S-1)=4n$ for a single
non-simple CS cluster exhausting the total mass.
Subtraction gives $m=2(N-2n)=2r>4$, hence $r\ge3$ and
$S=N/r\in\mathbb N$, $S\ge2$.
Conversely, these conditions with $m=2r$ imply
\eqref{mc:compat}--\eqref{mc:total}, including the parity required
by $n=r(S-1)/2$.
For $S=2$, Subsection~\ref{sec:universal-proof} provides the divisor
and its trivial isometry stabilizer.
For $S\ge3$, translate $p$ to zero and set
$R_*=N-n=m(S+1)/4$, $q=R_*-S$. Integer compatibility gives
$m\ge6$, $n\ge3$, and $q\ge3$. Choose $a\ne0$ outside the
fixed-point sets of nonidentity elements of the finite lattice
point group. Seek a divisor
\[
 D_t=n[0]+\sum_{j=1}^S[tb_j(t)]+q[a],\qquad
 b_j(t)=b_j^0+O(t),\quad b_j^0=e^{2\pi i(j-1)/S},
\]
with $b_S=b_S^0$ fixed. Its background is
\[
 \mathcal H_{t,b}(z)=2\sum_j\log|z-tb_j|+E_{t,b}(z),
\]
where $E_{t,b}$ is smooth through $t=0$ and harmonic in $z$,
since $R_*-S-q=0$ cancels the regular-part Laplacians.
Writing 
\[
 E_{t,b}(z)=E_{t,b}(0)+\operatorname{Re}\sum_{k\ge1}E_k(t,b)z^k,
\]
we must solve
$\sum_jb_j^{-k}-(kt^k/2)E_k(t,b)=0$ for $1\le k<S$.
At $(0,b^0)$, the equations hold. Their derivative in
$(b_1,\ldots,b_{S-1})$ is the invertible complex Vandermonde matrix
$J_{kj}=-k(b_j^0)^{-k-1}$. Its real determinant is
$|\det_{\C}J|^2>0$, so the real implicit function theorem gives
$b_j(t)=b_j^0+O(t)$, canceling the first $S-1$ coefficients;
holomorphic dependence on $b$ is unnecessary. The next coefficient is
$\gamma_S(t)=-2t^{-S}(1+O(t))\ne0$, since
$\sum_j(b_j^0)^{-S}=S$. Hence \eqref{af:jet} holds.
Multiplicities distinguish the polygon from $\{0,a\}$, and
its distances to these points distinguish $0$ from $a$, even if $n=q$.
Every divisor-preserving isometry therefore fixes both points
and is the identity by the choice of $a$. Fix such a small $t$
before taking $\varepsilon\to0$ and apply Theorem~\ref{mc:main}(iii).
\end{proof}
\section{Mean-field two-core concentration}
\label{sec:mf-global-direct}
A quartic interaction between the local geometry and the Green function
selects the mean-field pair. For a vortex $p$ of strength $n$, set
$\rho=4\pi N_{\tot}$ and
\begin{equation}\label{mfg:background-jet-field}
 \mathcal B_p(x)=(\rho-4\pi n)R(x,p)
       -4\pi\sum_{p_i\ne p}N_iG(x,p_i).
\end{equation}
This field is harmonic near $p$. In the chart \eqref{mc:chart}, let $P_4$
be the homogeneous quartic term of $\sigma+\mathcal B_p$. For
$e_\theta=(\cos\theta,\sin\theta)$, set
\begin{equation}\label{mfg:quartic-definition}
 Q_{\rm eff}(e)=2P_4(e)+4\pi D_x^2D_y^2R(0,0)[e,e;e,e].
\end{equation}
\begin{theorem}[Existence of mean-field type non-simple blow-up]
\label{mfg:thm:general-quartic-pair}
Let $(M,g)$ be a fixed smooth closed surface with a fixed integer
divisor of total strength four, and let $p$ have strength one. Assume
$K_g(p)=0$, $\nabla\mathcal B_p(p)=0$, and $D^2\mathcal B_p(p)=0$
in a normalized conformal chart. Suppose that, for some $\theta_0$,
\begin{equation}
 \label{mfg:global-quartic-existence-conditions}
 \begin{aligned}
  &\Delta_gK_g(p)>0, \qquad 
  Q_{\rm eff}(e_{\theta_0})>0, \\
  \partial_\theta Q_{\rm eff}(e_{\theta_0})&=0,\quad
 \partial_\theta^2Q_{\rm eff}(e_{\theta_0})\ne0.
 \end{aligned}
\end{equation}
Then \eqref{eq:main} admits, for every sufficiently small $\eps>0$,
a mean-field type non-simple blow-up solution with two cores at $p$,
and the resulting family satisfies
\[
 \eps^{-2}F(u_\eps)\,dV_g
 \rightharpoonup 16\pi \delta_p \quad \mbox{as } \eps\to 0.
\]
\end{theorem}
The vanishing conditions eliminate the quadratic interaction;
$Q_{\rm eff}$ selects the scale and direction without an isometric
symmetry assumption.

\emph{Coordinate invariance.}
Under the vanishing jet assumptions, a change between normalized charts,
after removing its rotation, has the form
\[
 f(\zeta)=\zeta+c_4\zeta^4+c_5\zeta^5+O(|\zeta|^6).
\]
 Then
\begin{align*}
 \widetilde P_4(\zeta)&=P_4(\zeta)+4\operatorname{Re}(c_5\zeta^4),\\
 4\pi D_\xi^2D_\eta^2\widetilde R(0,0)[e,e;e,e]
 &=4\pi D_x^2D_y^2R(0,0)[e,e;e,e]-8\operatorname{Re}(c_5e^4).
\end{align*}
The second identity follows from
\[
 \widetilde R(\xi,\eta)=R(f(\xi),f(\eta))
 +(2\pi)^{-1}\log| (\xi-\eta)/(f(\xi)-f(\eta)) |.
\]
The changes cancel in $Q_{\rm eff}$. Rotations and reflections preserve
its positivity and angular nondegeneracy, and $\Delta_gK_g(p)$ is intrinsic.

Throughout the construction, the surface, metric, and divisor are fixed,
with $n=1$, $S=2$, $N_{\tot}=4$, and $\rho=16\pi$.
Each core has normalized mass four, so the two $8\pi$ cores exhaust
the total mass. Set $A_M=|M|_g$ and
use $u_0$ from \eqref{eq:u0-w},
so $\Delta_gu_0=4\pi\sum_iN_i\delta_{p_i}-\rho/A_M$.

\subsection{Normalization and the approximate solution}
\label{mfg:normalization-ansatz}
For mean-zero $w$, set
\begin{equation}
 \label{mfg:normalization-data}
 \begin{aligned}
  Q_w&=e^{u_0+w}, \quad 
  a_w=\int_MQ_w\,dV_g, \\
  b_w&=\int_MQ_w^2\,dV_g, \quad 
  u=u_0+w+\log(\eps^2\lambda).
 \end{aligned}
\end{equation}
Integer strengths make $Q_w$ smooth at the vortices. Integrating
\eqref{eq:main} gives
\begin{equation}\label{mfg:scalar-mass}
 \lambda a_w-\eps^2\lambda^2b_w=\rho.
\end{equation}
\begin{proposition}[Normalization and variational structure]\label{mfg:prop:normalization}\label{mfg:prop:linear-energy}
If $a_w^2>4\rho\eps^2b_w$, define the smaller positive root
\begin{equation}\label{mfg:lambda}
 \lambda_\eps(w)=\frac{2\rho}{a_w+\sqrt{a_w^2-4\rho\eps^2b_w}}.
\end{equation}
On this branch, \eqref{eq:main} is equivalent to
\begin{equation}
 \label{mfg:residual}
 \begin{aligned}
  \mathcal R_\eps(w)&=\Delta_gw+T_\eps(w)-\rho/A_M=0, \\
  T_\eps&=\lambda_\eps Q_w-\eps^2\lambda_\eps^2Q_w^2.
 \end{aligned}
\end{equation}

Every admissible $w$ satisfies $\int_MT_\eps\,dV_g=\rho$ and
$\int_M\mathcal R_\eps\,dV_g=0$.

Put $V_w=\lambda_\eps Q_w-2\eps^2\lambda_\eps^2Q_w^2$.
Then 
\[
 \int_MV_w\,dV_g=\lambda_\eps\sqrt{a_w^2-4\rho\eps^2b_w}>0
\]
 and,
for every smooth mean-zero variation $\phi$,
\begin{equation}\label{mfg:linearization}
 D\mathcal R_\eps(w)[\phi]=\Delta_g\phi+V_w
 \left(\phi-\frac{\int_MV_w\phi\,dV_g}{\int_MV_w\,dV_g}\right).
\end{equation}
The exact energy
\begin{equation}\label{mfg:energy}
 \mathcal E_\eps(w)=\frac12\int_M|\nabla_gw|^2\,dV_g
 -\lambda_\eps a_w+\frac{\eps^2}{2}\lambda_\eps^2b_w
 +\rho\log\lambda_\eps
\end{equation}
satisfies 
\[
 D\mathcal E_\eps(w)[\phi]=-\int_M\mathcal R_\eps(w)\phi\,dV_g.
\]
Its critical points on the admissible mean-zero space solve
\eqref{mfg:residual}.
\end{proposition}
\begin{proof}
Substituting \eqref{mfg:normalization-data} into \eqref{eq:main} gives
the residual and mass identities;
\eqref{mfg:lambda} is the rationalized smaller root of \eqref{mfg:scalar-mass}.
Differentiating \eqref{mfg:scalar-mass} gives
\[
 D\lambda_\eps[\phi]=-\int_MV_w\phi/(a_w-2\eps^2\lambda_\eps b_w);
\]
the denominator is the positive square root in \eqref{mfg:lambda}.
Since $Q_w-2\eps^2\lambda_\eps Q_w^2=V_w/\lambda_\eps$, this proves
\eqref{mfg:linearization}. In \eqref{mfg:energy}, all $D\lambda_\eps$
terms cancel by $-a_w+\eps^2\lambda_\eps b_w+\rho/\lambda_\eps=0$.
Integration by parts gives the energy variation. At a constrained
critical point, the residual is constant and has zero integral, hence vanishes.
\end{proof}
At $\eps=0$ the linearization is
$\Delta_g\phi+\rho p_w(\phi-\int_Mp_w\phi\,dV_g)$, where $p_w=Q_w/a_w$.

\emph{The matched family.}
In the chart \eqref{mc:chart}, write $u_0=2n\log|z|+r_0$ and
$G(x,p)=-(2\pi)^{-1}\log|z|+R(x,p)$. The Green equations give
$r_0+\rho R=\mathcal B_p$ and $\Delta\mathcal B_p=0$, with
$\mathcal B_p$ from \eqref{mfg:background-jet-field}. Hence
$\mathscr F=\sigma+\mathcal B_p$ satisfies $\Delta\mathscr F=-2K_ge^\sigma$.

Identify the chart with $\C$. Write $q_\pm=m\pm s$,
$a=s^2$, $A=|a|=|s|^2$, $e=s/|s|$, and fix a radial cutoff $\chi$.
Here $m$ denotes the common center; the normalized core mass remains four.
For $m=0$, set
\begin{align*}
 v_{b,a}&=\log\frac{32b^2}{(b^2+|z^2-a|^2)^2},
 &H_{b,a}&=|z|^2e^{v_{b,a}},\\
 \mathscr F_q&=\sigma+r_0+8\pi\{R(z,s)+R(z,-s)\},
 &C_q&=\mathscr F_q(0).
\end{align*}
Glue $v_{b,a}+8\pi\{R(z,s)+R(z,-s)\}-C_q$ to
$\log(32b^2)+8\pi\{G(x,s)+G(x,-s)\}-C_q$ using $\chi$, and subtract
its $g$-mean to obtain $w_{b,a}$.
For small $b_0,A_0>0$, let $L_0=\log(1/b_0)$, $A_0\asymp b_0\sqrt{L_0}$,
and $\eps^2A_0/b_0^2\asymp A_0^2$, with fixed comparability bounds.
Let $\zeta=(A/A_0,b/b_0,\arg s)$ range in a fixed compact parameter
set with positive first two coordinates, and let $t$ range in a fixed
bounded subset of $\R^2$. Put $m=A_0^2t$ and $\delta_0=\sqrt{A_0}$.
Translate the polynomial to $(z-m)^2-s^2$ and the Green poles to
$m\pm s$; mean subtraction removes the matching constant. Denote this
mean-zero family by $w_{m,b,a}$ and its normalized regular part by
$v_0=w_{m,b,a}+\log\lambda_\eps(w_{m,b,a})$.
The separation and core width are $d=2|s|$ and $\ell=b/(2\sqrt A)$.
Thus $R=d/\ell=4A/b\asymp\sqrt{L_0}$. The inverse estimate involves
$L=\log R$, whereas the quartic energy involves the larger logarithm $L_0$.

\emph{Lyapunov--Schmidt reduction.}
There are six local modes: two translations and one dilation per core.
Proposition~\ref{mfg:prop:relative-model-gap} inverts the relative
dilation through its gap $3/(2\log R)$, leaving four translations and
the common dilation. They correspond to the five parameters $(m,s,b)$.
With the projection $\Pi_5$ and zero-integral sources $K_\alpha$
constructed below, we solve
\begin{equation}
 \label{mfg:projected-equation}
 \begin{aligned}
  \Delta_g(v_0+\phi)+e^{u_0+v_0+\phi}
 -\eps^2e^{2u_0+2v_0+2\phi}-\rho/A_M
 &=\sum_\alpha t_\alpha K_\alpha, \\
  \Pi_5\phi&=0,
 \end{aligned}
\end{equation}
then eliminate the five multipliers through the corrected energy.
The leading energy variation is $O(\delta_0^4)$, whereas the center
variation is $O(\delta_0^6)$. The general inverse estimate bounds the
energy correction only by $O(\delta_0^6\log^2R)$. We therefore compute
its leading odd contribution with a remainder $o_{C_t^1}(\delta_0^6)$.

\subsection{Moment identities and linear estimates}
\label{mfg:shared-inputs}
\begin{lemma}[Two-core moments]\label{mfg:lem:pair-moments}
For $a,b>0$, put $H_{b,a}(z)=32b^2|z|^2/(b^2+|z^2-a|^2)^2$.
Then, with $dz$ denoting Euclidean area,
\begin{equation}\label{mfg:pair-exact-moments}
 \int_\C H_{b,a}\,dz=16\pi,\qquad
 \int_\C z^2H_{b,a}\,dz=16\pi a.
\end{equation}
As $b/a\to0$, also
\[
 \int_\C|z|^2H_{b,a}\,dz=16\pi a+O((b^2/a)\log(a/b)).
\]
\end{lemma}
\begin{proof}
The two-to-one map $w=z^2$ gives
\[
 \int\Phi(z^2)H_{b,a}\,dz=16b^2\int\Phi(w)(b^2+|w-a|^2)^{-2}\,dw.
\]
Taking $\Phi=1,w$ proves the identities. For $\Phi=|w|$, put
$w=a+by$ and $t=b/a$. On $|y|<(2t)^{-1}$, expand
$|1+ty|=1+t\Re y+\frac12t^2(\Im y)^2+O(t^3|y|^3)$.
The linear term is odd, the quadratic integral is $O(\log(1/t))$,
and the integrated remainder and exterior error are $O(t^2)$.
\end{proof}
For complex $a$, rotation preserves the exact identities and replaces
$a$ by $|a|$ in the real moment estimate.

\label{mfg:after-jets}
Assume the vanishing conditions of Theorem~\ref{mfg:thm:general-quartic-pair}.
In the normalized chart,
$\mathscr F-\mathscr F(0)=P_3+P_4+O(|z|^5)$, with homogeneous $P_j$. Write
\begin{equation}\label{mfg:quartic-decomposition}
 P_4=q_0|z|^4+|z|^2\Re(q_2z^2)+\Re(q_4z^4),
 \qquad q_0\in\R,\quad q_2,q_4\in\C.
\end{equation}

The identity $\Delta\mathscr F=-2K_ge^\sigma$ gives
$q_0=-\Delta_gK_g(p)/32$; coordinate and metric Laplacians agree at $p$.

Quartic moments require a radial cutoff: $\int_\C|z|^4H_{b,a}$
diverges logarithmically and $\int_\C z^4H_{b,a}$ is not absolutely
convergent. Cubic terms vanish by parity.

Write $\theta=\arg s$, $R_*=A/b$, and use
$D_1=A\partial_A$, $D_2=b\partial_b$, $D_3=\partial_\theta$ and
$\|h\|_{1,*}=\sup(|h|+\sum_j|D_jh|)$ on the fixed parameter set.
This norm is equivalent to $C^1_\zeta$; $R_*=R/4\asymp\sqrt{L_0}$.
In the moment calculations, abbreviate $H_{b,a}$ by $H$.
\emph{Differentiated moments.}
For even $h$, with $B_{b,a}(w)=8b^2/(b^2+|w-a|^2)^2$,
\[
 \int\chi H_{b,a}h\,dz
 =2\int\chi(\sqrt w)B_{b,a}(w)h(\sqrt w)\,dw.
\]
Evenness makes the right side single valued. The identities
$\partial_{a_j}B=-\partial_{w_j}B$ and
$b\partial_bB=-\operatorname{div}_w((w-a)B)$ transfer parameter derivatives
to the test function, giving
\[
 \begin{aligned}
  \left\|\int\chi H|z|^2\right\|_{1,*}
 &\le C\left(A+\frac{b^2}{A}\log(2+R_*)\right), \\
  \left\|\int\chi H|z|^4\right\|_{1,*}
 &\le C(A^2+b^2\log(1/b)).
 \end{aligned}
\]
If $h$ and its scaled parameter derivatives satisfy
$|\nabla^jh|\le C|z|^{6-j}$ for $j\le2$, the bound is
$C(A^3+b^2+Ab^2\log(1/b))$. The derivatives after the square map
are $O(|w|^{3-j})$; integration by parts reduces the estimate to
moments of $|w|^3+A|w|^2$, with cutoff errors $O(b^2)$.
The same bound holds for \mbox{$O(|z|^6+A|z|^4+A^2|z|^2)$} remainders.
In particular,
\begin{equation}\label{mfg:uniform-quartic-moment}
 \begin{split}
 \int\chi P_4H_{b,a}&=16\pi P_4(s)+32\pi q_0b^2\log(1/b)\\
 &\quad+24\pi b^2\Re(q_2e^{2i\theta})\log R_*+O_{1,*}(b^2).
 \end{split}
\end{equation}

For these coefficients, the square map sends $P_4$ to
$\Phi(w)=q_0|w|^2+|w|\Re(q_2w)+\Re(q_4w^2)$.
Its radial and harmonic terms contribute, respectively,
$32\pi q_0b^2\log(1/b)$ and $16\pi\Re(q_4a^2)+O_{1,*}(b^2)$.
For $\Phi_2(w)=|w|\Re(q_2w)$, Taylor expansion on $|w-a|<A/2$ yields
$8\pi b^2\Delta\Phi_2(a)\log R_*$, with
$\Delta\Phi_2(a)=3\Re(q_2e^{2i\theta})$; its third derivatives
are $O(A^{-1})$, giving error $O(b^2)$.
Rescaling $w=A\xi$ gives the $A\partial_A$ and angular bounds;
the boundary terms from adjacent integration regions cancel.
The kernel identities give the $b\partial_b$ bound. Outside $|w|=2A$, angular cancellation leaves
the integrable majorant $CAb^2|w|^{-5}|\Phi_2(w)|$.

The square map and expansion of $|a+by|/A$ on $|y|<R_*/2$ give
\begin{equation}\label{mfg:shared-square-density}
 \int H^2=\frac{512\pi A}{3b^2}[1+O_{1,*}(R_*^{-2})].
\end{equation}
The linear term integrates to zero, the quadratic moment is finite,
and the exterior error is $O_{1,*}(R_*^{-6})$.
For $x=z-m$, write $v_P(x)=v_{b,s^2}(x)$,
$H_P=|x|^2e^{v_P}$ and $\varrho=|x|$. With $\delta=\delta_0$,
the preceding bounds and the square map give the following moments.
Integrals without a subscript are over a fixed coordinate disk.
\begin{gather*}
 \int H_P\varrho^2\le C\delta^2,\quad
 \int H_P\varrho^4\le C\delta^4,\quad
 \int H_P\varrho^6\le C(A^3+b^2),\\
 \int_\C e^{v_P}\le C/A,\quad
 \int_\C\varrho e^{v_P}\le C/\sqrt A,\quad
 \int_\C\varrho^2e^{v_P}=16\pi,\\
 \int_\C\varrho^2e^{2v_P}=512\pi/(3b^2),\quad
 \int_\C\varrho^4e^{2v_P}\le CA/b^2,\quad
 \int_\C e^{2v_P}\le C/(Ab^2).
\end{gather*}
For the last bound, split the square-map integral over
$|w-s^2|<A/2$, $|w|<A/2$, and their complement, using local
integrability of $|w|^{-1}$. Cauchy--Schwarz gives
$\int\varrho^3e^{2v_P}\le C\sqrt A/b^2$.

\emph{The model inverse.}\label{mfg:general-inputs}
\begin{proposition}[Inverse on the complement of six modes]\label{mfg:prop:six-mode-inverse}
Fix a nonnegative smooth cutoff $\chi$ in a coordinate disk about $p$,
equal to one near $p$. Let $q_1,q_2\to p$, $d=|q_1-q_2|$, and $\ell/d\to0$.
Define
\[
 \begin{aligned}
  d\mu=\sum_{i=1}^2\chi(z)\frac{4\ell^2}{(\ell^2+|z-q_i|^2)^2}\,dz
       +\ell^2dV_g, \quad 
  H=2\frac{d\mu}{dV_g}, \quad 
  \widehat g=(H/2)g.
 \end{aligned}
\]
For small parameters, the spectral projection $\Pi_6$ of
$-\Delta_{\widehat g}$ onto $(1,3)$ has rank six. For every $p_*>1$,
the unique solution of
$(\Delta_{\widehat g}+2)\phi=(I-\Pi_6)F$, $\Pi_6\phi=0$, satisfies
\begin{equation}\label{mfg:six-mode-estimate}
 \|\phi\|_\infty+
 \left(\int_M|\nabla\phi|_g^2dV_g+\int_M\phi^2d\mu\right)^{1/2}
 \le C_{p_*}\|F\|_{L^{p_*}(d\mu)}.
\end{equation}
The bound is uniform; the six limiting modes are the degree-one
spherical harmonics on the two cores, without an evenness restriction.
\end{proposition}
\begin{proof}
First, for every finite $q\ge2$,
\begin{equation}\label{mfg:model-weighted-sobolev}
 \|v\|_{L^q(d\mu)}^2\le C_q
 \left(\int_M|\nabla v|_g^2dV_g+\int_Mv^2d\mu\right).
\end{equation}

Write $E=\int_M|\nabla v|^2dV_g$,
$M_v=\int_Mv^2d\mu$, and
$W_i=4\ell^2/(\ell^2+|z-q_i|^2)^2$.
Choose a fixed $r_*>0$ with $B(q_i,2r_*)\subset\{\chi=1\}$.
Under $z=q_i+\ell y$, $W_i\,dz$ is round-sphere area measure,
and $B(q_i,r_*)$ omits a cap of area $O(\ell^2)$.
For the boundary trace $g_i$, trace--Poincar\'e and radial
Cauchy--Schwarz from a circle of radius in $[\ell,2\ell]$ give
$\|g_i-\bar g_i\|_{H^{1/2}}^2\le CE$, $|\bar g_i|^2\le C\{M_v+(1+\log(1/\ell))E\}$.
Averaging $v^2W_i$ selects a starting circle with squared mean at most
$CM_v$. Extend the $n$th nonconstant Fourier mode harmonically into
the omitted cap by
$g_{i,n}(r_*/(\ell|y|))^{|n|}e^{in\vartheta}$ and the mean stays
constant. The added energy is $O(E)$, and the cap mass is
at most $C\ell^2\{M_v+(1+\log(1/\ell))E\}$. The sphere Sobolev
inequality bounds each core $L^q(W_i\,dz)$ norm by
$C_q(E+M_v)^{1/2}$.

Fixed-surface Poincar\'e and the same trace bound give
\[
 \|v\|_{L^q(dV_g)}^2\le C_q\{M_v+(1+\log(1/\ell))E\}.
\]
Multiplication by $\ell^{4/q}$ controls $L^q(\ell^2dV_g)$ uniformly.
Outside $B(q_i,r_*)$,
$\chi W_i\,dz\le C\ell^2dV_g$; the sum of the energies on the overlapping core disks is at most
$2E$. Density gives \eqref{mfg:model-weighted-sobolev}.

Outside $B(q_1,T\ell)\cup B(q_2,T\ell)$ the measure is
$O(T^{-2}+\ell^2)$. Thus H\"older and \eqref{mfg:model-weighted-sobolev}
prevent weighted $L^2$ mass from escaping for bounded-energy sequences.
Core limits belong to $H^1(S^2)$ because the omitted point has zero
capacity. Conversely, extend spherical harmonics by logarithmic cutoffs
inside $B(q_i,d/8)$, at additional energy $O(1/\log(d/\ell))$.
The min--max principle gives convergence to the spectrum of two round spheres:
$0$ twice, $2$ six times, $6$ ten times, and so on.
The uniform Sobolev inequality gives
$L^{p_*}(d\mu)\hookrightarrow\mathcal H'$ for every $p_*>1$,
where $\|v\|_{\mathcal H}^2=\int|\nabla v|^2+\int v^2d\mu$.
In the eigenbasis, the complementary inverse has
$\mathcal H'\to\mathcal H$ norm bounded by
\[
 \sup_{\lambda_j\notin(1,3)}(1+\lambda_j)/|2-\lambda_j|\le C.
\]

For the $L^\infty$ bound, take first $p_*=6/5$ and $q=16$ in
\eqref{mfg:model-weighted-sobolev}. If
$-\Delta_{\widehat g}|u|\le C|u|+|f|$, set
$v=|u|+\|f\|_{6/5}$, with a positive regularization if necessary.
Testing this inequality with truncated powers $v^{\beta-1}$ gives
\[
 \begin{aligned}
  \|v^{\beta/2}\|_{16}^2\le C\beta^2\|v^{\beta/2}\|_{12}^2, \qquad
  \|v\|_{8\beta}\le(C\beta^2)^{21/\beta}\|v\|_\beta.
 \end{aligned}
\]
The second bound follows by interpolating $L^{12}$ between $L^2$ and
$L^{16}$ with exponent $20/21$. Iteration at $\beta=2\cdot8^j$
yields $\|u\|_\infty\le C(\|u\|_2+\|f\|_{6/5})$.
For general $p_*>1$, use $q>2p_*'$, with $p_*'$ the conjugate exponent. Kato's inequality covers complex
resolvents. Applying the bound first to the eigenfunctions controls
$\Pi_6$ on $L^{p_*}(d\mu)$, then gives \eqref{mfg:six-mode-estimate}
on its complement.
\end{proof}

\begin{proposition}[Relative dilation gap]\label{mfg:prop:relative-model-gap}
Put $\mathscr R=d/\ell$. Exactly one eigenvalue in $(1,3)$ satisfies
\begin{equation}\label{mfg:relative-model-gap}
 \lambda_{\rm rel}=2+\frac{3}{2\log\mathscr R}
 +O\left(\frac{\log\log\mathscr R}{(\log\mathscr R)^2}\right).
\end{equation}
Its core limits are opposite dilations $(Z_3,-Z_3)$; the other five
eigenvalues are $2+O(\mathscr R^{-1})$; denote their spectral projection
by $\Pi_5$. On its complement, \eqref{mfg:six-mode-estimate} holds with an extra factor $\log\mathscr R$.
Assume
\begin{equation}\label{mfg:five-mode-potential-test}
 \log\mathscr R\,\|V/H-1\|_{L^{p_*}(d\mu)}\to0.
\end{equation}
Let $Y_1,\ldots,Y_5$ be an $L^2(d\mu)$-orthonormal basis of
$\operatorname{ran}\Pi_5$. For every $f$ with
$2f/H\in L^{p_*}(d\mu)$, the projected problem
\[
 (\Delta_g+V)\phi=f+\sum_{\alpha=1}^5\kappa_\alpha HY_\alpha,
 \qquad \Pi_5\phi=0,
\]
has a unique solution $(\phi,\kappa)$, and
\eqref{mfg:six-mode-estimate} holds with right-hand side
$C_{p_*}\log\mathscr R\,\|2f/H\|_{L^{p_*}(d\mu)}$.
\end{proposition}
\begin{proof}
Write 
\[
 \begin{gathered}
  L=\log\mathscr R, \quad 
  M(u,v)=\int uv\,d\mu, \\
  Q(u,v)=\int\nabla u\cdot\nabla v\,dV_g-2M(u,v), \\
  \text{and}\quad \|v\|_{\mathcal H}^2=\int|\nabla v|^2+M(v,v).
 \end{gathered}
\]
Set $T=L^4$, $r=\ell T$, and
$\Omega=M\setminus(B(q_1,r)\cup B(q_2,r))$.
For boundary values $b_i+g_i$ on $\partial B(q_i,r)$, with zero-mean
$g_i$, the harmonic extension energy $\mathcal D$ on $\Omega$ satisfies
\begin{align}
 \mathcal D(b,0)&=\frac{\pi(b_1-b_2)^2}{\log(d/r)+O(1)},
                                      \label{mfg:mesoscopic-capacity}\\
 c\sum_i\|g_i\|_{\dot H^{1/2}}^2
 &\le\mathcal D(0,g)\le C\sum_i\|g_i\|_{\dot H^{1/2}}^2,
                                      \label{mfg:angular-dtn-bounds}\\
 |\mathcal D((b,0),(0,g))|
 &\le\frac{Cr/d}{\log(d/r)}|b_1-b_2|
       \left(\sum_i\|g_i\|_{\dot H^{1/2}}^2\right)^{1/2}.
                                      \label{mfg:angular-dtn-mixed}
\end{align}

All trace norms use unit-circle parametrizations. Set
\[
 \begin{gathered}
  \vartheta=r/d, \quad 
  L_r=\log(d/r), \\
  \text{and}\quad \mathscr E(v,w)=\int_\Omega\nabla v\cdot\nabla w\,dV_g.
 \end{gathered}
\]
For zero-mean traces $g_i$, trace--Poincar\'e on the disjoint annuli
$B(q_i,2r)\setminus B(q_i,r)$ gives the lower bound in
\eqref{mfg:angular-dtn-bounds}; extending their Fourier modes to zero
on the outer circles and minimizing the energy gives the upper bound.

For the constant traces, set $P=G(\cdot,q_1)-G(\cdot,q_2)$ and write
$P|_{\partial B(q_i,r)}=\beta_i+k_i$, where $k_i$ has zero mean.
The Green expansion yields
\[
 \begin{aligned}
  D:&=\beta_1-\beta_2=\pi^{-1}L_r+O(1), \quad 
  \|k\|_{\dot H^{1/2}}\le C\vartheta, \\
  r\partial_{\nu_\Omega}P&=\frac{\epsilon_i}{2\pi}+\tau_i, \quad 
  \|\tau_i\|_{\dot H^{-1/2}}\le C\vartheta,
 \end{aligned}
\]
where $\epsilon_1=1$, $\epsilon_2=-1$, and $\int\tau_i=0$.
The normal points into the deleted disks; its flux is conformally invariant. Removing each pole's logarithm leaves a gradient
$O(d^{-1})$, giving the error bounds.

Let $h_k$ and $h_g$ be harmonic extensions of $k$ and $g$.
The angular estimate and Green's identity give
\[
 \begin{gathered}
  \mathscr E(h_k,h_k)=O(\vartheta^2), \quad 
  \mathscr E(P,h_k)=O(\vartheta^2), \\
  \text{and}\quad \mathscr E(P,P)=D+O(\vartheta^2).
 \end{gathered}
\]
Up to a common constant, the extension of $b$ is
$h_b=(b_1-b_2)(P-h_k)/D$. Hence
\[
 \begin{aligned}
  \mathcal D(b,0)&=\frac{(b_1-b_2)^2}{D}
       +O(\vartheta^2D^{-2})(b_1-b_2)^2, \\
  |\mathscr E(h_b,h_g)|&\le
 C\frac{\vartheta}{L_r}|b_1-b_2|\|g\|_{\dot H^{1/2}}.
 \end{aligned}
\]
In the second bound, the constant flux pairs to zero with $g$, and
Cauchy--Schwarz controls $h_k$. This proves
\eqref{mfg:mesoscopic-capacity} and \eqref{mfg:angular-dtn-mixed}.
Completing the square and using angular coercivity shows that eliminating
the zero-mean traces changes the energy by at most
$C\vartheta^2L_r^{-2}(b_1-b_2)^2$.

Let 
\[
 \begin{gathered}
  Z_{i1},Z_{i2}=2\ell(z-q_i)_{1,2}/(\ell^2+|z-q_i|^2) \\
  \text{and}\quad Z_{i3}=(\ell^2-|z-q_i|^2)/(\ell^2+|z-q_i|^2).
 \end{gathered}
\]
Define $U_{i\alpha}$ as $Z_{i\alpha}$ on the $i$th disk, zero on the
other, and harmonic on $\Omega$. For $Uc=\sum c_{i\alpha}U_{i\alpha}$,
\begin{align}
 M(Uc,Uc)&=\frac{4\pi}{3}|c|^2+O(T^{-2})|c|^2,
                                      \label{mfg:bridge-gram}\\
 Q(Uc,Uc)&=\frac{\pi(c_{13}-c_{23})^2}{L-\log T+O(1)}
             +O(T^{-1})|c|^2.\label{mfg:bridge-form}
\end{align}
The core Gram matrix is $(4\pi/3)I$; dilation boundary values are
$-1+O(T^{-2})$, and translation traces have norm $O(T^{-1})$.
The core equation reduces the energy to boundary terms $O(T^{-2})$.
Mass from the other core and background is
$O(\mathscr R^{-4}T^2+\ell^2r^2)$; exterior mass is $O(T^{-2})$.
Equations \eqref{mfg:mesoscopic-capacity}--\eqref{mfg:angular-dtn-mixed}
then give \eqref{mfg:bridge-form}.

Radial Cauchy--Schwarz gives
\[
 |\operatorname{avg}_{\partial B(q_i,r)}v|
\le C\sqrt{1+\log T}\|v\|_{\mathcal H};
\]
 oscillatory traces are bounded by
$C\|v\|_{\mathcal H}$. The inner translation and dilation fluxes are
$O(T^{-1})$ and $O(T^{-2})$, while the principal exterior flux is
$\pi(c_{13}-c_{23})/[L-\log T+O(1)]$ times the difference of means.
Weighted $L^4$ controls the potential tails by
$CT^{-3/2}\|v\|_{\mathcal H}$. Thus
\[
 |Q(Uc,v)|\le C\frac{\sqrt{\log L}}{L}|c|\|v\|_{\mathcal H}.
\]
The complementary spectral gap yields
\[
 \|(I-\Pi_6)U_{i\alpha}\|_{\mathcal H}\le C\sqrt{\log L}/L.
\]
Since both forms are spectrally block diagonal, projection changes
their matrices by $O(\log L/L^2)$. The radial energy matrix is
$\frac\pi L\left(\begin{smallmatrix}1&-1\\-1&1\end{smallmatrix}\right)$,
with Gram matrix $(4\pi/3)I$, and hence
$\lambda_{\rm rel}=2+\frac{3}{2L}+O(\log L/L^2)$.

For the other modes, choose radial $\eta_i=1$ on $B(q_i,cd/2)$,
supported in $B(q_i,cd)$ with $|\nabla\eta_i|\le C/d$, and use
\begin{equation}\label{mfg:natural-five-trials}
 N_{i\alpha}=\eta_iZ_{i\alpha}\quad(\alpha=1,2),\qquad
 N_0=-1+\sum_i\eta_i(Z_{i3}+1).
\end{equation}
Their Gram matrix tends to $\operatorname{diag}(4\pi/3,4\pi/3,
4\pi/3,4\pi/3,8\pi/3)$, and
\begin{equation}\label{mfg:natural-five-residual}
 |Q(N_a,v)|\le C\mathscr R^{-1}\|v\|_{\mathcal H}.
\end{equation}

Translation cutoff defects are first angular harmonics of size
$O(\mathscr R^{-1})$, so annular Poincar\'e removes the circle mean.
For $N_0$, the cutoff quantity $Z_{i3}+1$ is $O(\mathscr R^{-2})$,
even after its radial trace loss $\sqrt{\log\mathscr R}$.
The potential tails contribute $O(\mathscr R^{-1}+\ell)
\|v\|_{\mathcal H}$, with $\ell\le C/\mathscr R$.
Applying $\Pi_6$ changes the trial functions by $O(\mathscr R^{-1})$.
Min--max for the square of the shifted operator on this cluster
therefore gives five eigenvalues $2+O(\mathscr R^{-1})$.
Pairing their residuals with the relative mode and dividing by its
$c/L$ gap also gives
\begin{equation}\label{mfg:natural-five-projection}
 \|\Pi_5-P_{\rm nat}\|_{L^2(d\mu)\to\mathcal H}
 \le C\frac{\log\mathscr R}{\mathscr R},
\end{equation}
where $P_{\rm nat}$ projects onto \eqref{mfg:natural-five-trials}.
The uniform eigenfunction bounds give the same $L^\infty$ estimate
for $\Pi_5v$ when $P_{\rm nat}v=0$.

The gap $c/L$ and the preceding $L^\infty$ iteration give the five-mode
inverse. Dividing $(\Delta_g+V)\phi=f$ by $H/2$ leaves the perturbation
$2(V/H-1):L^\infty\to L^{p_*}(d\mu)$; condition
\eqref{mfg:five-mode-potential-test} makes its composition with that
inverse a contraction. The projected equation has sources $HY_\alpha$, where $Y_\alpha$ is
an $L^2(d\mu)$-orthonormal frame of $\operatorname{ran}\Pi_5$.
\end{proof}

\subsection{The projected equation}
\label{mfg:rem:center-correction}
We first differentiate the inverse, then verify the required bounds
for the approximate solution.

\begin{lemma}[Parameter-dependent inverse]\label{mfg:lem:parameter-inverse}
In Proposition~\ref{mfg:prop:relative-model-gap}, put
$L=\log(d/\ell)$, $Y=L^{6/5}(d\mu)$ and
$\|h\|_X=\|h\|_\infty+\|h\|_{\mathcal H}$.
Let $B_0=2\Delta_g/H+2$ and $B=2(\Delta_g+V)/H$, where
$|V/H|\le C$ and $L\|V/H-1\|_Y=o(1)$.
Identify a local parameter family by composition with smooth
diffeomorphisms of uniformly bounded distortion, and denote the
pulled-back measure by $\nu_\theta$.
Assume that these identifications give locally uniformly equivalent
$X$ and $Y$ norms. For each fixed positive concentration scale, assume
that $B_0$ and $B$ are $C^1$ as operators from the common strong domain
$D=W^{2,6/5}(M)$ to $Y$, and that the density ratio of $\nu_\theta$
is $C^1$ in $L^\infty$.
Equip $D$ with the complete graph norm
$\|h\|_D=\|h\|_X+\|Bh\|_Y$.
Suppose a parameter direction $\mathscr D$ satisfies
\[
 \begin{aligned}
  \|\mathscr D\log\nu_\theta\|_\infty&\le a, \\
  \|\mathscr DB_0h\|_Y+\|\mathscr DBh\|_Y&\le Ca\|h\|_D, \\
  \max_\alpha\|\mathscr DN_\alpha\|_X&\le Ca,
 \end{aligned}
\]
where the trial functions in \eqref{mfg:natural-five-trials} are
$C^1$ in $X$. Project these functions onto $\operatorname{ran}\Pi_5$
and normalize their Gram matrix to obtain an orthonormal frame
$Y_\alpha$. Define
\[
 \begin{aligned}
  \mathscr A(h,\kappa)&=
 \left(Bh-\sum_{\alpha=1}^5\kappa_\alpha Y_\alpha,\,
       \left(\int_MhY_\alpha\,d\mu\right)_{\alpha=1}^5\right), \\
  SF&=\operatorname{pr}_1\mathscr A^{-1}(F,0).
 \end{aligned}
\]
Then, uniformly in the parameters,
\[
 \begin{aligned}
  \|S\|_{Y\to X}\le CL, \quad 
  \|\mathscr D\Pi_5\|_{Y\to X}\le CaL, \quad 
  \|\mathscr DS\|_{Y\to X}\le CaL^3.
 \end{aligned}
\]
In addition,
\[
 \max_\alpha\|\mathscr DY_\alpha\|_D+
 \max_\alpha\left\|\mathscr D
   \left[h\mapsto\int_MhY_\alpha\,d\mu\right]\right\|_{X\to\mathbb R}
 \le CaL.
\]
\end{lemma}
\begin{proof}

Elliptic regularity gives the common domain and complete graph norm at
each positive scale. The $B$ and $B_0$ graph norms are uniformly
equivalent since
$B-B_0=2(V/H-1):X\to Y$ is bounded.

At a base parameter $\theta_*$, fix the contour
$\Gamma=\{|z|=r/L(\theta_*)\}$ with $r>0$ small. It encloses the
five eigenvalues of $B_0$ of size $O((d/\ell)^{-1})$, excludes the
relative eigenvalue $-3/(2L)+O(\log L/L^2)$, and remains at distance
$c/L$ from the spectrum in a small parameter neighborhood.
The spectral theorem, $Y\hookrightarrow\mathcal H'$, the preceding
iteration estimate and $B_0h=zh-F$ give
$\sup_{z\in\Gamma}\|(z-B_0)^{-1}\|_{Y\to D}\le CL$.
This bound persists after composition. Differentiation along the fixed
contour of length $O(L^{-1})$ gives
\[
 \begin{aligned}
  \mathscr D\Pi_5&=\frac1{2\pi i}\int_\Gamma
 (z-B_0)^{-1}(\mathscr DB_0)(z-B_0)^{-1}\,dz, \qquad 
  \|\mathscr D\Pi_5\|_{Y\to D}&\le CaL.
 \end{aligned}
\]
The bounded eigenfunctions and their $O((d/\ell)^{-1})$ eigenvalues
also give $\|\Pi_5\|_{Y\to D}\le C$. Thus
$U_\alpha=\Pi_5N_\alpha$ has bounded $D$ norm and derivative $O(aL)$.
By \eqref{mfg:natural-five-projection}, its Gram matrix is uniformly
positive. Differentiating this matrix and its inverse square root,
including the measure derivative, gives the frame and constraint bounds.

For zero constraints, Proposition~\ref{mfg:prop:relative-model-gap} yields
$\|h\|_X\le CL\|F\|_Y$. Projection onto the frame bounds the
multipliers by
$C(\|F\|_Y+\|V/H-1\|_Y\|h\|_X)$, and the equation controls $Bh$.
Nonzero constraints $\eta$ are removed by subtracting
$\sum\eta_\alpha Y_\alpha$. Consequently
\[
 \begin{gathered}
  \|\mathscr A^{-1}\|_{Y\times\R^5\to D\times\R^5}\le CL \quad
  \text{and}\quad \|\mathscr D\mathscr A\|\le CaL.
 \end{gathered}
\]
 The common-domain identity
\[
 \mathscr D\mathscr A^{-1} =-\mathscr A^{-1}(\mathscr D\mathscr A)\mathscr A^{-1}
\]
now proves $\|\mathscr DS\|_{Y\to X}\le CaL^3$.
\end{proof}
\begin{lemma}[Material transport bounds]\label{mfg:lem:material-transport}
Use the parameters and compact set of
Subsection~\ref{mfg:normalization-ansatz}, set $\delta=\delta_0$, and
assume only $A_0\asymp b_0\sqrt{\log(1/b_0)}$ of the scale relations.
Let $V$ be a smooth parameter family with $|V/H|\le C$.
After pullback by suitable local diffeomorphisms, the trial functions,
measures, and operators of Lemma~\ref{mfg:lem:parameter-inverse} depend $C^1$
on the parameters. Denote derivatives after pullback by
$\mathscr D_t,\mathscr D_\zeta$ (material derivatives).
Their transported measure densities satisfy
\[
 \begin{gathered}
  \|\mathscr D_t\log\nu\|_\infty\le C\delta^4, \quad 
  \|\mathscr D_\zeta\log\nu\|_\infty\le C, \\
  \text{and}\quad \mathscr D_tN_\alpha=0, \quad 
  \|\mathscr D_\zeta N_\alpha\|_X\le C.
 \end{gathered}
\]
Their strong-operator derivatives satisfy
\[
 \begin{aligned}
 \|\mathscr D_tB_0\|_{D\to Y}
    &\le C\delta^4(1+\ell^{-1/3}),\\
 \|\mathscr D_\zeta B_0\|_{D\to Y}
    &\le C(1+R^{1/3}),\\
 \|\mathscr D_\theta B\|_{D\to Y}
    &\le \|\mathscr D_\theta B_0\|_{D\to Y}
       +2\|\mathscr D_\theta(V/H)\|_Y
       \qquad(\theta=t,\zeta).
 \end{aligned}
\]
In particular, $\delta^4\ell^{-1/3}=o(\delta^3)$ on this parameter set.
\end{lemma}
\begin{proof}

For $t$, use $T_m(z)=z+\xi(z)(m-m_*)$, with $\xi=1$ on a fixed
disk containing the cluster and supported in a larger disk inside
$\{\chi=1\}$. Its velocity is $O(\delta^4)$.
For $\zeta$, fix $(q_i^*,d_*,\ell_*)$, choose disjoint cutoffs $\xi_i$
supported in $B(q_i^*,cd_*)$, and use
\[
 T_\zeta(z)=z+\sum_i\xi_i(z)
 \left[q_i(\zeta)-q_i^*
 +\left(\frac{\ell(\zeta)}{\ell_*}-1\right)(z-q_i^*)\right].
\]
The cutoffs equal one near the supports of the localized terms in the trial functions.
These maps have uniformly bounded distortion, are similarities near
the cores, and fix the vortex neighborhood. Their velocity and first
two spatial derivatives are $O(d)$, $O(1)$ and $O(d^{-1})$.
On the disks where the maps are similarities, only the cutoffs in the trial functions vary, so
$\mathscr D_tN_\alpha=0$ and $\|\mathscr D_\zeta N_\alpha\|_X\le C$.

Write 
\[
 \begin{gathered}
  T_\theta^*d\mu_\theta=w_\theta dz \quad 
  \text{and}\quad A_\theta=(\det DT_\theta)DT_\theta^{-1}DT_\theta^{-T}.
 \end{gathered}
\]
Conformal invariance gives
\[
 \widetilde B_0h=w_\theta^{-1}\operatorname{div}(A_\theta\nabla h)+2h.
\]
At the base parameter,
\begin{equation}\label{mfg:transport-operator-identity}
 \mathscr DB_0h=-(\mathscr D\log w)(B_0-2)h
             +w^{-1}\operatorname{div}((\mathscr DA)\nabla h).
\end{equation}
On each such disk, the measure contributed by that core is fixed and
$A_\theta=I$. Differentiating the remaining densities gives
$\|\mathscr D_t\log w\|_\infty\le C\delta^4$ and
$\|\mathscr D_\zeta\log w\|_\infty\le C$.
Their positive sum preserves these bounds. The transports fix the outer
cutoff region. On the fixed transition region for $t$,
$|\mathscr D_tA|+|\nabla\mathscr D_tA|\le C\delta^4$; on the
transition regions for $\zeta$, $|\mathscr D_\zeta A|\le C$ and
$|\nabla\mathscr D_\zeta A|\le C/d$.

On a cutoff transition annulus of radius $r$, $H\asymp h_r=\ell^2/r^4$.
Apply the scaled interior $W^{2,p}$ estimate, $p=6/5$, to
$\Delta_gh=(H/2)Bh-Vh$ on a slightly enlarged annulus $A'_r$.
After division by $H$, the lower-order factors are
$h_r^{1/p}r^{2/p}=(\ell/r)^{5/3}\le C$ and
$h_r^{1/p-1}r^{2/p-2}=(r/\ell)^{1/3}$; hence

\begin{equation}\label{mfg:transport-graph-budget}
 \|H^{-1}(|D^2h|+r^{-1}|\nabla h|)\|_{Y(A_r)}
 \le C\{\|Bh\|_{Y(A'_r)}+(r/\ell)^{1/3}\|h\|_\infty\}.
\end{equation}
Combining \eqref{mfg:transport-graph-budget} with
\eqref{mfg:transport-operator-identity} gives the $B_0$ bounds;
differentiating $B-B_0=2(V/H-1)$ gives the bound for $B$.
Smooth coefficients verify the common-domain $C^1$ hypotheses, and
\[
 \delta^4\ell^{-1/3}=O(\delta^{11/3}R^{1/3})=o(\delta^3).
\]
\end{proof}

\begin{lemma}[Refined odd correction]\label{mfg:lem:refined-odd-correction}
On the translated matched family and parameter set of
Subsection~\ref{mfg:normalization-ansatz}, let
$P(z)=(z-m)^2-s^2$ and $v=\log(32b^2/(b^2+|P|^2)^2)$. Set
\[
 J_s(x)=P_3(x)+8\pi D_xD_y^2R(0,0)[x;s,s],\qquad
 j_s=J_s(s),\qquad c_s=j_s+2m\cdot s/A.
\]
There is a $C^1$ projected family with five zero-integral source terms
whose full regular-part correction satisfies, in $X$,
\begin{equation}\label{mfg:refined-odd-correction}
 \phi=-\sum_{i=\pm}\chi_i\operatorname{sign}_i c_s(1+Z_{3,i})
                  +o_{C_t^1}(\delta_0^3).
\end{equation}
Here $\operatorname{sign}_\pm=\pm1$, and $\chi_i$ are identical radial
translates, equal to one near $q_i$ and supported in $B(q_i,cd)$, $c<1/8$;
$Z_{3,i}=(\ell^2-|z-q_i|^2)/(\ell^2+|z-q_i|^2)$ and
$\ell=b/(2\sqrt A)$. The energy change $G_\eps$ from adding $\phi$ satisfies
\begin{align}
 G_\eps&=8\pi c_s^2+o_{C_t^1}(\delta_0^6),
                                      \label{mfg:odd-energy-correction}\\
 \|G_\eps\|_{C_\zeta^1}&=o(A_0^2),\qquad
 \zeta=(A/A_0,b/b_0,\arg s).\label{mfg:odd-main-energy-small}
\end{align}
For functions, $C_t^1$ uses the material translations of
Lemma~\ref{mfg:lem:material-transport}; scalar derivatives are ordinary.
The errors are uniform at $c_s=0$.
\end{lemma}
\begin{proof}
Put $\delta=\delta_0$, $R=d/\ell$, $L_R=\log R$, and
$p_*=6/5$. On the fixed parameter set,
\[
 \begin{gathered}
  A\asymp\delta^2, \quad
  b\asymp\delta^2/R, \quad
  \ell\asymp\delta/R, \\
  \eps^2A/b^2\asymp\delta^4, \quad 
  R^2\asymp\log(1/b_0).
 \end{gathered}
\]
Thus $\delta^\alpha R^kL_R^j\to0$ for every $\alpha>0$ and fixed
$k,j$. Use $X,Y$ and the material coordinates from
Lemmas~\ref{mfg:lem:parameter-inverse}--\ref{mfg:lem:material-transport}.

\emph{Step 1: the density and exact normalization.}
Use $x,v_P,H_P$ from Subsection~\ref{mfg:shared-inputs}. Choose the
unnormalized representative $V_0$ so that, in the inner disk,
\begin{equation}\label{mfg:local-raw-density}
 e^{u_0+V_0}\,dV_g=|x+m|^2e^{v_P(x)}e^{f_m(x)}\,dx,
 \qquad f_m(0)=0.
\end{equation}
Here $f_m(x)$ is the difference between the values at $x+m$ and $m$ of
$\sigma+r_0+8\pi\{R(\,\cdot\,,m+s)+R(\,\cdot\,,m-s)\}$.
Taylor expansion in both Green-function variables gives
\begin{equation}
 \label{mfg:incident-jet-budget}
 \begin{aligned}
  f_m(x)&=J_s(x)+\mathcal E_m(x), \\
  |\mathcal E_m(x)|&\le C\{\varrho^4+A\varrho^2+(A^2+|m|)\varrho\}, \\
  \varrho&=|x|.
 \end{aligned}
\end{equation}

This bound holds for scaled $A,b,\arg s$ derivatives at fixed $x$, and
$|\mathscr D_tf_m|\le C\delta^4\varrho$: the vanishing jets remove
orders at most two, the even Green-pole expansion gives the mixed
cubic term in $J_s$, and translation contributes $O(|m|\varrho)$.

Set $I_j=\int_Me^{j(u_0+V_0)}dV_g$. For this representative,
$v_0=V_0+\log\Lambda$, with $\Lambda$ the smaller root in
\eqref{mfg:scalar-mass}, with $a_w,b_w$ replaced by $I_1,I_2$.
With a centered cutoff, $\int H_PJ_s$ and
$\int2m\cdot x\,e^{v_P}$ vanish by parity; replacing it by the fixed
cutoff introduces an error $O_{C_t^1}(\delta^4b^2)$.
The contributions to $I_1-\rho$ are $O(\delta^4)$ from the incident
remainder, $O(\delta^6+b^2)$ from nonlinear and mixed weights, and
$O(b^2)$ from the exterior. A $t$ derivative fixes the profile and
preserves the leading parity; differentiating $|x+m|^4$ gives the corresponding estimate for $I_2$.
The moment estimates give
\[
 \begin{gathered}
  I_1=\rho+O_{C_t^1}(\delta^4), \quad 
  \eps^2I_2=O_{C_t^1}(\delta^4), \\
  |\partial_\zeta I_1|+\eps^2|\partial_\zeta I_2|\le CR\delta^4.
 \end{gathered}
\]
For the $\zeta$ estimate, differentiate the parity identities first.
The remainder estimates acquire at most a factor $R$ from $|\partial_\zeta v_P|\le CR$;
incident and annular terms retain their bounds. The root formula gives
\begin{equation}\label{mfg:raw-root-derivatives}
 \Lambda=1+O_{C_t^1}(\delta^4),\qquad
 |\partial_\zeta\Lambda|\le CR\delta^4.
\end{equation}

\emph{Step 2: residual and potential estimates.}
Write $d\mu=D(z)\,dz$ in the conformal chart. Thus
$H=2e^{-\sigma}D$. Split the chart into the cores $|x\mp s|<cd$, the rest of
$|x|<C\delta$, the region $C\delta<|x|<r_0$, and the matching
annulus. On these regions, respectively,
\[
 \begin{aligned}
  D&\asymp \frac{\ell^{-2}}{(1+|y|^2)^2}, \quad
  D\asymp\ell^2\delta^{-4}, \\
  D&\asymp\ell^2\varrho^{-4}, \quad
  D\asymp\ell^2,
 \end{aligned}
\]
where $x=\pm s+\ell y$ in a core. The remaining part of $M$ has
$d\mu\asymp\ell^2dV_g$. The non-core part of $|x|<C\delta$ has
$\mu$-measure $O(R^{-2})$.

On a core the exact quotient of $H_P$ by the corresponding flat
Liouville density $8\ell^{-2}(1+|y|^2)^{-2}$ is
\[
 \left|1+\frac{\ell y}{\pm s}\right|^2
 \frac{(1+|y|^2)^2}
 {\left(1+|y|^2\left|1+\frac{\ell y}{\pm2s}\right|^2\right)^2}.
\]
This quotient is uniformly positive and bounded, with error
$O(|y|/R)$ from one. Its material $\zeta$ derivative obeys the same
bound, since $y$ is fixed and $\partial_\zeta(\ell/s)=O(R^{-1})$.
The other core and background contribute $O(R^{-5/3}+\ell^{5/3})$ in
$Y$. Since $\int |y|^{6/5}(1+|y|^2)^{-2}\,dy<\infty$, the total
core error in $Y$ is $O(R^{-1})$.

Set $\mathfrak q=e^{u_0+v_0}$ and
$a=\mathfrak q/H$, $c=\eps^2\mathfrak q^2/H$.
The explicit profile gives
$\|\mathfrak q\|_\infty\le CA/b^2$ and
$\eps^2\|\mathfrak q\|_\infty\le C\delta^4$.
The core quotient, \eqref{mfg:local-raw-density}, and
\eqref{mfg:raw-root-derivatives} give
\begin{equation}
 \label{mfg:potential-zero-budget}
 \begin{aligned}
  |a|&\le C, \quad 
  |c|\le C\delta^4, \\
  \|V/H-1\|_Y&\le C(R^{-1}+\delta^3+\ell^{5/3}), \quad
  V/H=a-2c.
 \end{aligned}
\end{equation}
The non-core region contributes $O(R^{-5/3})$ to this norm; the
annulus and exterior contribute $O(\ell^{5/3})$.

The material $t_j$ derivative of $|x+m|^2$ is
$2\delta^4e_j\cdot(x+m)$.
Equations \eqref{mfg:incident-jet-budget} and
\eqref{mfg:raw-root-derivatives} therefore bound $\mathscr D_ta$ by
$C\delta^3$ in the cluster, $C\delta^6\varrho^{-3}$ in
$C\delta<\varrho<r_0$, and $C\delta^6$ on the fixed annulus and exterior.
The middle region has measure $O(R^{-2})$; the outer contribution
to its $6/5$ power is bounded by
\[
 C\delta^{36/5}\ell^2\int_{C\delta}^{r_0}\varrho^{-33/5}d\varrho \le C\delta^{18/5}R^{-2}.
\]
Thus $\|\mathscr D_ta\|_Y\le C\delta^3$.
For $\zeta$, the fixed core variable $y$, the velocity $O(d)$ of the coordinate change,
and stationary exterior give $\|\mathscr D_\zeta a\|_Y\le C$.
The pulled-back logarithmic derivatives of $H$ are $O(\delta^4)$
and $O(1)$, respectively. Using
\[
 \mathscr Dc=2\eps^2\mathfrak q\,\mathscr Da+c\,\mathscr D\log H
\]
 gives
\begin{equation}
 \label{mfg:material-potential-budget}
 \begin{aligned}
  \|\mathscr D_t(V/H)\|_Y&\le C\delta^3, \quad 
  \|\mathscr D_\zeta(V/H)\|_Y&\le C, \\
  \|\mathscr D_tc\|_Y&\le C\delta^7.
 \end{aligned}
\end{equation}

Together with Lemmas~\ref{mfg:lem:material-transport} and
\ref{mfg:lem:parameter-inverse}, this gives
\begin{align}
 \|S\|_{Y\to X}&\le CL_R,\label{mfg:transport-inverse-zero}\\
 \|\mathscr D_tS\|_{Y\to X}+\|\mathscr D_t\Pi_5\|_{Y\to X}
 &\le C\delta^3L_R^4,\label{mfg:transport-inverse-t}\\
 \|\mathscr D_\zeta S\|_{Y\to X}+\|\mathscr D_\zeta\Pi_5\|_{Y\to X}
 &\le CR^2L_R^4.\label{mfg:transport-inverse-zeta}
\end{align}
Indeed, $\delta^4(1+\ell^{-1/3})=o(\delta^3)$ and
$1+R^{1/3}\le CR^2$.

In this proof, $\mathcal R_0$ denotes the residual of $v_0$ at the
fixed positive $\eps$, rather than the $\eps=0$ residual. Let
$H_i=8e^{-\sigma}\chi\ell^2/(\ell^2+|z-q_i|^2)^2$ be the two
bubble parts of $H$. Define
$r=2\{\mathcal R_0-c_s(H_+-H_-)\}/H$, $E_R=\delta^3R^{-1}+\delta^4R^4+b_0^{5/3}$.
We claim
\begin{equation}\label{mfg:explicit-odd-remainder}
 \|r\|_Y+\|\mathscr D_t r\|_Y\le CE_R,
 \qquad \|\mathscr D_\zeta r\|_Y\le CE_RR^2.
\end{equation}
In the inner disk, the residual is exactly
\begin{equation}\label{mfg:exact-local-residual-density}
 \mathcal R_0\,dV_g=
 \left[-H_P+\Lambda|x+m|^2e^{v_P+f_m}
 -\eps^2\Lambda^2|x+m|^4e^{2v_P+2f_m-\sigma(x+m)}\right]dx.
\end{equation}
On a core the leading incident coefficient is
$J_s(x)+2m\cdot x/|x|^2$. At $x=\pm s$ its value is
$\pm c_s$, and its variation is $O(\delta^3|y|/R)$.
A $t_j$ derivative replaces $m$ by $\delta^4e_j$ and preserves this
bound. After division by $D$, the quadratic weight term is
$O(\delta^6)$; the incident remainder, scalar-root correction, and
saturation contribute $O(\delta^4)$ in $Y$. The core quotient and
\eqref{mfg:exact-local-residual-density} give
\[
\begin{array}{c|c}
 \text{region}&\text{bound in }Y\text{ for }r\text{ and }\mathscr D_t r\\ \hline
 \text{two cores}&C(\delta^3/R+\delta^4)\\
 |x|<C\delta\text{ outside the cores}&C(\delta^3+\delta^4)R^{-5/3}\\
 C\delta<|x|<r_0&C(\delta^3R^{-5/3}+\delta^{11/3}R^{-5/3})\\
 \text{matching annulus and exterior}&C\delta^{11/3}R^{-5/3}
\end{array}
\]

On $C\delta<|x|<r_0$, $H_P/D=O(\delta^2/\varrho^2)$ and
$J_s=O(\varrho^3+\delta^2\varrho)$, so the cubic contribution has
$6/5$ power at most
\[
 C\delta^{12/5}\ell^2\int_{C\delta}^{r_0}\varrho^{-9/5}d\varrho
\le C\delta^{18/5}R^{-2}.
\]
The quartic remainder is $O(\delta^2\varrho^2)$ after division by
$D$, giving norm $C\delta^{11/3}R^{-5/3}$.
The potential estimate controls the translation-weight derivative.
The annular matching defect contributes
$Cb^2\ell^{-1/3}=C\delta^{11/3}R^{-5/3}$, as does the exterior.

For $\zeta$, subtracting $\pm c_s$ before differentiation preserves
the core error $O(\delta^3|y|/R)$. Derivatives in the intermediate and cutoff transition regions
are bounded; differentiating the scalar root and the polynomial profile introduces at most a factor $R$ in each case.
Thus each bound in the table acquires at most a factor $CR^2$, including division by $D$. Since $\delta^{11/3}R^{-5/3}\le Cb_0^{5/3}$, this proves
\eqref{mfg:explicit-odd-remainder} and $L_RE_R=o(\delta^3)$.

\emph{Step 3: the particular solution and contraction.}
Project the five trial functions and use their Gram matrix to obtain
the orthonormal frame $Y_\alpha$, and put $K_\alpha=HY_\alpha$. Their $g$-integrals vanish by spectral orthogonality to constants in
$L^2(d\mu)$. We solve \eqref{mfg:projected-equation}.
Set
\[
 \begin{aligned}
  \psi&=-\sum_{i=\pm}\chi_i\operatorname{sign}_i c_s(1+Z_{3,i}), \\
  Q&=I-\Pi_5, \\
  F_{\rm odd}&=2c_s(H_+-H_-)/H.
 \end{aligned}
\]
On the round sphere, $(\Delta_*+2)(1+Z_3)=2$ cancels each core's
leading source. On each cutoff annulus $|z-q_i|\asymp d$, the function
$1+Z_{3,i}$ and its scaled derivatives are $O(R^{-2})$.
The cutoff defect after division by $H/2$ is $O(\delta^3)$ on a set
of $\mu$-measure $O(R^{-2})$. Other-core and exterior terms are no
larger, so
\begin{equation}\label{mfg:odd-particular-defect}
 \|B_0\psi+F_{\rm odd}\|_Y+
 \|\mathscr D_t(B_0\psi+F_{\rm odd})\|_Y
 \le C\delta^3R^{-5/3}.
\end{equation}

For the $t$ derivative, the Dirichlet matrix is fixed on
$\operatorname{supp}\psi$ in the translation region.
The measure change adds $O(\delta^7)$, absorbed in
\eqref{mfg:odd-particular-defect}; $\psi$ vanishes on the transport annulus.

The translation overlaps of $\psi$ vanish by radial parity, and the
two common-dilation overlaps cancel by opposite signs. Other-core
weights and cutoff tails contribute $O(\delta^3/R)$.
With $\eta_R=CL_R/R$, \eqref{mfg:natural-five-projection} and the
bounded eigenfunctions imply
\begin{equation}\label{mfg:projected-particular-budget}
 \|\Pi_5\psi\|_X+\|\mathscr D_t(\Pi_5\psi)\|_X
 \le C\delta^3(R^{-1}+\eta_R+\delta^3L_R^4).
\end{equation}
The overlap bound persists after differentiation; $(\mathscr D_t\Pi_5)\psi$ gives the last term. Thus $\psi_c=Q\psi$ satisfies
\[
 \|\psi_c\|_X+\|\mathscr D_t\psi_c\|_X\le C\delta^3.
\]

Define the divided nonlinear remainder
\[
 \widehat{\mathcal N}(\varphi) =2a(e^\varphi-1-\varphi) -2c(e^{2\varphi}-1-2\varphi).
\]
For $\|\varphi_j\|_\infty\le1$, the bounded total measure and
\eqref{mfg:potential-zero-budget} give
\begin{equation}\label{mfg:nonlinear-material-budget}
 \begin{split}
 \|\widehat{\mathcal N}(\varphi_1)
       -\widehat{\mathcal N}(\varphi_2)\|_Y
 &\le C(\|\varphi_1\|_\infty+\|\varphi_2\|_\infty)
             \|\varphi_1-\varphi_2\|_X,\\
 \| (\mathscr D_t\widehat{\mathcal N})(\varphi)\|_Y
 &\le C\delta^3\|\varphi\|_\infty^2,\qquad
 \| (\mathscr D_\zeta\widehat{\mathcal N})(\varphi)\|_Y
 \le C\|\varphi\|_\infty^2.
 \end{split}
\end{equation}
The second line follows from \eqref{mfg:material-potential-budget}
by differentiating the coefficients.

By the definition of the bordered inverse, $S\Pi_5=0$ and
$S(B\psi_c)=\psi_c$. Since $B_0$ commutes with $\Pi_5$, the equation
for $h=\phi-\psi_c$ is exactly
\begin{equation}\label{mfg:refined-fixed-point}
 h=-S\{g+\widehat{\mathcal N}(\psi_c+h)\},\qquad
 g=(B_0\psi+F_{\rm odd})+r+2(V/H-1)\psi_c.
\end{equation}

Equations \eqref{mfg:explicit-odd-remainder}--\eqref{mfg:projected-particular-budget}
and the potential bounds give $\|g\|_Y+\|\mathscr D_tg\|_Y\le CE_R$.
For example, the derivative of the last product in $g$ is bounded by
$C\delta^6+C(R^{-1}+\delta^3+\ell^{5/3})\delta^3$, absorbed by $E_R$.
The map preserves $\|h\|_X\le C_0L_RE_R$ for large fixed $C_0$:
its nonlinear term is $O(L_R\delta^6)=o(L_RE_R)$ and its Lipschitz
constant is $CL_R\delta^3=o(1)$.

Differentiating the contraction and using
\eqref{mfg:transport-inverse-t} gives
\[
 \|\mathscr D_th\|_X\le C\delta^3L_R^4(E_R+\delta^6)
 +CL_R(E_R+\delta^6)+CL_R\delta^3\|\mathscr D_th\|_X.
\]
Absorption yields $\|h\|_X+\|\mathscr D_th\|_X\le CL_RE_R$.
Together with \eqref{mfg:projected-particular-budget}, this gives
\begin{equation}\label{mfg:correction-material-budget}
 \begin{split}
 \|\phi-\psi\|_X+\|\mathscr D_t(\phi-\psi)\|_X
 &\le C\delta^3\omega_R,\\
 \omega_R&=\eta_R+\frac{L_R^2}{R}+\delta R^4L_R
       +\frac{b_0^{5/3}L_R}{\delta^3}+\delta^3L_R^5=o(1).
 \end{split}
\end{equation}
No step divides by $c_s$, so the bound is uniform at $c_s=0$. For $\zeta$, differentiate the equivalent equation
\[
 \phi=-S\{F_{\rm odd}+r+\widehat{\mathcal N}(\phi)\}.
\]
The source has norm $O(\delta^3)$ and material derivative
$O(\delta^3R^2)$. Using \eqref{mfg:transport-inverse-zeta} and
absorbing $CL_R\delta^3\|\mathscr D_\zeta\phi\|_X$ gives
\begin{equation}\label{mfg:correction-zeta-budget}
 \|\mathscr D_\zeta\phi\|_X\le C\delta^3R^2L_R^4.
\end{equation}

The zero-integral sources enforce the exact mass relation. Since
$\|\phi\|_\infty=O(\delta^3)$ and
$\eps^2\|\mathfrak q\|_\infty=O(\delta^4)$, its constant stays
on the separated smaller-root branch. The full action therefore
equals the normalized energy.

\emph{Step 4: the energy and its derivatives.}
Let $\nu_j=T_t^*(\mathfrak q^j dV_g)$, including the Jacobian.
The preceding moments and direct differentiation of the polynomial
weights give
\begin{equation}
 \label{mfg:odd-density-variation}
 \begin{aligned}
  \nu_1(M)+\eps^2\nu_2(M)&\le C, \\
  \|\partial_t\nu_1\|_{\rm TV}
       +\eps^2\|\partial_t\nu_2\|_{\rm TV}&\le C\delta^3.
 \end{aligned}
\end{equation}

The differentiated weights contribute at most
\[
 \begin{gathered}
  C\delta^4\int(|x|+|m|)e^{v_P}=O(\delta^3) \\
  \text{and}\quad C\eps^2\delta^4\int(|x|^3+|m|^3)e^{2v_P}=O(\delta^7).
 \end{gathered}
\]
Derivatives of the scalar root and smooth coefficients contribute $O(\delta^4)$ and
$O(\delta^8)$; annular and Jacobian terms contribute $O(\delta^4b^2)$.

Put 
\[
 \begin{gathered}
  f_1(h)=\tfrac12h(e^h-1-h)-(e^h-1-h-h^2/2) \\
  \text{and}\quad f_2(h)=-\tfrac12h(e^{2h}-1-2h)+\tfrac12(e^{2h}-1-2h-2h^2).
 \end{gathered}
\]

For $|h|\le1$, $|f_j(h)|\le C|h|^3$ and
$|f_j'(h)|\le C|h|^2$. Expanding the full action and testing the projected equation with
$\phi$ gives the exact identity
\begin{equation}\label{mfg:odd-action-exact-remainder}
 G_\eps=-\frac12\int_M\mathcal R_0\phi\,dV_g
       +\int_M\mathfrak q f_1(\phi)\,dV_g
       +\eps^2\int_M\mathfrak q^2f_2(\phi)\,dV_g.
\end{equation}

Orthogonality removes the multiplier term. Pulling back the identity
and using \eqref{mfg:odd-density-variation} bounds the nonlinear
integrals by $O_{C_t^1}(\delta^9)$.
For the first term, write
\[
 \int_M\mathcal R_0\phi\,dV_g
 =c_s\int_M(H_+-H_-)\phi\,dV_g+\int_Mr\phi\,d\mu.
\]
The last pairing is $O_{C_t^1}(\delta^3E_R)$ by
$X\hookrightarrow L^6(d\mu)$ and the measure derivative bound.
In the first pairing, replacing $\phi$ by $\psi$ introduces an error
$O_{C_t^1}(\delta^6\omega_R)$. Radial integration gives
\[
 \begin{aligned}
  \int H_i\chi_i(1+Z_{3,i})\,dV_g&=8\pi+O(R^{-4}), \\
  \int H_i\chi_j(1+Z_{3,j})\,dV_g&=O(R^{-4}\log R)\quad(i\ne j).
 \end{aligned}
\]
Both coefficients have zero material $t$ derivative because their
supports translate within $\{\chi=1\}$. Thus

\begin{equation}\label{mfg:odd-energy-C1-budget}
 \|G_\eps-8\pi c_s^2\|_{C_t^1}
 \le C\{\delta^6\omega_R+\delta^3E_R
                +\delta^6R^{-4}\log R+\delta^9\}
 =o(\delta^6).
\end{equation}

For $\zeta$, the measure has bounded logarithmic derivative, and
$|\mathscr D_\zeta(H_i dV_g)|\le C H_i dV_g$ after pullback.
For the normalized densities,
\[
 \|\mathscr D_\zeta(\mathfrak q\,dV_g)\|_{\rm TV} +\eps^2\|\mathscr D_\zeta(\mathfrak q^2dV_g)\|_{\rm TV} \le CR.
\]
Parameter derivatives of the polynomial profile are $O(R)$, and transport velocities and gradients
are $O(d)$ and $O(1)$;
$\int|\nabla\mathfrak q|\le C/\ell$ and
$\eps^2\int|\nabla(\mathfrak q^2)|\le C/\ell$ follow from the
core moments and annular estimates. Differentiating
\eqref{mfg:odd-action-exact-remainder} and using
\eqref{mfg:explicit-odd-remainder} and
\eqref{mfg:correction-zeta-budget} gives
\[
 \|G_\eps\|_{C_\zeta^1}
 \le C\{\delta^6R^2L_R^4
       +\delta^3E_RR^2L_R^4
       +\delta^9(R+R^2L_R^4)\}
 \le C\delta^6R^2L_R^4=o(\delta^4).
\]
\end{proof}
\subsection{Expansion of the reduced energy}
We expand the energy of the approximate solution and add the correction
from Lemma~\ref{mfg:lem:refined-odd-correction}.

\begin{proposition}[Point interaction]\label{mfg:prop:global-quartic}
Assume $n=1$, $\rho=16\pi$, $K_g(p)=0$, and
$\nabla\mathcal B_p(p)=D^2\mathcal B_p(p)=0$. In the normalized chart,
write 
\[
 \begin{gathered}
  \mathscr F=\sigma+\mathcal B_p=\mathscr F(0)+P_3+P_4+O(|z|^5) \\
  \text{and}\quad R(x,y)=G(x,y)+(2\pi)^{-1}\log|x-y|.
 \end{gathered}
\]
 Define
\[
 \begin{aligned}
  W(q_+,q_-)&=\sum_{j=\pm}\{u_0(q_j)+4\pi R_g(q_j,q_j)\}
               +8\pi G(q_+,q_-), \\
  R_g(q,q)&=R(q,q)+\sigma(q)/(4\pi).
 \end{aligned}
\]
For $s=\delta e$, $|e|=1$,
\begin{equation}
 \label{mfg:global-effective-quartic}
 \begin{aligned}
  W(s,-s)&=C+\delta^4Q_{\rm eff}(e)+O(\delta^6), \\
  Q_{\rm eff}(e)&=2P_4(e)+4\pi D_x^2D_y^2R(0,0)[e,e;e,e].
 \end{aligned}
\end{equation}

The remainder's $s$ derivative is $O(\delta^5)$.

\end{proposition}
\begin{proof}
Using $r_0+\rho R=\mathcal B_p$, decompose
\begin{align*}
 W={}&2\log|q_+|+2\log|q_-|-4\log|q_+-q_-|
       +\mathscr F(q_+)+\mathscr F(q_-)+4\pi T,\\
 T={}&R(q_+,q_+)+R(q_-,q_-)+2R(q_+,q_-)
                         -4R(q_+,0)-4R(q_-,0).
\end{align*}
At $(s,-s)$ the logarithms are constant. Green symmetry cancels the
pure-variable and mixed $(1,1)$ terms, and parity cancels odd orders:
\[
 T(s,-s)=-4R(0,0)+D_x^2D_y^2R(0,0)[s,s;s,s]+O(|s|^6),
\]
with $s$ derivative $O(|s|^5)$ for the remainder. At $q_\pm=m\pm s$,
the singular part is $-4\log2+2\log|1-(m/s)^2|$, identifying the chart
with $\C$.
\end{proof}

\begin{proposition}[Energy of the approximate solution]\label{mfg:prop:bare-global-energy}
Under the hypotheses of Proposition~\ref{mfg:prop:global-quartic},
the centered matched family and parameter set of
Subsection~\ref{mfg:normalization-ansatz} satisfy
\begin{equation}\label{mfg:bare-global-energy}
 \begin{split}
 \mathcal E_\eps(w_{b,a})&=C_\eps-8\pi Q_{\rm eff}(e)A^2
 +\pi\Delta_gK_g(p)b^2\log(1/b)\\
 &+\frac{256\pi}{3}\frac{\eps^2A}{b^2}+o_{C^1_\zeta}(A_0^2).
 \end{split}
\end{equation}
Here $C_\eps$ is independent of $\zeta$; this is the energy before
adding $\phi$.
\end{proposition}
\begin{proof}
Write $H=H_{b,a}$ and let $V$ be the approximate solution before mean
subtraction. Its annular difference from the outer expression is
$d=-2\chi\log(1+b^2/|z^2-a|^2)$, so
$\|d\|_{C^k}+\sum_j\|D_jd\|_{C^k}\le C_kb^2$.
The nonconstant parts of $D_1V,D_3V$ are $O(A)$ and that of $D_2V$
is $O(b^2)$ there. Thus annular and exterior energy-gradient errors are
$O(b^2A+b^4)$; constant tangents vanish against the normalized residual.

Put $f=\mathscr F_q-\mathscr F_q(0)$ and
$\mathcal C[z,z;s,s]=D_x^2D_y^2R(0,0)[z,z;s,s]$.
Taylor expansion, valid through two spatial derivatives and one scaled
parameter derivative, gives
\[
 \begin{gathered}
  f_{\rm even}=P_4+4\pi\mathcal C[z,z;s,s]
+O(|z|^6+A|z|^4+A^2|z|^2) \\
  \text{and}\quad f_{\rm odd}=O(|z|^3+A|z|).
 \end{gathered}
\]
Since $\Delta_xR$ is independent of $y$, $\mathcal C[\cdot,\cdot;s,s]$
is harmonic. Green symmetry interchanges its variable pairs. With
$\mathcal T(s)=\mathcal C[s,s;s,s]$, this gives
\begin{align}
 M:=\int\chi Hf
 &=16\pi P_4(s)+64\pi^2\mathcal T(s)+32\pi q_0b^2\log(1/b)\notag\\
 &\quad+O_{1,*}\!\left(A^3+b^2(1+\log R_*)+Ab^2\log(1/b)\right),
                                      \label{mfg:bare-M-budget}\\
 \int\chi H D_jf
 &=32\pi^2D_j\mathcal T(s)
   +O\!\left(A^3+b^2+Ab^2\log(1/b)\right).
                                      \label{mfg:bare-half-budget}
\end{align}
Parity eliminates odd terms, while $\int H z^2=16\pi a$ evaluates the
harmonic quadratic contribution, with cutoff error $O_{1,*}(b^2)$.
In \eqref{mfg:bare-half-budget}, only the second variable pair is
differentiated, giving half the diagonal derivative and no trace term.

\emph{Energy gradient.}
Set $I_j=\int_Me^{j(u_0+V)}\,dV_g$. Mean subtraction multiplies $I_j$ by $e^{-j\overline V}$ and the
scalar root by $e^{\overline V}$, preserving the normalized density
and $I_2/I_1^2$. At $\eps=0$ the energy expressed in terms of this representative is
$\frac12\int|\nabla V|^2+(\rho/A_M)\int V-\rho\log I_1$, up to
a constant, with variation $-\int_M\mathcal R_0 D_jV\,dV_g$.
Inside the cutoff disk,
$e^{u_0+V}dV_g=He^f dz$ and
$(-\Delta_gV+\rho/A_M)dV_g=Hdz$.
The moment bounds give
$I_1=\rho+M+O(A^3+b^2)$ and $\|I_1-\rho\|_{1,*}\le CA_0^2$;
here $\int\chi H f^2\le C(A^3+b^2+Ab^2\log(1/b))$ and
$Ab^2\log(1/b)=O(b^2)$. Thus
\[
 \begin{aligned}
  \mathcal R_0(w_{b,a})dV_g&=\chi H(f-M/\rho)dz+\mathfrak e\,dV_g, \\
  \|\mathfrak e\|_1&\le C(A^3+b^2+A^4).
 \end{aligned}
\]
Pair the residual with the parameter derivatives of the polynomial profile, using
$\|D_jv_{b,a}\|_\infty\le CR_*$. The smooth nonconstant Green
tangent is $O(A|z|)$, contributing $O(A^3+Ab^2\log(1/b))$.
Also $\int H D_jv_{b,a}=0$ by the exact planar mass. Consequently
$D_j\mathcal E_0=-D_jM+\int\chi H D_jf+\mathfrak r_j$, where
\[
 |\mathfrak r_j|\le C[R_*(A^3+b^2+A^4)+A^3+Ab^2\log(1/b)+b^2A+b^4].
\]
Equations \eqref{mfg:bare-M-budget}--\eqref{mfg:bare-half-budget}
therefore give the uniform gradient remainder
\begin{equation}\label{mfg:bare-C1-quantitative-budget}
 \mathfrak B_0=C[\sqrt{L_0}(A_0^3+b_0^2)
                  +b_0^2(1+\log L_0)+A_0b_0^2L_0]=o(A_0^2).
\end{equation}
Using 
\[
 \begin{gathered}
  q_0=-\Delta_gK_g(p)/32 \\
  \text{and}\quad 16\pi P_4(s)+32\pi^2\mathcal T(s)=8\pi Q_{\rm eff}(e)A^2
 \end{gathered}
\]
gives the mean-field gradient. Integration from a fixed parameter gives the same error for the value.

\emph{Saturation.}
Apply \eqref{mfg:shared-square-density} to the actual density
$H^2e^{2f-\sigma}$. Since $K_g(p)=0$, the
normalized metric has no linear or quadratic terms. The even part of
$e^{2f-\sigma}-1$ and its scaled derivatives are bounded by
$C(|z|^4+A|z|^2+|z|^6+A^2|z|^2)$, while odd terms integrate to zero.
The fourth-power kernel and the $O_{1,*}(b^4)$ exterior error yield
\[
 \begin{aligned}
 I_2&=\frac{512\pi A}{3b^2}[1+O_{1,*}(A_0^2+L_0^{-1})],\\
 \tau&=\eps^2\rho I_2/I_1^2=O_{1,*}(A_0^2).
 \end{aligned}
\]
On the smaller root, $\lambda=(\rho/I_1)[1+\tau+O_{1,*}(\tau^2)]$.
Expansion of the exact energy therefore gives
\[
 \mathcal E_\eps-\mathcal E_0
 =\frac{\eps^2\rho^2I_2}{2I_1^2}+O_{1,*}(A_0^4)
 =\frac{256\pi}{3}\frac{\eps^2A}{b^2}
 +O_{1,*}\!\left(A_0^2(A_0^2+L_0^{-1})\right).
\]
This and \eqref{mfg:bare-C1-quantitative-budget} prove the expansion.
\end{proof}

\begin{proposition}[Corrected common center]\label{mfg:prop:corrected-center}
Write $E_\eps^{\rm bare}$ for the energy of the approximate solution and
$E_\eps=E_\eps^{\rm bare}+G_\eps$ for its corrected value.
With $s,b$ fixed, the latter satisfies
\begin{equation}\label{mfg:corrected-center-energy}
 E_\eps(m)-E_\eps(0)=16\pi\Re((m/s)^2)
 -16\pi\nabla J_s(s)\cdot m+o_{C_t^1}(\delta_0^6).
\end{equation}
The quadratic part has $m$-Hessian
$(32\pi/A)(2e\otimes e-I)$, with eigenvalues $\pm32\pi/A$.
The center equation $\nabla_m E_\eps=0$ has a zero satisfying
\[
 m_*=(A/2)(2e\otimes e-I)\nabla J_s(s)+o(\delta_0^4).
\]
\end{proposition}
\begin{proof}
With $x,v_P,H_P$ as above, integrate
$\partial_i(x_j e^{v_P})$ and $\nabla(|x|^2J_se^{v_P})$ over $\C$.
The resulting identities and the moment asymptotic below give
\begin{gather*}
 \int_\C H_P\frac{x_j}{|x|^2}\partial_iv_P
 =-\delta_{ij}\int_\C\frac{H_P}{|x|^2},\qquad
 \int_\C\frac{H_P}{|x|^2}=\frac{16\pi}{A}(1+o(1)),\\
 \int_\C H_PJ_s\nabla v_P
 =-\int_\C H_P(\nabla J_s+2J_sx/|x|^2).
\end{gather*}
The factor $|x|^2$ in $H_P$ removes singular boundary terms. For
uniform moment limits, including at $c_s=0$, set
$\delta=|s|$, $R=d/\ell$, and
$G_s(x)=\nabla J_s(x)+2J_s(x)x/|x|^2$. On each core disk
$B(s_i,c\delta)$, $s_i=\pm s$, the model mass is $8\pi+O(R^{-2})$
and its first absolute moment about $s_i$ is $O(\ell)$.
There, $|\nabla G_s|\le C\delta$ and
$|\nabla(|x|^{-2})|\le C\delta^{-3}$.
On the complement, $H_P\le Cb^2|x|^2/|x^2-s^2|^4$,
$|G_s(x)|\le C(|x|^2+\delta^2)$, and $H_P$ removes the singularity at zero. Splitting into
$|x|\le C\delta$ and $|x|>C\delta$ therefore gives
$\int_\C H_PG_s=16\pi G_s(s)+O(\delta^2/R)$, $\int_\C H_P/|x|^2=16\pi/A+O((AR)^{-1})$.
Uniformly on the compact $\zeta$ set, $|m|=O(\delta^4)$; substituting
these moments into the energy variation gives
\[
 \nabla_mE_\eps^{\rm bare}
 =-32\pi m/A-16\pi\{\nabla J_s(s)+2j_ss/A\}
       +O(\delta^2/R)+o(\delta^2).
\]
The cutoff and nonlinear errors have combined $L^1$ norm $O(A^3+b^2)$; the saturation error is $O(\delta_0^4)$ in $L^1$.
Their pairing with the center tangent $O(\ell^{-1})$ is $o(\delta_0^2)$, since
\[
 \frac{A^3+b^2+\delta_0^4}{\ell\delta_0^2} =O(\delta_0^3R+\delta_0/R+\delta_0R)=o(1).
\]
Integrating in $t$ and adding the correction
$8\pi(c_s^2-j_s^2)$ from \eqref{mfg:odd-energy-C1-budget} cancels
the linear $j_s$ terms, and 
\[
 -16\pi|m|^2/A+32\pi(m\cdot s)^2/A^2
=16\pi\Re((m/s)^2),
\]
 giving \eqref{mfg:corrected-center-energy}.
Degree stability of the invertible limiting center field gives the zero.
\end{proof}

\begin{lemma}[Dependence on the common center]\label{mfg:lem:bare-center-extension}
On the fixed parameter set, with $m=A_0^2t$ and $\delta=\delta_0$,
\begin{equation}\label{mfg:bare-center-C1-budget}
 \|E_\eps^{\rm bare}(m,\cdot)-E_\eps^{\rm bare}(0,\cdot)\|_{C_\zeta^1}
 \le C\{R(\delta^6+b_0^2)+\delta^5\}=o(\delta^4).
\end{equation}
\end{lemma}
\begin{proof}
Use the energy variation in $x=z-m$, where $m$ is independent of
$\zeta$ and the parameter derivatives of the polynomial profile are even and $O(R)$.
The weight change is exactly
$|x+m|^2/|x|^2=1+2m\cdot x/|x|^2+|m|^2/|x|^2$.
Parity cancels the linear term for a centered cutoff; replacing it by the fixed cutoff
introduces an error bounded by $Cb_0^2R$. The square-map bounds give
\[
 \begin{aligned}
  |m|^2\int\frac{H_P}{|x|^2}&\le C|m|^2/A\le C\delta^6, \\
  \int H_P\frac{|m\cdot x|}{|x|^2}|J_s(x)|
 &\le C|m|\left(A+\frac{b^2}{A}\log R\right)\le C\delta^6.
 \end{aligned}
\]
After removing its constant, the shifted incident term has odd change
$O(|m||x|)$ and even change $O(|m|(|x|^2+A)+|m|^2)$, with even
moment $O(\delta^6+b_0^2)$. Nonlinear and scalar-root terms obey the same bound; saturation gives
a smaller even change. Parameter derivatives of the polynomial profile introduce a factor bounded by $CR$.
Smooth Green tangents are $O(A)$ and the odd residual has $L^1$ norm
$O(\delta^3)$, giving the remaining $O(\delta^5)$.
This bounds the derivatives; integrating $\nabla_mE_\eps^{\rm bare}=O(\delta^2)$
over $|m|\le C\delta^4$ bounds the value.
The scale relations give the stated $o(\delta^4)$ bound.
\end{proof}

\begin{proof}[Proof of Theorem~\ref{mfg:thm:general-quartic-pair}]
Put 
\[
 \begin{gathered}
  q(\theta)=-8\pi Q_{\rm eff}(e_\theta), \qquad
  q_*=q(\theta_0)<0, \\
  c_*=\pi\Delta_gK_g(p)>0, \qquad
  \text{and}\quad \gamma=256\pi/3.
 \end{gathered}
\]
Choose the unique small $b_0$ and $A_0$ by
\[
 \begin{aligned}
  \eps^2&=\frac{\sqrt{(-q_*)c_*}}{\gamma}b_0^3\sqrt{2\log(1/b_0)-1}, \\
  A_0&=b_0\sqrt{\frac{c_*(2\log(1/b_0)-1)}{-4q_*}}.
 \end{aligned}
\]
Proposition~\ref{mfg:prop:bare-global-energy},
Lemma~\ref{mfg:lem:bare-center-extension} and
\eqref{mfg:odd-main-energy-small} show that the projected energy,
after removal of a constant and division by $A_0^2$, converges in
$C^1_\zeta$ to
$\Psi(X,Y,\theta)=q(\theta)X^2-2q_*Y^2-2q_*X/Y^2$,
where $X=A/A_0$, $Y=b/b_0$.
At the critical point $(1,1,\theta_0)$, the $(X,Y)$ Hessian determinant
is $-48q_*^2$ and the angular Hessian is $q''(\theta_0)\ne0$;
thus $\nabla\Psi$ has nonzero local degree.
With $m=A_0^2t$, use the normalized center components
$A_0^{-3}\partial_{t_j}E_\eps$, $j=1,2$.
Proposition~\ref{mfg:prop:corrected-center} gives uniform convergence
to an affine center field with derivative eigenvalues $\pm32\pi/X$.
For main parameters near $(1,1,\theta_0)$, choose a bounded center ball
containing its unique zero. The full limiting field has no boundary
zero and degree equal to the product of the two nonzero degrees.
Uniform convergence gives a critical point for every small $\eps$,
without requiring a differentiable choice of center zero.

To show that the critical point eliminates all five multipliers, set $\delta=\delta_0$,
$R=d/\ell$, $L=\log R$, and $k=A_0^2/\ell\asymp\delta^3R$.
Order the trial functions as the four translations followed by $N_0$, and set
$\Gamma_0=(4\pi/3)\operatorname{diag}(1,1,1,1,2)$ and
$C_0=\operatorname{diag}(2,2,2,2,-2)$.
For the parameters $p=(q_+,q_-,\beta)$, $\beta=\log b$, use
\[
 \mathscr D^p=(\ell\partial_{q_+},\ell\partial_{q_-},\partial_\beta).
\]
Differentiating the polynomial profile gives
\[
 \mathscr D_j^pv_0=(C_0)_{jj}N_j+\text{constant}
+O_{L^2(d\mu)}(R^{-1}):
\]
 on its own core the translation tangent is
$2Z_{i\alpha}$, on the other core it is $O(R^{-1})$, and the width tangent
is $-2N_0+O_{L^2}(R^{-1})$. Cutoff tails and smooth Green terms obey the same bound. Constants
from mean subtraction and scalar-root differentiation pair to zero,
since $\int_MK_\alpha\,dV_g=0$.
The Gram matrix of the projected trial functions is
$\Gamma_s=\Gamma_0+O(L/R)$ by \eqref{mfg:natural-five-projection}.
Writing the ordered frames as column vectors, we have
$\mathbf Y=\Gamma_s^{-1/2}\Pi_5\mathbf N$ and $K_\alpha dV_g=2Y_\alpha d\mu$; hence
\[
 \left(\int_MK_\alpha\mathscr D_j^pv_0\,dV_g\right)_{\alpha j} =2\Gamma_0^{1/2}C_0+O(L/R).
\]

For the correction, differentiate $\int_MK_\alpha\phi\,dV_g=0$.
The frame estimates in Lemma~\ref{mfg:lem:parameter-inverse} give
$\|\mathscr D_tY_\alpha\|_\infty\le C\delta^3L^2$ and
$\|\mathscr D_\zeta Y_\alpha\|_\infty\le CR^2L^2$.
Moreover $\|\nabla_gH\|_1+\|H\|_2\le C/\ell$ and
$\|\nabla_gY_\alpha\|_2+\|Y_\alpha\|_\infty\le C$.
Converting material to ordinary derivatives, with velocities $O(A_0^2)$
for $t$ and $O(d)$ for $\zeta$, gives
\[
 \begin{gathered}
  \|\partial_t(K_\alpha dV_g)\|_{\rm TV}\le C(k+\delta^3L^2), \\
  \|\partial_\zeta(K_\alpha dV_g)\|_{\rm TV}\le C(R+R^2L^2).
 \end{gathered}
\]
The transport term is bounded by
\[
 \int H|\nabla Y_\alpha|\le\|H\|_2\|\nabla Y_\alpha\|_2\le C/\ell;
\]
the differentiated volume density is included.
Using 
\[
 \begin{gathered}
  \|\phi\|_\infty\le C\delta^3 \\
  \text{and}\quad \int K_\alpha D\phi\,dV_g=-\int \phi\,D(K_\alpha dV_g),
 \end{gathered}
\]
 in particular,
\[
 k^{-1}\left|\int_MK_\alpha\partial_{t_j}\phi\,dV_g\right| \le C\delta^3(1+L^2/R).
\]

For the actual parameters $a=(X,Y,\theta,t_1,t_2)$, multiply the five
pairing columns by
\[
 c=(\ell/|s|,Y,\ell/|s|,\ell/A_0^2,\ell/A_0^2).
\]
The resulting Jacobian from these columns to $\mathscr D^p$ is
\[
 \widehat J=
 \begin{pmatrix}
 e/(2X)&0&e^\perp&e_1&e_2\\
 -e/(2X)&0&-e^\perp&e_1&e_2\\
 0&1&0&0&0
 \end{pmatrix},
 \qquad e=s/|s|,
\]
with block row sizes $2,2,1$. It and its inverse are uniformly bounded
on the compact parameter set. Consequently the pairing matrix satisfies
\begin{equation}\label{mfg:natural-constraint-matrix}
 \left(c_j\int_MK_\alpha\partial_{a_j}(v_0+\phi)\,dV_g\right)_{\alpha j}
 =2\Gamma_0^{1/2}C_0\widehat J
 +O\!\left(L/R+\delta^3R^2L^2+\delta^3L^2/R\right).
\end{equation}
The error tends uniformly to zero, including in the center columns.
Uniform invertibility and the energy variation force all five multipliers
to vanish. The mass relation selects the lower root because
$\eps^2\|e^{u_0+v_0+\phi}\|_\infty=O(A_0^2)$.
Thus $u=u_0+2\log\eps+v_0+\phi$ solves the original equation.

The corrections vanish uniformly, each core has mass $8\pi$, and the
density vanishes away from $p$, with $\sup e^u=O(A_0^2)\to0$.
Moreover $\delta_\eps/\ell_\eps\asymp A_0/b_0\to\infty$ and, at
both core centers $q_\pm=m\pm s$,
\[
 \we(q_\pm)+4\log|s|=2\log(A/b)+O(1)\longrightarrow+\infty.
\]
Since $|q_+-q_-|=2|s|$, any center is at distance at least $|s|$
from one core. Thus \eqref{eq:simple} fails, proving non-simplicity.
Elliptic regularity gives a smooth regular part, since $e^{u_0}$ is
smooth at the integer divisor.
\end{proof}

\subsection{An explicit metric without isometric symmetry}
\label{mfg:explicit-construction}\label{mfg:explicit-asymmetric-example}
On $E=\C/(\mathbb Z+i\mathbb Z)$, let $g_0=|dz|^2$ and place
strength-one vortices at the four two-torsion points $E[2]$.
Set $r=1/100$, $s_*=1/400$, $q=1/5+i/7$, and
$\eta(t)=e^{-1/t}$ for $t>0$ and $\eta(t)=0$ otherwise. Define
\[
 \begin{aligned}
  \chi_r(z)&=\frac{\eta(4-|z|^2/r^2)}
 {\eta(4-|z|^2/r^2)+\eta(|z|^2/r^2-1)}, \\
  \psi_q(z)&=\eta(1-|z-q|^2/s_*^2).
 \end{aligned}
\]
Use Euclidean lifts on their disjoint embedded disks and extend by zero.
Thus $\chi_r=1$ on $B_r(0)$ and vanishes outside $B_{2r}(0)$. Set
\begin{equation}\label{mfg:explicit-asymmetric-metric}
 g=e^\sigma g_0,\qquad
 \sigma=\chi_r\{-|z|^4+2\Re(z^4)\}+\psi_q.
\end{equation}
\begin{proof}[Proof of Theorem~\ref{mfg:cor:asymmetric-square-torus}]
The metric is smooth and positive. The flat Green identity
\[
 \sum_{a\in E[2]}G_0(z,a)=G_0(2z,0)
\]
gives
\[
 \mathcal B_0^0=16\pi R_0(z)-4\pi R_0(2z)+2\log2.
\]
Its first and second jets vanish, and its quartic term is $-48\pi R_4$,
where $R_4$ is the quartic term of the flat regular part.
Under conformal change,
\[
 G_g(x,y)-G_0(x,y)=\omega(x)+\omega(y)+C.
\]
The $x$ term cancels from $\mathcal B_0$ because $12\pi-4\pi\cdot3=0$,
and separable terms have zero mixed derivative. Translation invariance
gives
\[
 D_x^2D_y^2R_0(0,0)[e,e;e,e]=24R_4(e).
\]
Hence the Green contributions in \eqref{mfg:quartic-definition} cancel, giving
\[
 Q_{\rm eff}(e)=2\sigma_4(e).
\]

Near zero, $\sigma=-|z|^4+2\Re(z^4)$ and
$K_g=8e^{-\sigma}|z|^2$. Consequently
\begin{equation}\label{mfg:explicit-asymmetric-coefficients}
 \Delta_gK_g(0)=32,\qquad Q_{\rm eff}(e_\theta)=-2+4\cos4\theta.
\end{equation}
The chart is normalized, $K_g(0)=0$, and all the first and second derivatives of $\mathcal B_0$ at $0$ are equal to zero.
At $\theta=0$, $Q_{\rm eff}=2$, $\partial_\theta Q_{\rm eff}=0$ and
$\partial_\theta^2Q_{\rm eff}=-64$, so Theorem~\ref{mfg:thm:general-quartic-pair} applies.

Finally, every isometry of $g$ is a conformal or anticonformal
automorphism of the square torus, hence has the form
$z\mapsto\alpha z+t$ or $z\mapsto\alpha\bar z+t$,
$\alpha\in\{1,-1,i,-i\}$. These maps preserve $g_0$, hence $\sigma$ and its support: the closed
disks of radii $2r$ and $s_*$ centered at $0$ and $q$. The unequal
radii force both centers to be fixed, so $t=0$ modulo the lattice.
No nonidentity square symmetry fixes $q$, whose coordinates are
nonzero, distinct, and smaller than $1/2$. Thus the isometry group is trivial.
\end{proof}
\section*{Acknowledgments}

The authors thank Professor Ky Ho of the University of Economics
Ho Chi Minh City for valuable discussions. 

The authors also acknowledge the use of OpenAI's Codex AI assistant for checking algebraic identities and for organizing and preparing the \LaTeX{} manuscript. The main ideas in the proofs of all theorems are due to the authors, who are responsible for the mathematical arguments and the final content of the manuscript.

\begin{samepage}
\section*{Statements and Declarations}

\noindent\textbf{Funding.}
Y.~Lee was supported by a National Research Foundation of Korea (NRF)
grant funded by the Korea government (MSIT), No.~RS-2026-25489601.
L.~Zhang was supported by a Simons Foundation travel grant,
Award ID:~SFI-MPS-TSM-00013752.

\noindent\textbf{Competing interests.}
The authors declare that they have no competing interests.

\noindent\textbf{Data availability.}
No datasets were generated or analysed in this theoretical study.
Data sharing is therefore not applicable to this article.

\end{samepage}

\end{document}